\documentclass[11pt]{article}
\usepackage{amssymb}
\usepackage{amsfonts}
\usepackage{mathrsfs}
\usepackage{amssymb,amsfonts,amsmath}
\usepackage{graphics}
\usepackage{graphicx}
\usepackage{subfigure}
\usepackage{latexsym}
\usepackage{esint}
\usepackage{epstopdf}
\usepackage{graphicx,float}
\usepackage{times}
\usepackage{graphicx,multicol,enumerate,subfigure,amsmath}
\usepackage{color}
\usepackage{fullpage}
\usepackage{setspace}
\usepackage[english]{babel}
\usepackage[version=3]{mhchem}
\usepackage{amsmath, amssymb, amsthm}
\usepackage{verbatim, latexsym}
\usepackage{colortbl}
\usepackage{multirow}
\usepackage{hyper ref}
\newcommand{\diam}[1]{\mathrm{diam}(#1)}

\newcommand{\ext}{C_{\rm{ext}}(D,s)}

\newcommand{\prnt}[1]{\left( #1 \right)}

\newcommand{\norm}[1]{\left\|#1\right\|}

\newcommand{\normL}[2]{\norm{#1}_{L^2\prnt{#2}}}
\newcommand{\normHp}[3]{\norm{#1}_{H^{#2}\prnt{#3}}}

\newcommand{\normI}[2]{\norm{#1}_{C\prnt{#2}}}
\newcommand{\normLH}[3]{\norm{#1}_{L^2(\Omega,H^{#2}\prnt{#3})}}

\newcommand{\lebesgue}{\mathcal{L}^{(d)}}
\newcommand{\mc}[1]{\mathcal{#1}}

\newcommand{\proj}{\Pi_{\Vh;L^2}}
\newtheorem{theorem}{Theorem}[section]

\newtheorem{assumption}{Assumption}[section]

\newtheorem{lemma}{Lemma}[section]

\newtheorem{corollary}{Corollary}[section]
\newtheorem{proposition}{Proposition}[section]
\numberwithin{equation}{section}
\usepackage{soul}
\setstcolor{red}

\newcommand{\om}{\omega}
\newcommand{\nsamples}{M}
\newcommand{\nsamplesrun}{m}

\newcommand{\kllevel}{L}
\newcommand{\kllevelrun}{\ell}
\newcommand{\Vh}{\mathcal{V}_h}
\newcommand{\dimVh}{Q_h}
\newcommand{\dimVhrun}{k}
\newcommand{\spn}{\text{span}}
\renewcommand{\vec}[1]{\boldsymbol{#1}}

\newcommand{\lnrf}{\kappa}
\newcommand{\lnrfrv}{K}
\newcommand{\distribution}{\mu}
\newcommand{\prob}{\mathbb{P}}
\newcommand{\dom}{\mathcal{D}}
\newcommand{\R}{\mathbb{R}}
\newcommand{\N}{\mathbb{N}}
\newcommand{\taperingint}{\tau}
\newcommand{\taperingweight}{w}
\newcommand{\transpose}{T}

\newcommand{\CholStiffnessL}{\vec{L}^{(h)}}
\newcommand{\stiffnessmatrix}{\vec{S}^{(h)}}
\newcommand{\stiffnessmatrixM}{\vec{S}^{(h;\nsamples)}}
\newcommand{\tildestiffnessmatrixM}{\widetilde{\vec{S}}^{(h;\nsamples)}}

\newcommand{\tildstiffnessmatrix}{\widetilde{\vec{S}}^{(h)}}
\newcommand{\massmatrix}{\vec{M}^{(h)}}

\newcommand{\sigmaalg}{\mathcal{F}}

\newcommand{\constC}{C_{\infty}} 	
\newcommand{\constHs}{C_{H^{s}}}

\newcommand{\expec}{\mathbb{E}}
\newcommand{\normc}[2]{\left\|#1 \right\|_{C(#2)}}

\newcommand{\expect}[1]{\mathbb{E}\left[#1 \right] }
\newcommand{\decayclass}{\mathcal{F}}
\title{On the Numerical Approximation of the Karhunen-Lo\`{e}ve Expansion for Random Fields with Random Discrete Data}
\author{Michael Griebel\thanks{Institut f\"ur Numerische Simulation,
            Universit\"at Bonn,
            Friedrich-Hirzebruch-Allee 7, 53115 Bonn, Germany and Fraunhofer SCAI, Schloss Birlinghoven, 53754 Sankt Augustin, Germany,
            E-mail: griebel@ins.uni-bonn.de
             } \; and
             Guanglian Li\thanks{Department of Mathematics, The University of Hong Kong, Pokfulam Road, Hong Kong. E-mail: lotusli@maths.hku.hk} \; and Christian Rieger\thanks{Philipps-Universit\"at Marburg, Fachbereich Mathematik und Informatik, AG Numerik,
Hans-Meerwein-Stra\ss{}e 6, 35032 Marburg.
            E-mail:riegerc@mathematik.uni-marburg.de
             }
}
\date{}
\begin{document}
\maketitle
\begin{abstract}
Many physical and mathematical models involve random fields in their input data. Examples are ordinary differential equations, partial differential equations and integro--differential equations with uncertainties in the coefficient functions described by random fields. They also play a dominant role in problems  in machine learning. In this article, we  do not assume to have knowledge of the moments or expansion terms of the random fields but we instead have only given discretized samples for them. We thus model some measurement process for this discrete information and then approximate the covariance operator of the original random field. Of course, the true covariance operator is of infinite rank and hence we can not assume to get an accurate approximation from a finite number of spatially discretized observations. On the other hand, smoothness of the true (unknown) covariance function results in effective low rank approximations to the true covariance operator.
We derive explicit error estimates that involve the finite rank approximation error of the covariance operator, the Monte-Carlo-type errors for sampling in the stochastic domain and the numerical discretization error in the physical domain.
This permits to give sufficient conditions on the three discretization parameters to guarantee that an error below a prescribed accuracy $\varepsilon$ is achieved.

\end{abstract}

\section{Introduction}\label{sec:intro}
Mathematical models with random coefficients or random input data  have been widely employed to describe applications
that are affected by a certain amount of uncertainty arising from imperfect or insufficient information about the problem. The range of applications is broad and diverse and includes uncertainty quantification with ordinary differential equations, partial differential equations and integro-differential equations or problems in machine learning and data analysis in, e.g. oil field modeling, quantum mechanics or finance.

For instance in oil field modeling, a common source of uncertainty stems from the unknown soil parameters.
Often, one assumes a statistical model for the soil, for instance a Gaussian process with specified mean and covariance, and fits the possibly
remaining hyper-parameters of the process to the measured data. 
Moreover in machine learning, a common source of randomness stems from noisy observed data. This renders all  subsequent quantities to be random fields. Again one often assumes a statistical model for the noise, for instance a Gaussian model, which perfectly fits to the linear  Bayesian  framework as then the posterior is also Gaussian.
Anyway, let now a probability space $(\Omega,\sigmaalg,\prob)$ be given, where $\sigmaalg$ denotes a $\sigma$-algebra of measurable sets and $\prob$ is a probability measure on this $\sigma$-algebra. On this probability space, we consider a Gaussian random field
\begin{align*}
\lnrf:\Omega \times \dom \to \R.
\end{align*}
Having such a Gaussian random field at hand, one usually parametrizes
it by means of the Karhunen-Lo\`{e}ve (KL) expansion or a polynomial chaos (PC) expansion
\cite{ghanem2003stochastic}, which fits to the $L^2$ setting, i.e.,
\begin{align*}
\lnrf(\vec{\om},\vec{x})=\sum_{\ell=1}^{\infty} \psi_{\ell}(\vec{\om}) \phi_{\ell}(\vec{x}).
\end{align*}
This greatly facilitates the subsequent treatment of problems with $\lnrf$ modeling their random input or their random coefficient functions,
e.g., by the stochastic Galerkin method or the stochastic collocation method. Alternatively, in the Banach space setting, one may expand the random field with respect
to the hierarchical Faber basis or some wavelet type basis; see \cite{Bachmayr.Cohen.Dinh.Schwab:2017}
for details.

To further fix notation, a realization of the associated stochastic process is the function
\begin{align*}
\lnrf(\vec{\om},\cdot): \dom \to \R \quad \vec{\om} \in \Omega \text{ fixed}.
\end{align*} 
We now can interpret the stochastic process $\lnrf$ as a $\R^{\dom}$-valued random variable, i.e.,
\begin{align*}
\lnrfrv: \Omega \to \R^{\dom}, \quad \om \mapsto \lnrf(\vec{\om},\cdot):\dom \to \R,
\end{align*}
where $\R^{\dom}$ denotes the set of all maps $\dom \to \R$.
Hence, using the notion of a push forward measure, we define a distribution of this random variable $\lnrfrv$, and thus the random 
process $\lnrf$, as
\begin{align*}
\distribution=\prob \circ \lnrfrv^{-1}.
\end{align*}
Any Gaussian random field $\lnrf$ is determined completely by its first two moments, i.e.,
\begin{align*}
	\expect{\lnrf}&:\dom \to \R, \quad \vec{x} \mapsto \mathbb{E}[\lnrf]=\expect{\lnrf(\cdot,\vec{x})} \quad \text{and} \\
	\text{Cov}_{\lnrf}&:\dom \times \dom \to \R, \quad (\vec{x},\vec{x}^{\prime})\mapsto  \expect{\left(\lnrf-\expect{\lnrf(\cdot,\vec{x})}\right)\left(\lnrf(\cdot,\vec{x}^{\prime})-\expect{\lnrf(\cdot,\vec{x}^{\prime})}\right)}.
\end{align*}
For the sake of simplicity, we will assume that $\lnrf$ is a centered Gaussian field, i.e., $\expect{\lnrfrv}\equiv 0$.
This yields
\begin{align}\label{carlemannkern}
\text{Cov}_{\lnrf}(\vec{x},\vec{x}^{\prime})
=:R(\vec{x},\vec{x}^{\prime})=\int_{\Omega}\lnrf (\vec{\om},\vec{x}) \lnrf (\vec{\om},\vec{x}^{\prime})\, \mathrm{d}\prob(\vec{\om}) = \sum_{\kllevelrun=1}^{\infty} \lambda_{\kllevelrun} \phi_{\kllevelrun} (\vec{x})\phi_{\kllevelrun}(\vec{x}^{\prime}),
\end{align}
where the series is a Mercer-type expansion. 
We will assume that the eigenvalues are sorted in strictly decreasing order, i.e., 
\begin{align}\label{sortedeigenvals}
\lambda_1 > \lambda_2> \dots \ge 0.
\end{align}
In the case of multiple eigenvalues the projection onto the eigenspaces, i.e., 
\begin{align}\label{truemercer}
R(\vec{x},\vec{x}^{\prime})=\sum_{\kllevelrun \in \N} \lambda_{\ell} \Pi_{\ell}(\vec{x},\vec{x}^{\prime})= \sum_{\kllevelrun \in \N}\lambda_{\ell} \sum_{k=1}^{N(\kllevelrun,k)} \phi_{\kllevelrun,k}(\vec{x}) \phi_{\kllevelrun,k}(\vec{x^{\prime}}),
\end{align}
 is to be employed of course, but for reasons of simplicity we will stick here to the notationally simpler Mercer series \eqref{carlemannkern} instead.
By repeating multiple eigenvalues and modifying our approach accordingly, the more general case of multiple eigenvalues can be covered in an analogous way.

Usually, such an expansion will be truncated yielding a so-called finite noise approximation. To be precise, the approximation 
\begin{align}\label{RL}
R(\vec{x},\vec{x}^{\prime}) \approx R^{\kllevel}(\vec{x},\vec{x}^{\prime}):=\sum_{\kllevelrun=1}^{\kllevel} \lambda_{\kllevelrun}\phi_{\kllevelrun}(\vec{x})  \phi_{\kllevelrun}(\vec{x}^{\prime})
\end{align}
is considered, 
where the quality of the approximation relies on sharp eigenvalue decay estimates as given in Theorem \ref{thm:truncationError} later on. This approach however needs the set $\{ \phi_{\kllevelrun}\}$ of continuous eigenfunctions of the covariance operator $R$ in the first place, which is in general unknown. While its knowledge is often a priorily assumed in the literature for further numerical analysis, such complete information is mostly not available in practice and such an approach is thus often not viable.

Therefore, in the following, we consider a different situation: We  do not assume to have access to the moments of the random field $\lnrf$ in our analysis. Thus we can not compute its Karhunen Lo\`{e}ve expansion directly as an eigenvalue problem for the integral operator with the covariance functions as kernel.
Instead, we will only assume
\begin{align}\label{finitemoment}
	\constC :=\left(\int_{\Omega} \normc{\lnrf\left(\vec{\om} , \cdot\right)}{\dom}^2  \, \mathrm{d}\prob(\vec{\om})\right)^{\frac{1}{2}}<\infty.
\end{align}
Note that \eqref{finitemoment} implies uniform bounds on mean and covariance functions of the random field $\lnrf$.  
Indeed, the definition of the mean function and an application of the Cauchy-Schwarz inequality leads to
\begin{align}\label{eq:cinf1}
\left| \expect{\lnrf(\cdot,\vec{x})}\right|=\left|\int_{\Omega}\lnrf(\cdot,\vec{x}) \, \mathrm{d}\prob\right|
\le \int_{\Omega}\left|\lnrf(\cdot,\vec{x}) \right|\, \mathrm{d}\prob \le \left(\int_{\Omega}\left|\lnrf(\cdot,\vec{x}) \right|^2\, \mathrm{d}\prob \right)^{\frac{1}{2}}\left(\int_{\Omega}1\, \mathrm{d}\prob \right)^{\frac{1}{2}}\le \constC.
\end{align}
Moreover, we can obtain by the definition of the covariance function and an application of the Cauchy-Schwarz inequality the bound
\begin{align}\label{eq:cinf2}	
\text{Cov}(\vec{x},\vec{x}^{\prime})
&=\int_{\Omega}\left(\lnrf(\cdot,\vec{x})-\expect{\lnrf(\cdot,\vec{x})}\right)
\left(\lnrf(\cdot,\vec{x}^{\prime})-\expect{\lnrf(\cdot,\vec{x}^{\prime})}\right)\, \mathrm{d}\prob \nonumber \\
	&\le\left( \int_{\Omega}\left(\lnrf(\cdot,\vec{x})-\expect{\lnrf(\cdot,\vec{x})}\right)^2\, \mathrm{d}\prob \right)^{\frac{1}{2}}\left( \int_{\Omega}\left(\lnrf(\cdot,\vec{x}^{\prime})-\expect{\lnrf(\cdot,\vec{x}^{\prime})}\right)^2\, \mathrm{d}\prob\right)^{\frac{1}{2}} \nonumber \\
	&\le 2\left( \int_{\Omega} \lnrf(\cdot,\vec{x})^2+\expect{\lnrf(\cdot,\vec{x})}^2 \, \mathrm{d}\prob \right)^{\frac{1}{2}}\left( \int_{\Omega}\lnrf(\cdot,\vec{x}^{\prime})^2+\expect{\lnrf(\cdot,\vec{x}^{\prime})}^2\, \mathrm{d}\prob\right)^{\frac{1}{2}}
	\le 4\constC^2.
\end{align}
Furthermore, we define
$\constHs :=\left(\int_{\Omega} \normHp{\lnrf\left(\vec{\om} , \cdot\right)}{s}{\dom}^2  \, \mathrm{d}\prob(\vec{\om})\right)^{\frac{1}{2}}<\infty$ for later use as well.
We refer to \cite[Theorem 5.2]{steinwart2017} for conditions on $\lnrf$ which allow for such types of bounds.

Next we assume some spatial regularity of the random field $\lnrf$.
\begin{assumption}\label{A:11} Let $\lnrf:\Omega \times \dom \to \R$ and thus $\lnrfrv:\Omega \to \R^{\dom}$ be  contained in the Bochner space
\begin{align}\label{eqA11}
L^{\infty}(\Omega, H^{s}(\dom)) \quad \text{for some } s\ge 0,
\end{align}
where, for $s=0$, we have $H^{0}(\dom)=L^{2}(\dom)$.
Alternatively, let $\lnrf:\Omega \times \dom \to \R$ and thus $\lnrfrv:\Omega \to \R^{\dom}$ be contained in the Bochner space
\begin{align}\label{eqA22}
L^{\infty}(\Omega, H^{s}(\dom)) \quad \text{for some } s>\frac{d}{2}.
\end{align}
\end{assumption}
Now we introduce a discretization of the function space $\R^{\dom}:=\{f:\dom \to \R\}$.
To be precise, we consider an underlying finite element space of the form
\begin{align}\label{eq:FEspace}
\Vh:=\{v\in H^{1}(\dom):v|_{K}\in P^{
\lceil s \rceil }(K) \text{ for all } K\in \mc{T}_h\} \subset \R^{\dom},
\end{align}
where $\mathcal{T}_{h}$ is a regular quasi-uniform triangulation over the
physical domain $\dom$ with a maximal mesh size $h$ and $ \lceil s \rceil  \in \N$ will depend on the spatial regularity of $\lnrf$.
We denote the dimension of $\Vh$ by $\dimVh <\infty$. Moreover, we expect $\dimVh = \mathcal{O}(h^{-d}s^d)$.
For the rest of the paper, we will fix some arbitrary basis of $\Vh$
\begin{align*}
	\Vh=\spn\left\{\theta_{\dimVhrun}^{(h)} \ : \  1 \le \dimVhrun \le \dimVh \right\}.
\end{align*}
Here we may employ a nodal basis as well as a discrete orthonormal basis later on.
Let $\vec{\theta}^{(h)}$ be 
\begin{align*}
\vec{\theta}^{(h)}: \dom \to \R^{\dimVh}, \quad \vec{x} \mapsto \left(\theta^{(h)}_{1},\dots,\theta^{(h)}_{\dimVh} \right)^{\transpose}.
\end{align*}
Any function $v^{(h)} \in \Vh$ then admits the expression
\begin{align*}
v^{(h)}= \vec{V}^{(h)}\cdot \vec{\theta}^{(h)}
\quad \text{with} \quad \vec{V}^{(h)} = \left( V^{(h)}_{1},\dots, V^{(h)}_{\dimVh} \right)^{\transpose}\in \R^{\dimVh},
\end{align*}
and each function $v^{(h)} \in \Vh$ can be identified with its coefficient vector $\vec{V}^{(h)}\in \R^{\dimVh}$.

Next we consider the projection map $\Pi_{\Vh}:L^{2}(\dom) \to \Vh$. Thus $s=0$ in \eqref{eqA11} from Assumption \ref{A:11}, which is here the weakest assumption possible.
Given such regularity, we recall the approximation power of the $L^2$-projection $\Pi_{\Vh;L^2}: L^2(\dom)\to \Vh$, which is given in \cite[Theorem 4.4.20]{MR2373954} as
\begin{align}
\normL{v-\Pi_{\Vh}v}{\dom}&\leq C_{\Pi_{\Vh;L^2}} h^{s}\normHp{v}{s}{\dom} \text{ for all } v\in H^s(\dom),\label{eq:approxL2}\\
\normI{v-\Pi_{\Vh}v}{\dom}&\leq C_{\Pi_{\Vh;C}}h^{s-\frac{d}{2}}\normHp{v}{s}{\dom} \text{ for all } v\in H^s(\dom) \text{ for } s>d/2. \label{eq:approxLinfty}
\end{align}
Here, the positive constants $C_{\Pi_{\Vh;L^2}}$ and $C_{\Pi_{\Vh;C}}$ depend only on the shape regularity parameter of $\mathcal{T}_h$ and are independent of the mesh size $h$.

In contrast to other approaches where  the continuous random field $\lnrf:\Omega \times \dom \to \R$  and thus $\lnrfrv:\Omega \to \R^{\dom}$ are assumed to be completely known, we in the following only assume to have access to
\begin{align}\label{lnrfexpansionh}
\lnrf^{(h)}(\vec{\om},\vec{x})=\vec{\lnrfrv}^{(h)}(\vec{\om})\cdot \vec{\theta}^{(h)}(\vec{x}),
\end{align}
where $\vec{\lnrfrv}^{(h)}(\vec{\om})$ is a multivariate random variable defined by
\begin{align}
\label{observedrandomvector}
\vec{\lnrfrv}^{(h)}: \Omega \to \R^{\dimVh}, \quad \vec{\om} \mapsto \vec{\lnrfrv}^{(h)}(\vec{\om})=\left(\lnrfrv_{1}^{(h)}(\vec{\om}),\dots,\lnrfrv_{\dimVh}^{(h)}(\vec{\om}) \right)^{\transpose}.
\end{align}
Furthermore we assume that the coefficients, i.e. the random variables  $\{\lnrfrv^{(h)}_{j}(\vec{\om})\}_{j=1,\cdots,Q_h}$, are independent identically distributed with a common distribution $\distribution^{(h)}$.
As we work with centered Gaussian random fields, we have $\expect{\vec{\lnrfrv}^{(h)}}=\vec{0} \in \R^{\dimVh}$.

This resembles more closely the practical situation in real applications where in general the full random field $\lnrfrv$ is  never completely accessible but only given at a finite number of  sample points which stem from independent measurements in the first place.
The random vector $\vec{\lnrfrv}^{(h)}$ is indeed the information we practically have on the random input. 

Note that an alternative approach  would be to consider directly the \emph{Carleman operator} $R$ (see \eqref{carlemannkern}) for centered fields and  to approximate it by $\bar{R}^{(h)}$, i.e.,
\begin{align*}
R(\vec{x},\vec{x}^{\prime})&\approx R^{(h)}(\vec{x},\vec{x}^{\prime}):=\int_{\Omega} \lnrf^{(h)}(\vec{\om},\vec{x})\lnrf^{(h)}(\vec{\om},\vec{x}^{\prime})\, \mathrm{d}\prob(\vec{\om})\\&\approx \nsamples^{-1}\sum_{\nsamplesrun=1}^{\nsamples} \lnrf^{(h)}(\vec{\om}_{\nsamplesrun},\vec{x})\lnrf^{(h)}(\vec{\om}_{\nsamplesrun},\vec{x}^{\prime})=:\bar{R}^{(h)}(\vec{x},\vec{x}^{\prime}).
\end{align*}
This kernel $\bar{R}^{(h)}$ gives rise to an integral operator $f\mapsto \int_{\dom} f(\vec{x})\bar{R}^{(h)}(\vec{x},\cdot) d \vec{x}$ for which the stochastics is integrated out. Our aim is however to derive a kernel via sampling estimation instead. But this sampling is not involved in the above reference kernels ${R}$ and $\bar{R}^{(h)}$ at all.  Therefore, we consider the kernel $R^{(h)}$ as integral kernel on the discretized space $\Vh$. In particular, we will have a Mercer-type expansion of the form 
\begin{align*}
R^{(h)}(\vec{x},\vec{x}^{\prime})=\sum_{\kllevelrun=1}^{\dimVh} \lambda^{(h)}_{\ell}
\phi_{\kllevelrun}^{(h)}(\vec{x}) \phi_{\kllevelrun}^{(h)}(\vec{x}^{\prime}) \quad \text{with} \quad \phi^{(h)}_{\kllevelrun}(\vec{x})= \vec{\Phi}^{(h)}_{\kllevelrun} \cdot \vec{\theta}^{(h)}(\vec{x}).  
\end{align*}
This gives rise to a generalized matrix eigenvalue problem and we merely estimate the associated stiffness matrix from samples.

Now, due to \eqref{eqA22} from Assumption \ref{A:11}, we obtain the estimate
\begin{align} \label{errorhom}
\normc{\Pi_{\Vh} \lnrf(\vec{\om},\cdot) }{\dom}&\le \normc{\lnrf\left(\vec{\om} , \cdot\right) }{\dom}+	\normc{\lnrf\left(\vec{\om} , \cdot\right)-\Pi_{\Vh} \lnrf(\vec{\om},\cdot) }{\dom}\nonumber\\&\le\normc{\ \lnrf(\vec{\om},\cdot) }{\dom} +  C_{\Pi_{\Vh;C} }h^{s-\frac{d}{2}} \normHp{\lnrf\left(\vec{\om} , \cdot\right)}{s}{\dom} \quad \text{for }\prob\text{ almost all } \vec{\om}\in \Omega.
\end{align}
Hence, after squaring and integrating over $\Omega$, we have  
\begin{align*}
&\int_{\Omega}\normc{\Pi_{\Vh} \lnrf(\vec{\om},\cdot) }{\dom}^2\, \mathrm{d}\prob(\vec{\om}) \le \int_{\Omega}\left( \normc{\ \lnrf(\vec{\om},\cdot) }{\dom} +  C_{\Pi_{\Vh;C} }h^{s-\frac{d}{2}} \normHp{\lnrf\left(\vec{\om} , \cdot\right)}{s}{\dom} \right)^2\, \mathrm{d}\prob(\vec{\om}) \\
&\le 2  \int_{\Omega} \normc{\ \lnrf(\vec{\om},\cdot) }{\dom}^2 +  C_{\Pi_{\Vh;C} }^2 h^{2s-d} \normHp{\lnrf\left(\vec{\om} , \cdot\right)}{s}{\dom}^2\, \mathrm{d}\prob(\vec{\om})\le 2 \constC^2 +  2 C_{\Pi_{\Vh;C} }^2 h^{2s-d} \constHs^2 .
\end{align*}
Thus, with $\lnrf^{(h)}\left(\vec{\om} , \cdot\right):=\Pi_{\Vh} \lnrf(\vec{\om},\cdot)$, we obtain
\begin{align}\label{finitemomenth}
	\constC^{(h)} :=\left(\int_{\Omega} \normc{\lnrf^{(h)}\left(\vec{\om} , \cdot\right)}{\dom}^2  \, \mathrm{d}\prob(\vec{\om})\right)^{\frac{1}{2}}<\infty. 
\end{align}
Analogously by employing the $H^s$ norm instead of the $C(\dom)$ norm, we have
\begin{align*}
\constHs^{(h)} :=\left(\int_{\Omega} \normHp{\lnrf^{(h)}\left(\vec{\om} , \cdot\right)}{s}{\dom}^2  \, \mathrm{d}\prob(\vec{\om})\right)^{\frac{1}{2}} \le \left(\int_{\Omega} \normHp{\lnrf\left(\vec{\om} , \cdot\right)}{s}{\dom}^2  \, \mathrm{d}\prob(\vec{\om})\right)^{\frac{1}{2}} <\infty. 
\end{align*}
As in the continuous case \eqref{eq:cinf1} and \eqref{eq:cinf2} , this implies 
\begin{align}\label{finitemomentsh}
	\expect{\lnrf^{(h)}(\cdot, \vec{x})}\le\constC^{(h)}\quad \text{and}\quad \text{Cov}^{(h)}(\vec{x},\vec{x}^{\prime})\le 4(\constC^{(h)})^2 \quad \text{ for all }\vec{x},\vec{x}^{\prime} \in \dom.
\end{align}
Furthermore, with the representation \eqref{lnrfexpansionh} and since we deal with (centered) random vectors, the covariance is given as
\begin{align}
\label{truecovariance}
\vec{\Sigma}_{\vec{\lnrfrv}^{(h)}}:=\expect{\vec{\lnrfrv}^{(h)} \otimes \vec{\lnrfrv}^{(h)} }-\expect{\vec{\lnrfrv}^{(h)}}\otimes \expect{\vec{\lnrfrv}^{(h)}}=\expect{\vec{\lnrfrv}^{(h)} \otimes \vec{\lnrfrv}^{(h)} } \in \R^{\dimVh \times \dimVh}.
\end{align}
In contrast to the vast majority of the literature, where the covariance is assumed to be known, we only  assume to have given $\nsamples$ independent identically distributed discrete samples  $\vec{\lnrfrv}^{(h)}_{1}, \dots, \vec{\lnrfrv}^{(h)}_{\nsamples} $ of the random variable $\lnrf$. From these finitely many values we merely aim to derive an estimate of the true covariance function. To this end, we have to deal with the high-dimensional covariance matrix  $\vec{\Sigma}_{\vec{\lnrfrv}^{(h)}} \in \R^{\dimVh \times \dimVh}$.
Since $\dimVh \sim h^{-d}$, this involves the well-known curse of dimension.  Consequently the estimation of the covariance matrix gets more challenging the smaller the values of $h$ get. Therefore, we will assume that the covariance matrix $\vec{\Sigma}_{\vec{\lnrfrv}^{(h)}} \in \R^{\dimVh \times \dimVh}$ has a certain off-diagonal decay, that is $$\vec{\Sigma}_{\vec{\lnrfrv}^{(h)}} \in \decayclass_{\alpha}=\decayclass_{\alpha}(C_{\decayclass;1},C_{\decayclass;2}),$$ where 
\begin{align}\label{covarianceclass}
\decayclass_{\alpha}:=\left\{\vec{\Sigma} \in \R^{\dimVh \times \dimVh} \ : \ \max_{\dimVhrun=1,\dots,\dimVh} \sum_{\genfrac{}{}{0pt}{}{\dimVhrun^{\prime}=1}{\left|\dimVhrun^{\prime}-\dimVhrun \right|>c }}^{\dimVh}\left|\vec{\Sigma}_{\dimVhrun,\dimVhrun^{\prime}}\right| \le C_{\decayclass;1} c^{-\alpha}\text{ for all } 1\le c \le \dimVh \text{  and  } \lambda_{\max}(\vec{\Sigma})\le C_{\decayclass;2}\right\}.
\end{align}
Here, $\alpha$ modulates the speed of decay and $ C_{\decayclass;1},C_{\decayclass;2}$ are positive constants.
This specific class of matrices contains several relevant examples of discrete covariance functions, see \cite{bickel2008} and the overview \cite{cai2016}.
Moreover, the choice of this class of matrices will influence the approximation of continuous covariance functions via so-called tapering weights, see \eqref{taperingweights}. We will make use of the optimal rates derived in \cite{bickel2008}.
We furthermore will work with generalized eigenvalue problems. i.e., we do not compute the eigensystem of the covariance matrix $\vec{\Sigma}_{\lnrfrv^{(h)}}$ but instead that of a certain transformed matrix, see \eqref{genalizedeigenvalue3-estimator}. 
The eigensystem of this system gives rise to a computable Mercer-like expansion
\begin{align}\label{mercerh3}
R^{(h;M)}(\vec{x},\vec{x}^{\prime}) :=\sum_{\kllevelrun=1}^{\dimVh}\lambda^{(h;M)}_{\kllevelrun}
\phi_{\kllevelrun}^{(h;M)}(\vec{x})\phi_{\kllevelrun}^{(h;M)}(\vec{x}^{\prime}) \approx R^{(\kllevel;h;M)}(\vec{x},\vec{x}^{\prime}):= \sum_{\kllevelrun=1}^{\kllevel}\lambda^{(h;M)}_{\kllevelrun}
\phi_{\kllevelrun}^{(h;M)}(\vec{x})\phi_{\kllevelrun}^{(h;M)}(\vec{x}^{\prime}),
\end{align}
where the  $(h;M)$--notation indicates that we use an estimator which is based on the $\nsamples$  finite samples $\vec{\lnrfrv}^{(h)}_{1}, \dots, \vec{\lnrfrv}^{(h)}_{\nsamples} $ and makes use of certain appropriate tapering weights. Moreover $\kllevel$ indicates the finite-noise approximation. 

The main aim of this article to derive a bound on the approximation error $\norm{{R}-R^{(\kllevel;h;M)}}$ and of its expectation $ \expect{\norm{{R}-R^{(\kllevel;h;M)}}}$.
To this end, we  couple the discretization parameters $\kllevel, h, \nsamples $ and consider their limits $h\to 0,\kllevel \to \dimVh \to \infty, \nsamples  \to \infty$ to recover the true continuous covariance operator. 
For the approximation step with respect to $\kllevel$, we will make use of the optimal decay estimates for the eigenvalues of the true covariance operator as given \cite{griebel2017decay}. For the two  approximations steps involving  $h$ and $\nsamples$, we will derive corresponding estimates for the spatial discretization error and the sampling error.
Finally we properly combine the three approximation steps. This coupling is our central result, see Theorem \ref{finalthm}.
There we show that the bound
\begin{align*}
	\norm{{R}-R^{(\kllevel;h;\nsamples)}}_{L^{2}(\dom\times\dom)}&\lesssim  \kllevel^{-\frac{2s}{d}-\frac{1}{2}}  +  \left(  \kllevel^{\frac{1}{2}} + G(\kllevel) \right) \left\|\tildstiffnessmatrix-\tildestiffnessmatrixM \right\|_{2\to 2}  \  ,
\end{align*}
holds with high probability, where $G(\kllevel)$ is a function of the truncation parameter $\kllevel$ that depends on spectral properties of the true operator   associated to $R$ as given in \eqref{deltagl}. Moreover $\tildstiffnessmatrix$ denotes the finite element discretization of the true operator and $\tildestiffnessmatrixM$ denotes its sampling approximation, c.f. \eqref{genalizedeigenvalue3} and \eqref{tildestiffnessample}.

Since sampling involves randomness we are also interested in estimates in expectation.
Our second main result is Theorem \ref{thm:final2}. It states that 
\begin{align*}
\expect{\norm{{R}-R^{(\kllevel;h;\nsamples)}}_{L^{2}(\dom\times\dom)}} &\lesssim  \kllevel^{-\frac{2s}{d}-\frac{1}{2}}  + \left( \kllevel^{\frac{1}{2}}+ G(\kllevel) \right) \tilde{\rho}^{\frac{1}{2}}_{h}(\nsamples) \lambda_{\max}\left(\massmatrix\right) \nonumber\\
	&+  L^{\frac{1}{2}} h^{-d} \exp\left(- \nsamples  \rho_1 H(\kllevel)\lambda_{\max}^{-2} \left(\massmatrix\right)  \right),
\end{align*}
where $H(\kllevel)$  is again a function of the truncation parameter $\kllevel$ that depends on spectral properties \eqref{HL} of the true operator,  $\tilde{\rho}_{h}(\nsamples)$ measures the approximation quality of our employed (tapering) estimator, $\massmatrix$ is a classical finite element mass matrix and $\rho_1$ denotes a constant given in \eqref{rho1def} which is determined by the sub-Gaussian property of the involved random variables.
This permits to give sufficient conditions on the three discretization parameters to guarantee that an error below a prescribed accuracy $\varepsilon$ is achieved.

A novel contribution is the combination of our recent sharp bound (see \cite{griebel2017decay}) on spectra of covariance operators with optimal statistical covariance estimation methods, i.e., we derive an error analysis for the reconstruction of the continuous covariance operator from finite measurements. We believe that the presented framework is useful if only finite measured information on uncertain problem parameters is available and no a priori modeling assumption for the random field can be given. Such a situation is indeed often encountered in many practical problems in uncertainty quantification and machine learning.

The remainder of this paper is organized as follows: In Section \ref{sec:prelim} we introduce notation and review some basic facts. In Section \ref{sec:contkl} we recall the KL-expansion and sharp eigenvalue estimates. In Section \ref{sec:apprcov}, we present the reconstruction of the covariance matrix as stiffness matrix with respect to the finite element discretization. We focus on the discrete spatial approximation in Section \ref{sec:spatial} and on the statistical approximation in Section \ref{subsec:approxcovhM}. In Section \ref{sec:tails} we review some tail estimates for Gaussian random variables.
In Section \ref{subsec:approxcovhMfinal} we give bounds on the sampling covariance error. In Section \ref{sec:covoprec} we present our final error estimates for the covariance operator reconstruction and give sufficient conditions on the three discretization parameters to guarantee that an error below a prescribed accuracy $ \sim \varepsilon$ is achieved. We give some concluding remarks in Section \ref{sec:conclusion}.

\section{Notation and basic facts}\label{sec:prelim}
We start with some notation. Let two Banach spaces $V_1$ and $V_2$ be given. Then, $\mathcal{B}(V_1,V_2)$ stands for the Banach space composed of all continuous linear operators from $V_1$ to $V_2$ and $\mc{B}(V_1)$ stands for $\mc{B}(V_1, V_1)$. The set of non-negative integers is denoted by  $\mathbb{N}$. For any index $\alpha\in \mathbb{N}^d$, $|\alpha|$ is the
sum of its components. The letters $\kllevel$, $\nsamples$ and $h$ are reserved for the truncation number of the KL modes, the number of sampling points and the mesh size. We write $A\lesssim B$ if $A\leq cB$ for some absolute constant $c$ which is independent of $\kllevel$, $\nsamples$ and $h$, and we likewise write $A\gtrsim B$. Moreover, for any $s\in \mathbb{N}$, $1\leq p\leq \infty$, we follow \cite{Adam78} and define the Sobolev space $W^{s,p}(\dom)$ by	
$$W^{s,p}(\dom)=\{u\in L^{p}(\dom): D^{\alpha}u\in L^{p}(\dom) \text{ for } 0\leq|\alpha|\leq s\}.$$
It is equipped with the norm
\begin{equation*}
 \|u\|_{W^{s,p}(\dom)} = \left\{\begin{aligned}
 \Big(\sum\limits_{0\leq |\alpha|\leq s}\norm{D^{\alpha}u}_{L^p(\dom)}^{p}\Big)^{\frac{1}{p}}, & \text{ if }1\leq p<\infty,\\
 \max\limits_{0\leq |\alpha|\leq s}\norm{D^{\alpha}u}_{L^\infty(\dom)} ,& \text{ if } p=\infty.
\end{aligned}\right.
\end{equation*}
The space $W_{0}^{s,p}(D)$ is the closure of $C^{\infty}_{0}(\dom)$ in $W^{s,p}(D)$. Its dual space is $W^{-s,q}(\dom)$, with ${1}/{p}+{1}/{q}=1$. Also we use $H^{s}(\dom)=W^{s,p}(\dom)$ for $p=2$. Finally $(\cdot,\cdot)_{\dom}$ denotes the inner product in $L^2(\dom)$.

We now recall the classical Monte--Carlo algorithm.
To this end, we consider a probability space $(\Omega,\mathcal{F},\prob)$ and a given probability distribution $\distribution$. Let $X_{1},\cdots,X_{\nsamples}$ for $\nsamples \in \N$ be independent identically distributed real-valued random variables on $(\Omega,\mathcal{F},\prob)$ with the same associated probability distribution $\distribution$, i.e,
\begin{align*}
 \prob(X^{-1}_{\nsamplesrun}(B))=\distribution(B) \quad \text{for all } B\in \mathcal{B}(\mathbb{R}) \text{ and all } 1\leq \nsamplesrun \leq \nsamples.
\end{align*}
Here, $\mathcal{B}(\mathbb{R})$ denotes the Borel $\sigma$-algebra.
Each random variable $X_{\nsamplesrun}$ induces a sub $\sigma$-algebra, defined by
\begin{align*}
\sigma(X_{\nsamplesrun})=\left\{X^{-1}_{\nsamplesrun}(B) \ : \  B \in \mathcal{B}(\mathbb{R})\right\}.
\end{align*}
Note that the $\{\sigma(X_{\nsamplesrun})\}_{\nsamplesrun=1}^{\nsamples}$ are $\prob$-independent because the random variables $\{X_{\nsamplesrun}\}_{\nsamplesrun=1}^{\nsamples}$ are pairwise independent, i.e., for every finite set $\{\nsamplesrun_1,\dots \nsamplesrun_k\} \subset \N$ and every $B_{\nsamplesrun_{j}}\in \sigma(X_{\nsamplesrun_{j}}) $ for $1\le j \le k$, we have
\begin{align*}
\prob\left(\bigcap_{j=1}^{k}B_{\nsamplesrun_{j}} \right)=\prod_{j=1}^{k}\prob\left(B_{\nsamplesrun_{j}} \right).
\end{align*}
Let $f:\R \to \R$ be a Borel function. Then $Y_{\nsamplesrun}=f(X_{\nsamplesrun})$ is also a random variable, with its associated probability distribution being
\begin{align*}
 \prob((f\circ X_{\nsamplesrun})^{-1}(B))=\prob( X^{-1}_{\nsamplesrun}(f^{-1}(B)))=\distribution(f^{-1}(B)) \quad \text{for all } B\in \mathcal{B}(\mathbb{R}) \text{ and all } \nsamplesrun \in \N.
\end{align*}
Therefore, the composite random variables $\{Y_{\nsamplesrun}\}$ are also identically distributed, with their generated $\sigma$-algebras defined by
\begin{align*}
\sigma(Y_{\nsamplesrun})=\left\{(f\circ X_{\nsamplesrun})^{-1}(B) \ : \  B \in \mathcal{B}(\mathbb{R}) \right\}=\left\{X^{-1}_{\nsamplesrun}(f^{-1}(B)) \ : \  B \in \mathcal{B}(\mathbb{R})  \right\}\subset \sigma(X_{\nsamplesrun}).
\end{align*}
Consequently, the random variables $\{Y_\nsamplesrun\}$ are pairwise independent as a result of the independence of the $\{X_\nsamplesrun\}$.

Next we introduce a real-valued random variable $X$ equipped with the probability $\mu$. Let $Y$ be a random variable defined by $Y=f(X)$. Assume that the first two moments of $Y$ exist. Then its mean and variance are given as
\begin{align*}
\mathbb{E}[Y]&=\int_{\Omega} Y(\om) \, \mathrm{d}\prob(\om)=\int_{\Omega} (f\circ X)(\om) \, \mathrm{d}\prob(\om)=\int_{\R} f(x)  \, \mathrm{d}\distribution(x)<\infty,\\
\mathbb{V}[Y]&=\mathbb{E}[(Y-\mathbb{E}[Y])^2]=\int_{\R} \left(f(x)- \mathbb{E}[Y]\right)^2 \, \mathrm{d}\distribution(x)<\infty.
\end{align*}
By definition, $Y$ and $\{Y_\nsamplesrun\}$ have the same mean
\begin{align*}
\mathbb{E}[Y_\nsamplesrun]=\int_{\R} f(x)  \, \mathrm{d}\distribution(x)=\mathbb{E}[Y]
\end{align*}
and hence, by linearity, we get
\begin{align*}
\mathbb{E}\left[\frac{1}{\nsamples}\sum_{\nsamplesrun=1}^{\nsamples}Y_\nsamplesrun\right]=\mathbb{E}[Y].
\end{align*}
Furthermore, we obtain
\begin{align}
\mathbb{V}\left[\frac{1}{\nsamples}\sum_{\nsamplesrun=1}^{\nsamples}Y_\nsamplesrun\right]
&=\mathbb{E}\left[\left(\frac{1}{\nsamples}\sum_{\nsamplesrun=1}^{\nsamples}Y_\nsamplesrun-\mathbb{E}[Y]\right)^2\right]
=\text{Cov}\left(\frac{1}{\nsamples}\sum_{\nsamplesrun=1}^{\nsamples}Y_\nsamplesrun,\frac{1}{\nsamples}\sum_{\nsamplesrun=1}^{\nsamples}Y_\nsamplesrun\right)\nonumber\\
&=\frac{1}{\nsamples^2}\sum_{\nsamplesrun,\tilde{\nsamplesrun}=1}^{\nsamples}\text{Cov}\left(Y_\nsamplesrun,Y_{\tilde{\nsamplesrun}}\right)=\frac{1}{\nsamples^2}\sum_{\nsamplesrun=1}^{N}\text{Cov}\left(Y_\nsamplesrun,Y_\nsamplesrun\right)=\frac{1}{\nsamples}\mathbb{V}[Y],
\label{eq:xxx1}
\end{align}
where we used the definition of covariance and the fact that mixed terms vanish due to independence.
The idea of Monte--Carlo is to approximate the mean $\mathbb{E}[Y]$ by the estimator $\hat{\mathbb{E}}_{\nsamples}[Y]$ defined by
\begin{align*}
\hat{\mathbb{E}}_{\nsamples}[Y]:=\frac{1}{\nsamples} \sum_{\nsamplesrun=1}^{\nsamples} Y_{\nsamplesrun}=\frac{1}{\nsamples} \sum_{\nsamplesrun=1}^{\nsamples} f(X_{\nsamplesrun}).
\end{align*}
Then an application of \eqref{eq:xxx1} yields the corresponding RMSE error
\begin{align*}
\text{RMSE}&:=\sqrt{\mathbb{E}\left[\left(\hat{\mathbb{E}}_{\nsamples}[Y]-\mathbb{E}[Y]\right)^2\right]}
=\frac{1}{\sqrt{\nsamples}}\mathbb{V}\left[Y\right]^{\frac{1}{2}}.
\end{align*}
For any $\epsilon>0$, a combination of Chebyshev's inequality and the equality \eqref{eq:xxx1} shows
\begin{align*}
\mathbb{P}\left(\left\{\left|\hat{\mathbb{E}}_{\nsamples}[Y]-\mathbb{E}[Y] \right| >\epsilon\right\}\right) \le \epsilon^{-2}\mathbb{V}\left[\hat{\mathbb{E}}_{\nsamples}[Y]\right] =\frac{1}{\nsamples \epsilon^2}\mathbb{V}(Y).
\end{align*}
Taking $\epsilon:=\frac{1}{\sqrt{\nsamples \delta}}\left(\mathbb{V}(Y)\right)^{\frac{1}{2}}$ for some $\delta>0$, we obtain
\begin{align*}
\mathbb{P}\left(\left\{ \left|\hat{\mathbb{E}}_{\nsamples}[Y]-\mathbb{E}[Y] \right| > \frac{1}{\sqrt{\nsamples \delta}}\left(\mathbb{V}(Y)\right)^{\frac{1}{2}}\right\}\right) \le \delta.
\end{align*}

\section{Karhunen-Lo\`{e}ve expansion: Continuous level}\label{sec:contkl}
This section is concerned with the Karhunen-Lo\`{e}ve expansion of the centered Gaussian random field $\lnrf$. Let $\lebesgue(\vec{x})$ be the Lebesgue measure on the physical domain $\dom$. For the sake of simplicity, $L^{2}(\dom)$ and $L^2(\Omega)$ are short for $L^{2}(\dom;\mathrm{d}\lebesgue(\vec{x}))$ and $L^{2}(\Omega;\mathrm{d}\prob)$. We denote the associated integral operator $\mathcal{S}: L^{2}(\dom)\rightarrow  L^{2}(\Omega)$ by
\begin{align}\label{eq:S}
\mathcal{S}v=\int_{\dom}\lnrf (\vec{\om},\cdot)v\, \mathrm{d}\lebesgue,
\end{align}
whereas its adjoint operator $\mathcal{S}^{*}: L^{2}(\Omega)\rightarrow  L^{2}(\dom)$ is defined by
\begin{align}\label{eq:ajoint_S}
\mathcal{S}^{*}v=\int_{\Omega}\lnrf(\vec{\om},\cdot)v\, \mathrm{d}\prob(\vec{\om}).
\end{align}
Let $\mathcal{R}: L^2(\dom)\rightarrow L^2(\dom)$ be defined by  $\mathcal{R}:=\mathcal{S}^{*}\mathcal{S}$.
Then $\mathcal{R}$ is a non-negative self-adjoint Hilbert-Schmidt operator with kernel
$R\in L^2(\dom \times \dom)$ given by
\begin{align}\label{integralkernel}
R(\vec{x},\vec{x}^{\prime})=\int_{\Omega}\lnrf (\vec{\om},\vec{x}) \lnrf (\vec{\om},\vec{x}^{\prime})\, \mathrm{d}\prob(\vec{\om}) = \expec\left[\lnrf (\cdot,\vec{x}) \lnrf (\cdot,\vec{x}^{\prime}) \right]
\end{align}
Moreover, for any $v\in L^2(\dom)$, we have
\begin{align}\label{integralop}
\mathcal{R}v(\vec{x})=\int_{\dom}R(\vec{x},\vec{x}^{\prime})v(\vec{x}^{\prime}) \, \mathrm{d}\lebesgue(\vec{x}^{\prime}) = \int_{\dom}\int_\Omega \lnrf(\vec{\om},\vec{x})\lnrf(\vec{\om},\vec{x}^\prime) \, \mathrm{d}\prob(\vec{\om})\,v(\vec{x}^\prime) \mathrm{d}\lebesgue(\vec{x}^{\prime}).
\end{align}
The standard spectral theory for compact operators \cite{yosida78} implies that the operator $\mathcal{R}$ has
at most countably many discrete eigenvalues with zero being the only accumulation point and each non-zero
eigenvalue has only finite multiplicity. Let $\{\lambda_{\kllevelrun}\}_{\kllevelrun=1}^{\infty}$ be the sequence of eigenvalues (with multiplicity counted) associated to
$\mathcal{R}$, which are ordered non-increasingly, and let $\{\phi_{\kllevelrun}\}_{\kllevelrun=1}^\infty$ be the corresponding eigenfunctions that are orthonormal in $L^2(\dom)$.
Furthermore, for any $\lambda_{\kllevelrun}\neq 0$, define
\begin{equation}\label{eq:psi}
  \psi_{\kllevelrun}(\vec{\om})=\frac{1}{\sqrt{\lambda_{\kllevelrun}}}\int_{\dom}\lnrf(\vec{\om},\vec{x})
  \phi_{\kllevelrun} \, \mathrm{d}\lebesgue(\vec{x}).
\end{equation}
One can verify that the sequence $\{\psi_{\kllevelrun}\}_{{\kllevelrun}=1}^\infty$ is uncorrelated and
orthonormal in $L^2(\Omega)$ and therefore, $\{\psi_{\kllevelrun}\}_{{\kllevelrun}=1}^{\infty}$ are i.i.d normal random functions.
Note that the sequence $\{\lambda_{\kllevelrun}\}_{{\kllevelrun}=1}^{\infty}$ can be characterized by the so-called approximation numbers (cf. \cite[Section 2.3.1]{Pietsch:1986:ES:21700}). They are defined by
\begin{align}\label{eq:a_n}
\lambda_{\kllevelrun}=\inf\{\norm{\mathcal{R}-\mathcal{T}}_{\mathcal{B}(L^2(\dom))}:\mathcal{T}\in \mathfrak{F}(L^2(\dom)), {\text{rank}}(\mathcal{T})< \kllevelrun\},
\end{align}
where $\mathfrak{F}(L^2(\dom))$ denotes the set of the finite rank operators on $L^2(\dom)$.
This equivalency is frequently employed to estimate eigenvalues by constructing finite rank approximation operators.

The KL expansion of the bivariate function $\lnrf$ then refers to the expression
\begin{equation}\label{eq:KL}
\lnrf(\vec{\om},\vec{x})=\sum\limits_{\kllevelrun=1}^{\infty}\sqrt{\lambda_{\kllevelrun}}\phi_{\kllevelrun}(\vec{x})\psi_{\kllevelrun}(\vec{\om}),
\end{equation}
where the series converges in $L^{2}(\Omega) \otimes L^{2}(\dom )\cong L^2( \Omega\times D)$.

\subsection{$\kllevel$-term truncation in case of continuous Karhunen-Lo\`{e}ve expansion}
Now we will truncate the KL expansion and discuss the resulting error. The studies on the $\kllevel$-term KL approximation to random fields are extensive. For example, in \cite{Schwab:2006:KAR:1167051.1167057}, decay rates for the eigenvalues of covariance kernels possessing certain regularity were considered and  generalized fast multipole methods to solve the associated eigenvalue problems were studied. The robust computation of  eigenvalues for smooth covariance kernels was treated in \cite{Todor2006}. A comparison of $\kllevel$-term KL truncation and the sparse grids approximation was given in \cite{Griebel.Harbrecht:2017}.

The result of this section is based on our recent paper \cite{griebel2017decay}, which proves a sharp eigenvalue decay rate under a mild assumption on the regularity of the bivariate function $\lnrf $ in the physical domain. To this end, we recall Assumption \ref{A:11} concerning the
regularity of $\lnrf$. Its weaker variant \eqref{eqA11} states that $\lnrfrv\in L^{\infty}(\Omega, H^{s}(\dom))$ for some $s\ge 0$.
Its stronger variant \eqref{eqA22} implies that the kernel belongs to $H^{s}(\dom)\times H^{s}(\dom)$.
The following eigenvalue decay estimate \cite[Theorems 3.2, 3.3 and 3.4]{griebel2017decay} will be used repeatedly.
\begin{theorem}
\label{thm:truncationError}
Let \eqref{eqA22} from Assumption \ref{A:11} hold. Then, there holds
\begin{align}
\label{thm:truncationErrorLambda}
{{\lambda_{\kllevelrun}}}&\le C_{\ref{thm:truncationError}}\kllevelrun^{-\frac{2s}{d}-1} \quad \text{ for all $\kllevelrun\ge 1$  and }\\
\Big\|{\sum\limits_{\kllevelrun>\kllevel}\sqrt{\lambda_{\kllevelrun}}
\phi_{\kllevelrun}\psi_{\kllevelrun}}\Big\|_{L^2(\Omega\times \dom)}&\leq C_{\ref{thm:truncationError}}^{1/2}\sqrt{\frac{d}{2s}}(\kllevel+1)^{-\frac{s}{d}} \quad \text{ when $\kllevel$ is sufficiently large.}
\label{thm:truncationErrorTrunc}
\end{align}
\end{theorem}
\noindent 
Here, we use the constant
$C_{\ref{thm:truncationError}}:=\diam{\dom}^{2s}C_{\rm em}(d,s)\ext\normLH{\lnrf}{s}{\dom} ^2$, where $C_{\rm{em}}(d,s)$ denotes an embedding constant between certain Lorentz sequence spaces, and $\ext$ is a constant which depends only on $\dom$ and $s$.

The next lemma provides a regularity result for the eigenfunctions $\{\phi_{\kllevelrun}\}_{\kllevelrun=1}^{\infty}$ which follows with operator interpolation from  \cite[Lemma 3.1 \& Remark after Assumption 3.1]{griebel2017decay}
\begin{lemma}[Regularity of the eigenfunctions $\{\phi_{\kllevelrun}\}_{\kllevelrun=1}^{\infty}$]
Let \eqref{eqA22} from Assumption \ref{A:11} be valid. Then for all $0\leq \beta \leq 1$, there holds
\begin{align}\label{eq:phi_theta}
\normHp{\phi_{\kllevelrun}}{\beta s}{\dom}\leq C(\dom,d, s)\kllevelrun^{\frac{\beta s}{d}} \quad \text{ when $\kllevelrun$ is sufficiently large}.
\end{align}
\end{lemma}
\begin{proof}
	Assumption \eqref{eqA22} states that $\lnrf \in L^{\infty}(\Omega, H^{s}(\dom)) \text{ for some } s\ge 0$. As $\int_{\Omega}1 \, \mathrm{d}\prob(\vec{\om})=1$, we deduce $\lnrf \in L^{2}(\Omega, H^{s}(\dom))$ and 
	\begin{align*}
		\left\| \lnrf \right\|^{2}_{ L^{2}(\Omega, H^{s}(\dom))}=\sum_{\kllevelrun=1}^{\infty} \lambda_{\kllevelrun} \left\| \phi_{\kllevelrun}\right\|^{2}_{H^{s}(\dom)} 
		\le 	 C_{\ref{thm:truncationError}} \sum_{\kllevelrun=1}^{\infty} \kllevelrun^{-1-\epsilon} \kllevelrun^{-\frac{2s}{d}+\epsilon} \left\| \phi_{\kllevelrun}\right\|^{2}_{H^{s}(\dom)}
		\end{align*}
		for all $\epsilon>0$. Considering the limit $\epsilon \to 0 ^{+}$, we infer that $ \left\| \phi_{\kllevelrun}\right\|_{H^{s}(\dom)}\lesssim \kllevelrun^{-\frac{s}{d}} $ for $\kllevelrun$ large. Since $ \left\| \phi_{\kllevelrun}\right\|_{H^{0}(\dom)}= \left\| \phi_{\kllevelrun}\right\|_{L^{2}(\dom)}=1$, the statement follows via an interpolation argument.
\end{proof}
\noindent  Here, $C(\dom,d, s)$ denotes a positive constant depending only on $\dom$, $d$ and $s$.
We also get a bound in the uniform norm by using Sobolev's embedding theorem for $s=\frac{d}{2}+\epsilon\ge \frac{d}{2}$, i.e., we have
\begin{align}\label{uniformphiest}
\normI{\phi_{\kllevelrun}}{\dom} \le \normHp{\phi_{\kllevelrun}}{s}{\dom} \le C(\dom,d, s)\kllevelrun^{\frac{s}{d}} =C(\dom,d)\kllevelrun^{\frac{1}{2}+\epsilon} \quad \text{ when $\kllevelrun$ is sufficiently large},
\end{align}
where we use $\epsilon>0$ for an arbitrary small number which may change from line to line.
The eigenfunctions $\{\phi_{\kllevelrun}\}_{\kllevelrun=1}^\infty$ are optimal in the sense that the mean-square error resulting from a finite-rank approximation of $\lnrf$ is minimized \cite{ghanem2003stochastic}. Thus, the eigenfunctions indeed minimize the truncation error in the $L^2$-sense, i.e.
\begin{align}\label{eq:opt_eigen}
\min\limits_{\substack{\{c_{\kllevelrun}(\vec{x})\}_{\kllevelrun=1}^{\kllevel}\subset L^2(\dom)\\\{c_{\kllevelrun}(\vec{x})\}_{\kllevelrun=1}^{\kllevel} \text{ orthonormal}}}\normL{\lnrf(\vec{\om},\vec{x})-\sum\limits_{\kllevelrun=1}^{\kllevel}\left(\int_{\dom}\lnrf(\vec{\om},\vec{x}^{\prime})c_{\kllevelrun}(\vec{x}^{\prime})\,\mathrm{d}\lebesgue(\vec{x}^{\prime})\right) c_{\kllevelrun}(\vec{x})}{\Omega\times \dom}=\sqrt{\sum\limits_{\kllevelrun>\kllevel}\lambda_{\kllevelrun}}.
\end{align}
\section{Approximation of the covariance matrix}\label{sec:apprcov}
\subsection{Spatial approximation}\label{sec:spatial}
In this section, we investigate the influence of the spatial discretization.
Recall that the random field $\lnrf^{(h)}$ defined in \eqref{lnrfexpansionh} is an approximation to the random field $\lnrf$. We assume for the moment to have access to its covariance function
\begin{align}\label{integralkernelh}
R^{(h)}(\vec{x},\vec{x}^{\prime}):=\int_{\Omega}\lnrf^{(h)} (\vec{\om},\vec{x}) \lnrf^{(h)} (\vec{\om},\vec{x}^{\prime})\, \mathrm{d}\prob(\vec{\om})=\expec \left[\lnrf^{(h)} (\cdot,\vec{x}) \lnrf^{(h)} (\cdot,\vec{x}^{\prime}) \right]
\end{align}
and the associated integral operator
\begin{align*}
\mc{R}^{(h)}f(\vec{x}):=\int_{\dom} f(\vec{x}^{\prime})R^{(h)}(\vec{x},\vec{x}^{\prime}) \, \mathrm{d}\vec{x}^{\prime}.
\end{align*}
Then we have the following result:
\begin{lemma}[Semi-discrete spatial approximation error estimate]
Let $R\in L^2(\dom \times \dom)$ be defined as in \eqref{integralkernel} and let its numerical approximation $R_h$ be defined as in \eqref{integralkernelh}. Then it holds
\begin{align}\label{integralkernelerror}
\normc{R^{(h)} -R}{\dom \times \dom}
\le C_{\Pi_{\Vh;L^2}}\constC \constHs h^{s-\frac{d}{2}}.
\end{align}
Furthermore, for the associated integral operators  $\mathcal{R}$ and $\mathcal{R}_h$, it holds 
\begin{align}\label{integraloperatorerror}
\norm{\mc{R}-\mc{R}^{(h)}}_{\mc{B}(L^2(\dom))}\le C_{\Vh;L^2}(\constC\constHs+\constC^{(h)}\constHs^{(h)}) h^{s}.
\end{align}
\end{lemma}
\begin{proof}
Using $ab-a_h b_h=(a-a_h)b+a_h(b-b_h)$, we obtain
\begin{align*}
\normc{R^{(h)} -R}{\dom \times \dom}&\le \sup_{\vec{x},\vec{x}^{\prime} \in \dom} \left|\int_{\Omega}\left(\lnrf^{(h)} (\vec{\om},\vec{x}) -\lnrf (\vec{\om},\vec{x})\right)\lnrf^{(h)} (\vec{\om},\vec{x}^{\prime})\, \mathrm{d}\prob(\vec{\om}) \right|\\
&+ \sup_{\vec{x},\vec{x}^{\prime} \in \dom}\left|\int_{\Omega}\left(\lnrf^{(h)} (\vec{\om},\vec{x}^{\prime}) -\lnrf (\vec{\om},\vec{x}^{\prime})\right)  \lnrf (\vec{\om},\vec{x})\, \mathrm{d}\prob(\vec{\om}) \right|\\
& \le  \int_{\Omega}\normc{\lnrf^{(h)} (\vec{\om},\cdot) -\lnrf (\vec{\om},\cdot)}{\dom} \normc{\lnrf^{(h)} (\vec{\om},\cdot)}{\dom}\, \mathrm{d}\prob(\vec{\om}) \\
&+ \int_{\Omega}\normc{\lnrf^{(h)} (\vec{\om},\cdot) -\lnrf (\vec{\om},\cdot)}{\dom}  \normc{\lnrf (\vec{\om},\cdot)}{\dom}\, \mathrm{d}\prob(\vec{\om}) \\
& \le C_{\Pi_{\Vh;L^2}}h^{s-\frac{d}{2}} \int_{\Omega}\normHp{\lnrf (\vec{\om},\cdot)}{s}{\dom} \normc{\lnrf^{(h)} (\vec{\om},\cdot)}{\dom}\, \mathrm{d}\prob(\vec{\om}) \\
&+ C_{\Pi_{\Vh;L^2}}h^{s-\frac{d}{2}}\int_{\Omega}\normHp{\lnrf (\vec{\om},\cdot)}{s}{\dom}  \normc{\lnrf (\vec{\om},\cdot)}{\dom}\, \mathrm{d}\prob(\vec{\om}),
\end{align*}
where we used \eqref{errorhom} in the last inequality. This, together with the assumption \eqref{finitemoment}, leads to
\begin{align*}
\normc{R^{(h)} -R}{\dom \times \dom}&\le C_{\Pi_{\Vh;L^2}}h^{s-\frac{d}{2}} \left(\int_{\Omega}\normHp{\lnrf (\vec{\om},\cdot)}{s}{\dom}^2 \, \mathrm{d}\prob(\vec{\om}) \right)^{\frac{1}{2}} \left(\int_{\Omega}\normc{\lnrf (\vec{\om},\cdot)}{\dom}^2 \, \mathrm{d}\prob(\vec{\om}) \right)^{\frac{1}{2}} \nonumber \\
&\le C_{\Pi_{\Vh;L^2}} \constHs \constC h^{s-\frac{d}{2}} .
\end{align*}
This proves the first assertion \eqref{integralkernelerror}.
Now note  that
\begin{align*}
\norm{\mc{R}-\mc{R}^{(h)}}_{\mc{B}(L^2(\dom))}:=\sup_{\genfrac{}{}{0pt}{}{v \in L^2(\dom)}{\normL{v}{\dom}=1}}\normL{\mathcal{R}v-\mathcal{R}^{(h)}v}{\dom}
\le \normL{R-R^{(h)}}{\dom\times\dom}.
\end{align*}
Using \eqref{integralkernel} and \eqref{integralkernelh}, we observe 
\begin{align*}
&\normL{R-R^{(h)}}{\dom\times\dom}^2=\int_{\dom} \int_{\dom} \left( R(\vec{x},\vec{x}^{\prime})-R^{(h)}( \vec{x},\vec{x}^{\prime}) \right)^2 \, \mathrm{d}( \lebesgue \otimes \lebesgue)(\vec{x},\vec{x}^{\prime})\\
&=\int_{\dom} \int_{\dom} \left(\int_{\Omega} \kappa(\vec{\om},\vec{x}) \kappa(\vec{\om},\vec{x}^{\prime})-\kappa^{(h)}(\vec{\om},\vec{x}) \kappa^{(h)}(\vec{\om},\vec{x}^{\prime})  \, \mathrm{d}\prob(\vec{\om})  \right)^2 \, \mathrm{d}( \lebesgue \otimes \lebesgue)(\vec{x},\vec{x}^{\prime})\\
&=\int_{\dom} \int_{\dom}\left(  \expec\left[\kappa(\cdot,\vec{x}) \kappa(\cdot,\vec{x}^{\prime})
-\kappa^{(h)}(\cdot,\vec{x}) \kappa^{(h)}(\cdot,\vec{x}^{\prime}) \right]
  \right)^2 \, \mathrm{d}( \lebesgue \otimes \lebesgue)(\vec{x},\vec{x}^{\prime}).
\end{align*}
Thus, we obtain
\begin{align*}
\normL{R-R^{(h)}}{\dom\times\dom}=\normL{\expec\left[\kappa \kappa
-\kappa^{(h)} \kappa^{(h)} \right] }{\dom\times\dom}.
\end{align*}
We get by an application of the Jensen's inequality, i.e., $f(\expec{X}) \le \expec{f(X)}$ with the convex function $f(X)=\normL{X}{\dom\times\dom}$, the bound 
\begin{align*}
\normL{{R}-{R}^{(h)}}{\dom\times\dom}
&\leq \expec\left[\normL{{\kappa \kappa
-\kappa^{(h)}\kappa^{(h)}}}
{\dom\times\dom}\right]\\
&\leq \expec\left[\normL{(\kappa-
\kappa^{(h)})\kappa}
{\dom\times\dom}+\normL{\kappa^{(h)} (\kappa-
\kappa^{(h)})}
{\dom\times\dom}\right]\\
&\le C_{\Vh;L^2}(\constC \constHs+\constC^{(h)} \constHs^{(h)}) h^{s}
\end{align*}
where we used \eqref{eq:approxL2} in the last step.
This proves the second assertion \eqref{integraloperatorerror}.
\end{proof}
We are interested in solving the eigenvalue problem for the operator $\mc{R}^{(h)}$, i.e., we consider finding $(\phi_{\kllevelrun}^{(h)},\lambda^{(h)}_{\kllevelrun})\in \Vh \times \R$ such that
\begin{align}\label{genalizedeigenvalue}
\left(\mc{R}^{(h)}\phi_{\kllevelrun}^{(h)},v^{(h)}\right)_{\dom}=\lambda^{(h)}_{\kllevelrun} \left(\phi_{\kllevelrun}^{(h)},v^{(h)}\right)_{\dom} \quad \text{for all  } v^{(h)}\in \Vh=\spn\left\{\theta_{\dimVhrun}^{(h)} \ : \  1 \le \dimVhrun \le \dimVh \right\}.
\end{align}
Using $\phi_{\kllevelrun}^{(h)}:= \vec{\Phi}^{(h)}_{\kllevelrun}\cdot \vec{\theta}^{(h)}=\sum_{\dimVhrun=1}^{\dimVh} \phi_{\kllevelrun;\dimVhrun} \theta_{\dimVhrun}^{(h)}  \in \Vh$,
we derive for fixed but arbitrary $ v^{(h)} =  \vec{V}^{(h)}\cdot \vec{\theta}^{(h)} \in \Vh$, i.,e, for fixed but arbitrary $ \vec{V}^{(h)}\in \R^{\dimVh}$, the identity
\begin{align*}
&\left(\mc{R}^{(h)}\phi_{\kllevelrun}^{(h)},v^{(h)}\right)_{\dom}= \int_{\dom} \int_{\dom} \int_{\Omega} \lnrf^{(h)}(\vec{\om},\vec{x}) \lnrf^{(h)}(\vec{\om},\vec{x}^{\prime}) \, \mathrm{d}\prob(\vec{\om}) \phi_{\kllevelrun}^{(h)}(\vec{x}) \, \mathrm{d}\lebesgue(\vec{x}) v^{(h)}(\vec{x}^{\prime})  \, \mathrm{d}\lebesgue(\vec{x}^{\prime}) \\
&= \int_{\dom} \int_{\dom} \int_{\Omega}\vec{\lnrfrv}^{(h)}(\vec{\om})\cdot \vec{\theta}(\vec{x})   \vec{\lnrfrv}^{(h)}(\vec{\om})\cdot \vec{\theta}(\vec{x}^{\prime}) \, \mathrm{d}\prob(\vec{\om}) \vec{\Phi}_{\kllevelrun}^{(h)} \cdot \vec{\theta}(\vec{x}) \, \mathrm{d}\lebesgue(\vec{x}) \vec{V}^{(h)}\cdot \vec{\theta}(\vec{x}^{\prime})  \, \mathrm{d}\lebesgue(\vec{x}^{\prime}) \\
&=\int_{\Omega} \left( \int_{\dom}  \vec{\lnrfrv}^{(h)}(\vec{\om})\cdot \vec{\theta}(\vec{x})  \vec{\Phi}_{\kllevelrun}^{(h)} \cdot \vec{\theta}(\vec{x})  \, \mathrm{d}\lebesgue(\vec{x})\right) \left( \int_{\dom}  \vec{\lnrfrv}^{(h)}(\vec{\om})\cdot \vec{\theta}(\vec{x}^{\prime})   \vec{V}^{(h)}\cdot \vec{\theta}(\vec{x}^{\prime}) \, \mathrm{d}\lebesgue(\vec{x}^{\prime}) \right)  \, \mathrm{d}\prob(\vec{\om}) 
\\&=\int_{\Omega}    \vec{\lnrfrv}^{(h)}(\vec{\om}) \cdot   \massmatrix   \vec{\Phi}_{\kllevelrun}^{(h)}    \vec{\lnrfrv}^{(h)}(\vec{\om}) \cdot   \massmatrix   \vec{V}^{(h)} \, \mathrm{d}\prob(\vec{\om}) 
\\&= \int_{\Omega}  \left( \vec{V}^{(h)}\right)^{\transpose} \massmatrix      \vec{\lnrfrv}^{(h)}(\vec{\om}) \otimes    \vec{\lnrfrv}^{(h)}(\vec{\om})  \massmatrix \vec{\Phi}_{\kllevelrun}^{(h)}  \, \mathrm{d}\prob(\vec{\om}) =  \left( \vec{V}^{(h)}\right)^{\transpose} \massmatrix  \vec{\Sigma}_{  \vec{\lnrfrv}^{(h)}}\massmatrix \vec{\Phi}_{\kllevelrun}^{(h)}. 
 \end{align*}
Here $\massmatrix$ and $\stiffnessmatrix$ are the mass matrix and the stiffness matrix defined by
\begin{align}\label{stiffnessmatrix2}
\massmatrix=(\vec{\theta}^{(h)},{\vec{\theta}^{(h)}}^{\transpose})_{\dom} \quad \text{and} \quad
\stiffnessmatrix=\massmatrix\vec{\Sigma}_{\vec{\lnrfrv}^{(h)}}\massmatrix.  
\end{align}
Note that the symmetry of the mass matrix $\massmatrix$ implies the symmetry of $\stiffnessmatrix$. 

Thus, the  corresponding matrix form of the generalized eigenvalue problem \eqref{genalizedeigenvalue} is to seek the eigenpair $(\vec{\Phi}^{(h)}_{\kllevelrun} ,\lambda^{(h)}_{\kllevelrun})\in \mathbb{R}^{\dimVh}\times \mathbb{R}$ for $\kllevelrun=1,\cdots,\dimVh$, satisfying
\begin{align}\label{genalizedeigenvalue2}
\stiffnessmatrix \vec{\Phi}^{(h)}_{\kllevelrun} =\lambda^{(h)}_{\kllevelrun} \massmatrix \vec{\Phi}^{(h)}_{\kllevelrun}.
\end{align}
This is a generalized eigenvalue problem. We now transform it to a conventional eigenvalue problem. To this end, observe that the mass matrix $\massmatrix$ is symmetric positive-definite, and thus possesses a Cholesky factorization
\begin{align*}
\massmatrix=:\CholStiffnessL\left(\CholStiffnessL\right)^{\transpose} \in \R^{\dimVh \times \dimVh}.
\end{align*}
Now let $\tildstiffnessmatrix$ be a symmetric positive-definite matrix defined by
\begin{align*}
\tildstiffnessmatrix  :=\left(\CholStiffnessL\right)^{\transpose}\vec{\Sigma}_{\vec{\lnrfrv}^{(h)}}\CholStiffnessL.
\end{align*}
Then we can solve for $(\widetilde{\vec{\Phi}}^{(h)}_{\kllevelrun},\lambda^{(h)}_{\kllevelrun}) \in \R ^{\dimVh} \times \R$ for $\kllevelrun=1,\cdots,\dimVh$, satisfying
\begin{align}\label{genalizedeigenvalue3}
\tildstiffnessmatrix \widetilde{\vec{\Phi}}^{(h)}_{\kllevelrun} 
= \lambda^{(h)}_{\kllevelrun} \widetilde{\vec{\Phi}}^{(h)}_{\kllevelrun}.
\end{align}
Suppose now that $(\widetilde{\vec{\Phi}}^{(h)}_{\kllevelrun},\lambda^{(h)}_{\kllevelrun}) \in \R^{\dimVh} \times \R$ is an eigenpair of \eqref{genalizedeigenvalue3}.
Using 
\begin{align*}
\stiffnessmatrix=\massmatrix\vec{\Sigma}_{\vec{\lnrfrv}^{(h)}}\massmatrix=\CholStiffnessL\tildstiffnessmatrix \left(\CholStiffnessL\right)^{\transpose}
\end{align*}
 we obtain
\begin{align*}
\stiffnessmatrix \left(\CholStiffnessL\right)^{-\transpose}\widetilde{\vec{\Phi}}^{(h)}_{\kllevelrun} = \CholStiffnessL \tildstiffnessmatrix \widetilde{\vec{\Phi}}^{(h)}_{\kllevelrun} = \lambda^{(h)}\CholStiffnessL \widetilde{\vec{\Phi}}^{(h)}_{\kllevelrun}= \lambda^{(h)} \massmatrix  \left(\CholStiffnessL\right)^{-\transpose}\widetilde{\vec{\Phi}}^{(h)}_{\kllevelrun}.
 \end{align*}
 Thus, we have shown the equivalence 
\begin{align*}
&(\widetilde{\vec{\Phi}}^{(h)}_{\kllevelrun},\lambda^{(h)}_{\kllevelrun}) \in \R^{\dimVh} \times \R
\text{ is an eigenpair of \eqref{genalizedeigenvalue3}}\Longleftrightarrow\\&
\left( \vec{\Phi}^{(h)}_{\kllevelrun}:=\left(\CholStiffnessL\right)^{-\transpose}\widetilde{\vec{\Phi}}^{(h)}_{\kllevelrun},
\lambda^{(h)}_{\kllevelrun}\right) \in \R^{\dimVh} \times \R
\text{ is an eigenpair of \eqref{genalizedeigenvalue2}.}
\end{align*}
Note that,  with $\phi_{\kllevelrun}^{(h)}= \vec{\Phi}^{(h)}_{\kllevelrun}\cdot \vec{\theta}^{(h)}$, we obtain a Mercer-like representation 
\begin{align}\label{mercerh}
R^{(h)}(\vec{x},\vec{x}^{\prime})&= \sum_{\dimVhrun=1}^{\dimVh}\lambda^{(h)}_{\dimVhrun}\phi_{\dimVhrun}^{(h)}(\vec{x})\phi_{\dimVhrun}^{(h)}(\vec{x}^{\prime}) \approx R^{(\kllevel;h)}(\vec{x},\vec{x}^{\prime}):= \sum_{\dimVhrun=1}^{\kllevel}\lambda^{(h)}_{\dimVhrun}\phi_{\dimVhrun}^{(h)}(\vec{x})\phi_{\dimVhrun}^{(h)}(\vec{x}^{\prime}),
\end{align}
for $1\le \kllevel \le \dimVh $.

We now give a classical error analysis by means of the
approximation theory of conforming finite element methods.
To this end, consider the error operator
\begin{align}\label{eq:notationErr}
\mc{E}^{(h)}=\mc{R}-\mc{R}^{(h)}:L^{2}(\dom) \to L^{2}(\dom),
\end{align}
which is a self-adjoint
operator on $L^2(\dom)$. We have the following error representation.
\begin{lemma}\label{lemma:errOper}
The error operator $\mc{E}^{(h)}$ has the property
\[
\left(\mc{E}^{(h)} v,v\right)_{\dom}=\left(v,(I-\proj)\mc{R}(I+\proj)v\right)_{\dom}\quad \text{ for all } \quad v\in L^2(D),
\]
where we use the operator $\proj:L^{2}(\dom) \to \Vh$ from \eqref{eq:approxLinfty}.
\end{lemma}
\begin{proof}
For given $v\in L^2(\dom)$ and since $\mc{R}^{(h)}=\proj\mc{R} \proj$ and $(\mc{R} \proj v-\proj\mc{R} v, v)=0$, we obtain
\begin{align*}
\left(\mc{E}^{(h)} v,v\right)_{\dom}&=\left(\left(\mc{R}-\proj\mc{R} \proj\right)v,v\right)_{\dom}+\left(\mc{R} \proj v-\proj\mc{R} v, v\right)_{\dom}\\
&=\left(\left(I-\proj\right)\mc{R}\left(I+\proj\right)v,v\right)_{\dom},
\end{align*}
which gives the assertion.
\end{proof}
A direct consequence of Lemma \ref{lemma:errOper} is the upper bound for the operator norm of $\mc{E}^{(h)}$
\begin{align}\label{eq:E_{h}}
\norm{\mc{E}^{(h)}}_{\mc{B}(L^2(\dom))}\leq 2C_{\Pi_{\Vh;L^2}} h^{s}\norm{\mc{R}}_{\mc{B}(L^2(\dom), H^s(\dom))},
\end{align}
where we employed \eqref{eq:approxLinfty}.

Finally, we are ready to present the main result of this section.
\begin{proposition}[Conforming Galerkin approximation estimate]\label{prop:babuska}
Let \eqref{eqA11} from Assumption \ref{A:11} hold.  Then there are constants $C_1$, $C_2$ and $h_0$ such that
\[
\left| \lambda_{\kllevelrun}^{(h)}-\lambda_{\kllevelrun}\right|\leq C_1\lambda_{\kllevelrun}^{-1}h^{2s} \quad\text{ for all }\quad 0<h\leq h_0.\]
Furthermore, the eigenvectors $\{\phi_{\kllevelrun}^{(h)}\}_{\kllevelrun=1}^{\dimVh}$ of $\mc{R}^{(h)}$ can be selected such that
\[
\normL{\phi_{\kllevelrun}^{(h)}-\phi_{\kllevelrun}}{\dom}\leq C_{2}\lambda_{\kllevelrun}^{-1} h^{s}\quad\text{ for all }\quad 0<h\leq h_0.\]
Here, the constants $C_1$ and $C_2$ are independent of $h$ and $h_0>0$ is to be sufficiently small.
\end{proposition}
\begin{proof}The proof follows from \cite[Theorem 9.1]{babuska&osborn91}.
\end{proof}
A consequence of Proposition \ref{prop:babuska} is that, for any $\epsilon>0$  and for the Lebesgue measure $\lebesgue$, there holds the estimate
\begin{align}\label{convmeas}
\lebesgue\left(\left\{ \ \left|\phi_{\kllevelrun}^{(h)}-\phi_{\kllevelrun} \right|\ge \epsilon \right\} \right)
\leq \epsilon^{-2} \int_{\dom} \left|\phi_{\kllevelrun}^{(h)}-\phi_{\kllevelrun} \right|^2 \, \mathrm{d}\lebesgue
=\epsilon^{-2}\normL{\phi_{\kllevelrun}^{(h)}-\phi_{\kllevelrun}}{\dom}^2\le C_{2}\epsilon^{-2}\lambda_{\kllevelrun}^{-2} h^{2s}.
\end{align}

\subsection{Approximation of the covariance operator from samples}\label{subsec:approxcovhM}
We will approximate the spectral decomposition \eqref{mercerh}, which is based on the unknown stiffness matrix $\stiffnessmatrix$ by approximating the true covariance matrix $\vec{\Sigma}_{\vec{\lnrfrv}^{(h)}}\in \R^{\dimVh \times \dimVh}$.
Let us assume for the moment that we have such an estimate $\vec{\Sigma}^{(h;\nsamples)}_{\vec{\lnrfrv}^{(h)}}\in \R^{\dimVh \times \dimVh},$
where $\vec{\Sigma}^{(h;\nsamples)}_{\vec{\lnrfrv}^{(h)}}$ is also symmetric. Then by repeating the procedure for the derivation  \eqref{mercerh} in an analogous way, we obtain the approximations 
\begin{align}
\stiffnessmatrixM&:=\left(\massmatrix\right)^{\transpose}\vec{\Sigma}^{(h;\nsamples)}_{\vec{\lnrfrv}^{(h)}} \massmatrix  \approx \stiffnessmatrix=\left(\massmatrix\right)^{\transpose}\vec{\Sigma}^{(h)}_{\vec{\lnrfrv}^{(h)}} \massmatrix  \label{stiffnessample}\quad \text{and} \quad\\
\tildestiffnessmatrixM&:=\left(\CholStiffnessL\right)^{\transpose}\vec{\Sigma}^{(h;\nsamples)}_{\vec{\lnrfrv}^{(h)}}
\CholStiffnessL \approx  \tildstiffnessmatrix=\left(\CholStiffnessL \right)^{\transpose}\vec{\Sigma}^{(h)}_{\vec{\lnrfrv}^{(h)}}\CholStiffnessL  \label{tildestiffnessample},
\end{align}
and we encounter the following eigenvalue problem: Seek $(\widetilde{\vec{\Phi}}^{(h;\nsamples)}_{\kllevelrun},\lambda^{(h;\nsamples)}_{\kllevelrun} ) \in \R^{\dimVh} \times \R$ for all $\kllevelrun=1,\cdots,\dimVh$ such that
\begin{align}\label{genalizedeigenvalue3-estimator}
 \tildestiffnessmatrixM\widetilde{\vec{\Phi}}^{(h;\nsamples)}_{\kllevelrun} = \lambda^{(h;\nsamples)}_{\kllevelrun} \widetilde{\vec{\Phi}}^{(h;\nsamples)}_{\kllevelrun}.
\end{align}
As $\tildestiffnessmatrixM$ is a symmetric matrix, we have that $\left(\widetilde{\vec{\Phi}}^{(h;\nsamples)}_{\kllevelrun} ,\widetilde{\vec{\Phi}}^{(h;\nsamples)}_{\kllevelrun^{\prime}}  \right)_{\ell_2}=\delta_{\kllevelrun, \kllevelrun^{\prime}}$ after normalization.
Moreover, we encounter the generalized eigenvalue problem
\begin{align}\label{genalizedeigenvalue4-estimator}
 \stiffnessmatrixM  \vec{\Phi}^{(h;\nsamples)}_{\kllevelrun} = \lambda^{(h;\nsamples)}_{\kllevelrun} \massmatrix \vec{\Phi}^{(h;\nsamples)}_{\kllevelrun}.
\end{align}
Analogously, we have the equivalence 
\begin{align*}
&(\widetilde{\vec{\Phi}}^{(h;\nsamples)}_{\kllevelrun},\lambda^{(h;\nsamples)}_{\kllevelrun}) \in \R^{\dimVh} \times \R
\text{ is an eigenpair of \eqref{genalizedeigenvalue3-estimator}}  \Longleftrightarrow \\ 
&\left( \vec{\Phi}^{(h;\nsamples)}_{\kllevelrun}:=\left(\CholStiffnessL\right)^{-\transpose}\widetilde{\vec{\Phi}}^{(h;\nsamples)}_{\kllevelrun},
\lambda^{(h;\nsamples)}_{\kllevelrun}\right) \in \R^{\dimVh} \times \R
\text{ is an eigenpair of  \eqref{genalizedeigenvalue4-estimator}} .
\end{align*}
Thus we can derive approximations to $\tildstiffnessmatrix$ and $\stiffnessmatrix$ by
\begin{align}\label{stiffnessmatrixShhat}
\tildestiffnessmatrixM=\sum_{\dimVhrun=1}^{\dimVh} \lambda^{(h;\nsamples)}_{\dimVhrun}\widetilde{\vec{\Phi}}^{(h;\nsamples)}_{\dimVhrun}\otimes \widetilde{\vec{\Phi}}^{(h;\nsamples)}_{\dimVhrun}\quad \text{and} \quad  \stiffnessmatrixM=\sum_{\dimVhrun=1}^{\dimVh} \lambda^{(h;\nsamples)}_{\dimVhrun}\vec{\Phi}^{(h;\nsamples)}_{\dimVhrun}\otimes \vec{\Phi}^{(h;\nsamples)}_{\dimVhrun} \in \R^{\dimVh \times \dimVh}.
\end{align}
Then the eigensystem \eqref{stiffnessmatrixShhat} gives rise to a computable Mercer-like expansion
\begin{align}\label{mercerh2a}
R^{(h;\nsamples)}(\vec{x},\vec{x}^{\prime}) :=\sum_{\dimVhrun=1}^{\dimVh}\lambda^{(h;\nsamples)}_{\dimVhrun}
\phi^{(h;\nsamples)}_{\dimVhrun}(\vec{x}) \, \phi^{(h;\nsamples)}_{\dimVhrun}(\vec{x}^{\prime}).
\end{align}
Here, we use analogously 
\begin{align}\label{Sheigenfunc2}
\phi_{k}^{(h;\nsamples)}:=\vec{\Phi}^{(h;\nsamples)}_{k} \cdot \vec{\theta}^{(h)} \text{ for all } k=1,\cdots,\dimVh.
\end{align}
We observe, using $
\massmatrix=\CholStiffnessL\left(\CholStiffnessL\right)^{\transpose} \in \R^{\dimVh \times \dimVh},
$
that 
\begin{align}\label{philhmortho}
\left( \phi_{k}^{(h;\nsamples)}, \phi_{k^{\prime}}^{(h;\nsamples)}\right)_{L_2(\dom)}&= \left(\vec{\Phi}^{(h;\nsamples)}_{k}\right)^{\transpose}\massmatrix \vec{\Phi}^{(h;\nsamples)}_{k^{\prime}}= \left( \left(\CholStiffnessL\right)^{\transpose}\vec{\Phi}^{(h;\nsamples)}_{k},   \left(\CholStiffnessL\right)^{\transpose} \vec{\Phi}^{(h;\nsamples)}_{k^{\prime}}\right)_{\ell_2}
\nonumber \\&=\left(\widetilde{\vec{\Phi}}^{(h;\nsamples)}_{k}, \widetilde{\vec{\Phi}}^{(h;\nsamples)}_{k^{\prime}}\right)_{\ell_2}=\delta_{k,k^{\prime}}.
\end{align}
To bound the sampling error for the eigenvalues, we can
apply Weyl's inequality (see for instance \cite[(Eq. 3.1)]{IpsenNadler09}) to any of the two eigenvalue problems  \eqref{genalizedeigenvalue3-estimator} or  \eqref{genalizedeigenvalue4-estimator}, respectively. We will focus on  \eqref{genalizedeigenvalue3-estimator} as  \eqref{genalizedeigenvalue4-estimator} is a generalized eigenvalue problem, which makes it harder to derive bounds on the eigenvectors.
We obtain
\begin{align}
\label{weyleigenvalue}
\left| \lambda^{(h)}_{\kllevelrun} -\lambda^{(h;\nsamples)}_{\kllevelrun} \right| \le \left\|\tildstiffnessmatrix -\tildestiffnessmatrixM \right\|_{2 \to 2}, \quad 1\le  \kllevelrun \le \min\{\dimVh,\kllevel\}.
\end{align}
Now we can take Proposition \ref{prop:babuska} into account to get that, under Assumption \ref{A:11}, there are constants $C_1$ and $h_0$ such that
\begin{align}\label{weyleigenvalue2}
\left| \lambda_{\kllevelrun} -\lambda^{(h;\nsamples)}_{\kllevelrun} \right|&\le \left| \lambda_{\kllevelrun}-\lambda^{(h)}_{\kllevelrun}\right| + \left| \lambda^{(h)}_{\kllevelrun} -\lambda^{(h;\nsamples)}_{\kllevelrun} \right|\nonumber \\ & \leq C_1\lambda_{\kllevelrun}^{-1}h^{2s}+\left\|\tildstiffnessmatrix-\tildestiffnessmatrixM \right\|_{2 \to 2} \quad\text{ for all }\quad 0<h\leq h_0 \text{ and } 1\le \kllevelrun \le \min\{\dimVh,\kllevel\}.
\end{align}
For perturbation bounds on the eigenvectors, see \cite{varah70}, and for a recent versions of the Davis-Kahan theorem including randomness, see  \cite{orourke18}.
We use the variant presented in \cite[Corollary 1]{yu15}. To this end, recall the assumption that the eigenvalues are ordered non-increasingly, i.e., $\lambda^{(h)}_{1}\ge \dots \ge \lambda^{(h)}_{\dimVh}$ and $\lambda^{(h;\nsamples)}_{1}\ge \dots \ge \lambda^{(h;\nsamples)}_{\dimVh}$.
We define the discrete spectral gap 
\begin{align}\label{spectralgap1}
\delta^{(h;\nsamples)}_{\kllevelrun}:=\min \left\{\left|\lambda^{(h;\nsamples)}_{\kllevelrun-1}- \lambda^{(h)}_{\kllevelrun}\right|, \left|\lambda^{(h)}_{\kllevelrun}- \lambda^{(h;\nsamples)}_{\kllevelrun+1}\right|\right\}, \quad  \text{with }\lambda^{(h;\nsamples)}_{0}=\infty \text{ and }  \lambda^{(h;\nsamples)}_{\dimVh+1}=-\infty.
\end{align}
The quantity $\delta_{\kllevelrun}^{(h;\nsamples)}$ is problematic as it contains $\lambda^{(h)}_{\kllevelrun}, \lambda^{(h;\nsamples)}_{\kllevelrun\pm 1}$. We therefore  aim to replace it by a quantity which only depends on the continuous problem. To this end, we adopt the strategy presented in \cite{yu15}:
Let $\delta_{\kllevelrun}$ be the continuous spectral gap defined by
\begin{align}\label{defn:gap}
\delta_{\kllevelrun}:=\min \left\{\lambda_{\kllevelrun-1}- \lambda_{\kllevelrun}, \lambda_{\kllevelrun}- \lambda_{\kllevelrun+1}\right\} \text{ with } \lambda_0=\infty.
\end{align}

\begin{assumption}[Spectral gap]\label{ass:spectralgap}
Assume that there is a sufficiently small positive parameter $h_{1}\le h_0$ and assume that
$\tildestiffnessmatrixM$ is a good approximation to $\tildstiffnessmatrix$, i.e., for $h\le h_1$, there holds
\begin{align}
\label{condspectralgap1}
\delta_{\kllevelrun}\ge 4C_1h^{2s} \lambda_{\kllevelrun+1}^{-1} +4\left\|\tildstiffnessmatrix-\tildestiffnessmatrixM \right\|_{2 \to 2}.
\end{align}
\end{assumption}
\begin{theorem}\label{theoremeigenfucntions}
Let  \eqref{condspectralgap1} be valid. It then holds
\begin{align}\label{spectralgap3}
\delta^{(h;\nsamples)}_{\kllevelrun} \ge \frac{1}{4} \delta_{\kllevelrun}.
\end{align}
Furthermore it holds 
\begin{align}
\label{kahandaviscor}
\left\| \widetilde{\vec{\Phi}}^{(h)}_{\kllevelrun}- \widetilde{\vec{\Phi}}^{(h;\nsamples)}_{\kllevelrun} \right\|_2\le C \frac{\left\|\tildstiffnessmatrix-\tildestiffnessmatrixM \right\|_{2\to 2}}{\delta^{(h;\nsamples)}_{\kllevelrun}} \le 4C \frac{\left\|\tildstiffnessmatrix-\tildestiffnessmatrixM \right\|_{2\to 2}}{\delta_{\kllevelrun}},
\end{align}
where we fixed a sign for $\widetilde{\vec{\Phi}}^{(h;\nsamples)}_{\kllevelrun}$ such that $\widetilde{\vec{\Phi}}^{(h)}_{\kllevelrun}\cdot \widetilde{\vec{\Phi}}^{(h;\nsamples)}_{\kllevelrun}\ge 0$.
Moreover for $\phi_{\kllevelrun}^{(h;\nsamples)}:=\left(\CholStiffnessL\right)^{\transpose} \widetilde{\vec{\Phi}}^{(h;\nsamples)}_{k} \cdot \vec{\theta}^{(h)}$ the bound
\begin{align}
\label{phivecerror}
\left\|\phi_{\kllevelrun}^{(h)}- \phi_{\kllevelrun}^{(h;\nsamples)} \right\|_{L_{2}(\dom)}\le C \frac{\left\|\tildstiffnessmatrix-\tildestiffnessmatrixM \right\|_{2\to 2}}{\delta^{(h;\nsamples)}_{\kllevelrun}} \le 4C \frac{ \left\|\tildstiffnessmatrix-\tildestiffnessmatrixM \right\|_{2\to 2}}{\delta_{\kllevelrun}}
\end{align}
holds.
\end{theorem}
\begin{proof}
The first assertion is a generalization of an argument presented in \cite{yu15} to the case of several discretization parameters.
Recall the discrete spectral gap $\delta^{(h;\nsamples)}_{\kllevelrun}$ in \eqref{spectralgap1}. 
We show \eqref{spectralgap3} by the following strategy:
First, we observe
\begin{align*}
\delta_{\kllevelrun}&\le \left|\lambda_{\kllevelrun-1}- \lambda_{\kllevelrun} \right|=\left|\lambda_{\kllevelrun-1}-\lambda^{(h)}_{\kllevelrun-1}+\lambda^{(h)}_{\kllevelrun-1}-\lambda^{(h)}_{\kllevelrun}+\lambda^{(h)}_{\kllevelrun}- \lambda_{\kllevelrun} \right|\\
&\le \left|\lambda_{\kllevelrun-1}-\lambda^{(h)}_{\kllevelrun-1}\right|+\left|\lambda^{(h)}_{\kllevelrun-1}-\lambda^{(h)}_{\kllevelrun}\right|+\left|\lambda^{(h)}_{\kllevelrun}- \lambda_{\kllevelrun} \right|.
\end{align*}
Similarly, we obtain
\begin{align*}
\delta_{\kllevelrun}&\le \left|\lambda_{\kllevelrun}- \lambda_{\kllevelrun+1} \right|=\left|\lambda_{\kllevelrun}-\lambda^{(h)}_{\kllevelrun}+\lambda^{(h)}_{\kllevelrun}-\lambda^{(h)}_{\kllevelrun+1}+\lambda^{(h)}_{\kllevelrun+1}- \lambda_{\kllevelrun+1} \right|\\
&\le \left|\lambda_{\kllevelrun}-\lambda^{(h)}_{\kllevelrun}\right|+\left|\lambda^{(h)}_{\kllevelrun}-\lambda^{(h)}_{\kllevelrun+1}\right|+\left|\lambda^{(h)}_{\kllevelrun+1}- \lambda_{\kllevelrun+1} \right|.
\end{align*}
Using Proposition \ref{prop:babuska} and the fact that the eigenvalues are sorted, i.e., $\lambda_{\kllevelrun-1} \ge \lambda_{\kllevelrun}\ge \lambda_{\kllevelrun+1}$, we obtain for all $0<h\leq h_0$ the bound
\begin{align*}
\max\left\{\left|\lambda_{\kllevelrun-1}-\lambda^{(h)}_{\kllevelrun-1}\right|,\left|\lambda^{(h)}_{\kllevelrun}- \lambda_{\kllevelrun} \right|,\left|\lambda_{\kllevelrun+1}-\lambda^{(h)}_{\kllevelrun+1}\right| \right\} \le C_1 \lambda_{\kllevelrun+1}^{-1}h^{2s} \le  \frac{1}{4} \delta_{\kllevelrun},
\end{align*}
where we have used the spectral gap assumption \eqref{condspectralgap1} in the last step.  
Hence, we infer
\begin{align*}
\delta_{\kllevelrun} \le \min \left\{ \left|\lambda^{(h)}_{\kllevelrun-1}-\lambda^{(h)}_{\kllevelrun}\right|,\left|\lambda^{(h)}_{\kllevelrun}-\lambda^{(h)}_{\kllevelrun+1}\right|\right\} + \frac{1}{2}\delta_{\kllevelrun} = \delta^{(h)}_{\kllevelrun} + \frac{1}{2}\delta_{\kllevelrun}.
\end{align*}
This inequality implies
\begin{align}\label{spectralgap_auxiliary1}
\frac{1}{2} \delta_{\kllevelrun} \le \delta^{(h)}_{\kllevelrun} =\min
\left\{ \left|\lambda^{(h)}_{\kllevelrun-1}-\lambda^{(h)}_{\kllevelrun}\right|,\left|\lambda^{(h)}_{\kllevelrun}-\lambda^{(h)}_{\kllevelrun+1}\right|\right\}.
\end{align}
Moreover, we have
\begin{align*}
\delta^{(h)}_{\kllevelrun}\le \left|\lambda^{(h)}_{\kllevelrun-1}- \lambda^{(h)}_{\kllevelrun} \right|=\left|\lambda^{(h)}_{\kllevelrun-1}-\lambda^{(h;\nsamples)}_{\kllevelrun-1}+\lambda^{(h;\nsamples)}_{\kllevelrun-1}- \lambda^{(h)}_{\kllevelrun} \right|\le \left|\lambda^{(h)}_{\kllevelrun-1}-\lambda^{(h;\nsamples)}_{\kllevelrun-1}\right|+\left|\lambda^{(h;\nsamples)}_{\kllevelrun-1}- \lambda^{(h)}_{\kllevelrun} \right|
\end{align*}
and also
\begin{align*}
\delta^{(h)}_{\kllevelrun}\le \left|\lambda^{(h)}_{\kllevelrun}- \lambda^{(h)}_{\kllevelrun+1} \right|=\left|\lambda^{(h)}_{\kllevelrun}-\lambda^{(h;\nsamples)}_{\kllevelrun}+\lambda^{(h;\nsamples)}_{\kllevelrun}- \lambda^{(h)}_{\kllevelrun+1} \right|\le \left|\lambda^{(h)}_{\kllevelrun}-\lambda^{(h;\nsamples)}_{\kllevelrun}\right|+\left|\lambda^{(h;\nsamples)}_{\kllevelrun}- \lambda^{(h)}_{\kllevelrun+1} \right|.
\end{align*}
Using Weyl's Theorem, i.e., \eqref{weyleigenvalue}, and the spectral gap assumption \eqref{condspectralgap1}, we obtain
\begin{align*}
\max\left\{ \left|\lambda^{(h)}_{\kllevelrun-1}-\lambda^{(h;\nsamples)}_{\kllevelrun-1}\right|,\left|\lambda^{(h)}_{\kllevelrun}-\lambda^{(h;\nsamples)}_{\kllevelrun}\right| \right\}\le \left\| \tildstiffnessmatrix-\tildestiffnessmatrixM \right\|_{2 \to 2} \le \frac{1}{4}\delta_{\kllevelrun} \le \frac{1}{2}\delta^{(h)}_{\kllevelrun},
\end{align*}
where we have used \eqref{spectralgap_auxiliary1} in the last step. Hence, we obtain
\begin{align*}
\delta^{(h;\nsamples)}_{\kllevelrun}= \min\left\{ \left|\lambda^{(h)}_{\kllevelrun}-\lambda^{(h;\nsamples)}_{\kllevelrun-1}\right|,\left|\lambda^{(h)}_{\kllevelrun+1}-\lambda^{(h;\nsamples)}_{\kllevelrun}\right| \right\}\ge \frac{1}{2}\delta^{(h)}_{\kllevelrun} \ge \frac{1}{4}\delta_{\kllevelrun}
\end{align*}
where have used again \eqref{spectralgap_auxiliary1} in the last step. 

Moreover, we observe 
\begin{align*}
\left\|\phi_{\kllevelrun}^{(h)}- \phi_{\kllevelrun}^{(h;\nsamples)} \right\|^2_{L_{2}(\dom)}= \left( \Phi^{(h)}_{\kllevelrun}-\Phi^{(h;\nsamples)}_{\kllevelrun}\right)^{\transpose} \massmatrix \left( \Phi^{(h)}_{\kllevelrun}-\Phi^{(h;\nsamples)}_{\kllevelrun}\right)=  \left\| \widetilde{\Phi}^{(h)}_{\kllevelrun}-\widetilde{\Phi}^{(h;\nsamples)}_{\kllevelrun}\right\|^2_{\ell_2}.
\end{align*}
The two other assertions follow directly with the Davis-Kahan theorem as in \cite[Corollary1]{yu15}.
\end{proof}
Finally, we are left with bounding $\left\| \tildstiffnessmatrix-\tildestiffnessmatrixM \right\|_{2\to 2}$. 
To this end, we first observe 
\begin{align*}
\tildstiffnessmatrix-\tildestiffnessmatrixM&:=\left(\CholStiffnessL\right)^{\transpose}\left(\vec{\Sigma}_{\vec{\lnrfrv}^{(h)}}- \vec{\Sigma}^{(h;\nsamples)}_{\vec{\lnrfrv}^{(h)}}\right)\CholStiffnessL  \in \R^{\dimVh \times \dimVh}.
\end{align*}
Moreover, we have 
\begin{align}\label{tildeSopnorm}
\left\| \tildstiffnessmatrix-\tildestiffnessmatrixM \right\|_{2\to 2}& =
	\sup_{\vec{x} \in \R^{\dimVh}\setminus \{ \vec{0}\}} \frac{\left( \left(\tildstiffnessmatrix-\tildestiffnessmatrixM \right) \vec{x},\vec{x}\right)_{\ell_2}}{\| \vec{x}\|^2}\nonumber \\&= \sup_{\vec{y} \in \R^{\dimVh}\setminus \{ \vec{0}\}} \frac{\left( \left( \vec{\Sigma}_{\vec{\lnrfrv}^{(h)}}- \vec{\Sigma}^{(h;\nsamples)}_{\vec{\lnrfrv}^{(h)}}\right) \vec{y},\vec{y}\right)_{\ell_2}}{\|\left(\CholStiffnessL\right)^{-1}  \vec{y}\|_{\ell_2}^2},
\end{align}
where the choice $ \vec{y}=\CholStiffnessL \vec{x}$ is admissible since $\CholStiffnessL$ is regular.
Note here also that we have
\begin{align*}
\lambda_{\min}\left(\left(\massmatrix\right)^{-1}\right) \|\vec{y}\|^{2} \le \left(\left( \massmatrix\right)^{-1}\vec{y},\vec{y} \right) \le \lambda_{\max}\left(\left(\massmatrix\right)^{-1}\right)\|\vec{y}\|^{2}
\end{align*}
and we observe $\lambda_{\min}\left(\left(\massmatrix\right)^{-1}\right)=\lambda^{-1}_{\max}\left(\massmatrix\right)$ and $\lambda_{\max}\left(\left(\massmatrix\right)^{-1}\right)=\lambda^{-1}_{\min}\left(\massmatrix\right)$. 
Consequently,  we obtain the bound
\begin{equation}\label{lieqneu}
	\lambda_{\min}\left(\massmatrix\right) \left\| \vec{\Sigma}_{\vec{\lnrfrv}^{(h)}}- \vec{\Sigma}^{(h;\nsamples)}_{\vec{\lnrfrv}^{(h)}}  \right\|_{2\to 2} \le \left\| \tildstiffnessmatrix-\tildestiffnessmatrixM \right\|_{2\to 2} \le \lambda_{\max}\left(\massmatrix\right) \left\| \vec{\Sigma}_{\vec{\lnrfrv}^{(h)}}- \vec{\Sigma}^{(h;\nsamples)}_{\vec{\lnrfrv}^{(h)}}  \right\|_{2\to 2}.
\end{equation}
The matrix norm $ \left\| \vec{\Sigma}_{\vec{\lnrfrv}^{(h)}}- \vec{\Sigma}^{(h;\nsamples)}_{\vec{\lnrfrv}^{(h)}}  \right\|_{2\to 2}$
will be estimated in probabilistic terms in Section \ref{subsec:approxcovhMfinal}.

\subsection{Sub-Gaussian tails for observed random vectors}\label{sec:tails}
In order to invoke recent results for bounding the sampling error, we first need to consider the sub-Gaussian property of the random variables \eqref{observedrandomvector}, i.e.,  of the random vector $\vec{\lnrfrv}^{(h)}$. 
This random vector consists of the coefficients of the discrete random field $\lnrf^{(h)}:\Omega \times \dom \to \R$, i.e., 
\begin{align*}
\lnrf^{(h)}(\vec{\om},\vec{x})= \sum_{\dimVhrun=1}^{\dimVh} \lnrfrv^{(h)}_{\dimVhrun}(\vec{\om}) \theta_{\dimVhrun}^{(h)}(\vec{x}) = \vec{\lnrfrv}^{(h)}(\vec{\om}) \cdot \vec{\theta}(\vec{x}).
\end{align*}
We mainly focus on two different basis functions. First, we choose $\{\theta^{(h)}_{\dimVhrun}\}_{1\le \dimVhrun \le \dimVh}$ to be a nodal basis and, next,  we choose $\{\theta^{(h)}_{\dimVhrun}\}_{1\le \dimVhrun \le \dimVh}$ to be an $L^{2}$-orthonormal basis.
\subsubsection{Nodal basis}
To start, let us consider the setting of so-called standard information. To this end, 
recall the point set $X_{\dimVh}=\left\{ \vec{x}_1,\dots,\vec{x}_{\dimVh}\right\}\subset \dom$, which determines the nodal basis functions and thus the degrees of freedom for the finite element space $\Vh$.
In this case, it holds 
\begin{align*}
\lnrf^{(h)}(\vec{\om},\vec{x}) = \vec{\lnrfrv}^{(h)}(\vec{\om}) \cdot \vec{\theta}(\vec{x})=\sum_{\dimVhrun=1}^{\dimVh} \lnrf^{(h)}_{\dimVhrun}(\vec{\om},\vec{x}_{\dimVhrun})\theta_{\dimVhrun}^{(h)}(\vec{x}).
\end{align*}
Thus, we have
\begin{align}\label{observedrandomvectornodal}
\vec{\lnrfrv}^{(h)}: \Omega \to \R^{\dimVh}, \quad \vec{\om} \mapsto \vec{\lnrfrv}^{(h)}(\vec{\om})=\left(\lnrf^{(h)}(\vec{\om},\vec{x}_1),\dots,\lnrf^{(h)}(\vec{\om},\vec{x}_{\dimVh}) \right)^{\transpose}.
\end{align}
Since $\lnrf^{(h)}$ is assumed to be a centered Gaussian random field, it follows by definition that the random vector $\vec{\lnrfrv}^{(h)}$ is also  distributed according to a multi-variate Gaussian law, i.e., 
\begin{align*}
\vec{\lnrfrv}^{(h)}\sim \mathcal{N}\left(\vec{0},\vec{\Sigma}_{\vec{K}^{(h)}} \right)
\end{align*}
with 
\begin{align*}
\vec{\Sigma}_{\vec{K}^{(h)}}
:=\expect{\left(\vec{\lnrfrv}^{(h)}-\expect{\vec{\lnrfrv}^{(h)}} \right) \otimes \left(\vec{\lnrfrv}^{(h)}-\expect{\vec{\lnrfrv}^{(h)}} \right)}
=\expect{\vec{\lnrfrv}^{(h)} \otimes \vec{\lnrfrv}^{(h)} }.
\end{align*}
We can invoke \eqref{finitemomentsh} and \eqref{truecovariance}, i.e., $\left( \vec{\Sigma}_{\vec{K}^{(h)}}\right)_{\dimVhrun,\dimVhrun^{\prime}} \le 4(\constC^{(h)})^2$ . 
Finally, we obtain by Chernoff's inequality \cite{cherstat} the bound
\begin{align*}
	&\prob\left\{  \vec{v} \cdot\left( \vec{\lnrfrv}^{(h)}-\expect{\vec{\lnrfrv}^{(h)}}\right) > t\right\} \le  \exp\left(-\frac{t^2}{8 (\constC^{(h)})^2} \right)
\end{align*}
for any $\vec{v} $ with $\| \vec{v} \|_2 =1$.
Note that this implies {$4 (\constC^{(h)})^2=\rho^{-1}$} in the sense of \cite[Definition 1]{cai2016}.
\subsubsection{$L^2$-orthonormal basis}
Now let us consider the setting of so-called linear information. Here, we can not directly rely on the definition of a centered Gaussian random field to deduce that the coefficients are distributed with a multivariate normal law.
In this case, the random vector  \eqref{observedrandomvector} gets encoded as
\begin{align*}
\vec{\lnrfrv}^{(h)}:\Omega \to \R^{\dimVh}, \quad \vec{\om} &\mapsto \left(\int_{\dom}  \lnrf^{(h)} (\vec{\om},\vec{x}) \theta^{(h)}_{1}(\vec{x})  \, \mathrm{d}\lebesgue(\vec{x}), \dots,\int_{\dom}  \lnrf^{(h)} (\vec{\om},\vec{x}) \theta^{(h)}_{\dimVh}(\vec{x})  \, \mathrm{d}\lebesgue(\vec{x})\right)^{\transpose}\\
&=\int_{\dom}\lnrf^{(h)} (\vec{\om},\vec{x}) \vec{\theta}^{(h)}(\vec{x})\, \mathrm{d}\lebesgue(\vec{x}),
\end{align*}
where $\{\theta_j^{(h)}, j=1,\ldots,\dimVh\}$ is an $L_2$-orthonormal basis.
The individual random variables are given as
\begin{align}\label{Ycomponents}
\lnrfrv^{(h)}_{\dimVhrun}:\Omega \to \R^{\dimVh}, \quad \vec{\om} \mapsto \int_{\dom}  \lnrf^{(h)} (\vec{\om},\vec{x}) \theta^{(h)}_{\dimVhrun}(\vec{x})  \, \mathrm{d}\lebesgue(\vec{x}).
\end{align}
Hence, for the expected value, we observe
\begin{align*}
\expect{\lnrfrv^{(h)}_{\dimVhrun}}&=\int_{\Omega} \int_{\dom}  \lnrf^{(h)} (\vec{\om},\vec{x}) \theta^{(h)}_{\dimVhrun}(\vec{x})  \, \mathrm{d}\lebesgue(\vec{x})\, \mathrm{d}\prob(\vec{\om})= \int_{\dom} \int_{\Omega} \lnrf^{(h)} (\vec{\om},\vec{x}) \, \mathrm{d}\prob(\vec{\om})\theta^{(h)}_{\dimVhrun}(\vec{x})  \, \mathrm{d}\lebesgue(\vec{x})\\&=\int_{\dom}\expect{\lnrf^{(h)} (\cdot,\vec{x})} \theta^{(h)}_{\dimVhrun}(\vec{x})  \, \mathrm{d}\lebesgue(\vec{x})=0,
\end{align*}
which implies
\begin{align*}
\expect{\vec{\lnrfrv}^{(h)}}=\int_{\dom}\expect{\lnrf^{(h)} (\cdot,\vec{x})} \vec{\theta}^{(h)}(\vec{x})  \, \mathrm{d}\lebesgue(\vec{x})=\vec{0}.
\end{align*}
For the variance, we observe
\begin{align*}
&\expect{\left(\vec{\lnrfrv}^{(h)}-\expect{\vec{\lnrfrv}^{(h)}} \right) \otimes \left(\vec{\lnrfrv}^{(h)}-\expect{\vec{\lnrfrv}^{(h)}} \right)}
=\expect{\vec{\lnrfrv}^{(h)} \otimes \vec{\lnrfrv}^{(h)} }=:\vec{\Sigma}_{\vec{K}^{(h)}},
\end{align*}
where
\begin{equation}\label{li2neu}
	\left(\vec{\Sigma}_{\vec{K}^{(h)}}\right)_{k,k^{\prime}} =\int_{\dom} \int_{\dom}
	\text{Cov}_{\lnrf^{(h)}}(\vec{x},\vec{x}^{\prime}) \theta_{k}^{h}(\vec{x})\theta_{k^{\prime}}^{h}(\vec{x}^{\prime}) \mathrm{d}\vec{x} \mathrm{d} \vec{x}^{\prime}
\end{equation}
and
\begin{equation}\label{li2neu2}
	\left(\vec{\Sigma}_{\vec{K}^{(h)}}\right)_{k,k^{\prime}} \leq 
	\|\text{Cov}_{\lnrf^{(h)}}\|_{L^2(\dom \times \dom)} \leq 4 (\constC^{(h)})^2 \cdot | \dom|.
\end{equation}
For all $\vec{v} \in \R^{\dimVh}$ with $\| \vec{v}\|_2=1$, we have
\begin{align*}
\vec{v}\cdot \left(\vec{\lnrfrv}^{(h)}-\expect{\vec{\lnrfrv}^{(h)}}\right)
	&=\sum_{k=1}^{\dimVh} \left(\lnrf^{(h)}(\om_j,\cdot), \theta_j^{(h)} \right)_{L^2(\dom)} v_j \sim \mathcal{N} (0,\vec{v}\cdot \vec{\Sigma}_{\vec{K}^{(h)}}\vec{v}).
\end{align*}
Furthermore, given $\vec{x} \in \dom$, it holds
\begin{align}\label{distxfixed}
\lnrf^{(h)}(\cdot,\vec{x}) \sim \mathcal{N}\left(0,\text{Cov}_{\lnrf^{(h)}}(\vec{x},\vec{x})\right).
\end{align}
Therefore we can invoke Chernoff's inequality to obtain for all $\vec{x} \in \dom$ and $t>0$
\begin{align*}
\prob \left\{ \ \left| \lnrf^{(h)} (\omega,\vec{x}) \right| >t  \right\}
&\le  \exp\left(-\frac{t^2}{2\text{Cov}_{\lnrf^{(h)}}(\vec{x},\vec{x})} \right)\le  \exp\left(-\frac{t^2}{8(\constC^{(h)})^2|\dom| } \right),
\end{align*}
which is uniform for $\vec{x}\in \dom$, see \eqref{li2neu2}. Consequently, we have, for all $\vec{v} \in \R^{\dimVh} $  with  $\| \vec{v}\|_2=1$, the estimate
\begin{align*}
	&\prob\left\{  \left| \vec{v} \cdot\left( \vec{\lnrfrv}^{(h)}-\expect{\vec{\lnrfrv}^{(h)}}\right) \right|  > t\right\} \le \exp\left(-\frac{t^2}{{8  (\constC^{(h)}})^2 |\dom| } \right).
\end{align*}
Note that this implies $4(\constC^{(h)})^2 |\dom| 
 =\rho^{-1} $ in the sense of \cite[Definition 1]{cai2016}.
Hence, the random vector $\vec{\lnrfrv}^{(h)}$ obeys a sub-Gaussian bound in the sense of \cite[Eq. (7)]{cai2010}.
We will make use of this estimate together with the results of \cite{cai2010} to bound the variance  of the sampled covariance matrix later on.

\subsection{Covariance estimation from samples using tapering and decay assumptions}\label{subsec:approxcovhMfinal}
In this subsection, we focus on the sampling error and specifically on the induced variance of our estimator for the covariance matrix. We assume in this subsection the parameter $h$ to be fixed. Thus the dimension of the covariance matrix is given as $\dimVh$. We will provide a constructive approach to obtain an estimation to the covariance matrix  $\vec{\Sigma}_{\vec{\lnrfrv}^{(h)}}$. To this end, with given samples $\vec{\lnrfrv}^{(\nsamplesrun;h)}$ for $\nsamplesrun=1,\cdots,M$, let the sample mean and the maximum likelihood estimator for the covariance matrix $\vec{\Sigma}_{\vec{\lnrfrv}^{(h)}}$ be \cite[Eq. (2)]{cai2010}
\begin{align}\label{astest}
\bar{\vec{\lnrfrv}}^{(h;\nsamples)}&:=\frac{1}{\nsamples}\sum_{\nsamplesrun=1}^{\nsamples}\vec{\lnrfrv}^{(\nsamplesrun;h)} \in \R^{\dimVh}\quad \text{and} \\
\label{astest2}
\bar{\vec{\Sigma}}^{(h;\nsamples)}_{\vec{\lnrfrv}^{(h)}}&:= \frac{1}{\nsamples} \sum_{\nsamplesrun=1}^{\nsamples} \left(\vec{\lnrfrv}^{(\nsamplesrun;h)}-\bar{\vec{\lnrfrv}}^{(h)}\right)\otimes \left(\vec{\lnrfrv}^{(\nsamplesrun;h)}-\bar{\vec{\lnrfrv}}^{(h)}\right)\in \R^{\dimVh \times \dimVh}.
\end{align}
We know from \cite[Lemma 1]{cai2016} that the  direct covariance estimator \eqref{astest2} suffers from the curse of dimension with respect to high spatial resolution $\dimVh$. Consequently, we need a better estimator when $\dimVh$ is large. To this end, we follow \cite{cai2016} and assume an off diagonal decay of the covariance matrix $\vec{\Sigma}_{\vec{\lnrfrv}^{(h)}} \in \R^{\dimVh \times \dimVh}$, i.e., we assume  $\vec{\Sigma}_{\vec{\lnrfrv}^{(h)}} \in \decayclass_{\alpha}=\decayclass_{\alpha}(C_{\decayclass;1},C_{\decayclass;2})$, where
\begin{align}\label{taperingclass}
\decayclass_{\alpha}:=\left\{\vec{\Sigma} \in \R^{\dimVh \times \dimVh} \ : \ \max_{\dimVhrun=1,\dots,\dimVh}\sum_{\genfrac{}{}{0pt}{}{\dimVhrun^{\prime}=1}{\left|\dimVhrun^{\prime}-\dimVhrun \right|>c }}^{\dimVh}\left|\vec{\Sigma}_{\dimVhrun,\dimVhrun^{\prime}}\right| \le C_{\decayclass;1} c^{-\alpha}\text{ for all } 1\le c \le \dimVh \text{  and  } \lambda_{\max}(\vec{\Sigma})\le C_{\decayclass;2}\right\}.
\end{align}
Here, $\lambda_{\max}(\vec{\Sigma})$ is the maximum eigenvalue of $\vec{\Sigma}$, and $\alpha$ modulates the speed of decay while $C_{\decayclass;1}$ and $C_{\decayclass;2}$ are positive parameters.
Now, let  weights for any even integer $1\le \taperingint \le \dimVh$ (see \cite[(Eq. 5)]{cai2010}) be given as
\begin{align}\label{taperingweights}
\taperingweight_{\dimVhrun,\dimVhrun^{\prime}}:=\taperingweight_{\dimVhrun,\dimVhrun^{\prime}}(\tau):=
\begin{cases}1, & \left|\dimVhrun-\dimVhrun^{\prime} \right|\le \frac{\taperingint}{2}, \\
2\left(1-\frac{\left|\dimVhrun-\dimVhrun^{\prime} \right|}{\taperingint} \right), &  \frac{\taperingint}{2} < \left|\dimVhrun-\dimVhrun^{\prime} \right|<\taperingint ,\\
0, & \left|\dimVhrun-\dimVhrun^{\prime} \right|\ge \taperingint.
 \end{cases}
\end{align}
Then the associated tapering estimator (c.f. \cite[(Eq. 4)]{cai2010}) is defined element-wise by
\begin{align}
\label{taperingestimator}
\left(\vec{\Sigma}^{(h;\nsamples;\tau)}_{\vec{\lnrfrv}^{(h)}}\right)_{_{\dimVhrun,\dimVhrun^{\prime}}}:=\taperingweight_{\dimVhrun,\dimVhrun^{\prime}}(\tau)\left(\bar{\vec{\Sigma}}^{h;\nsamples}_{\vec{\lnrfrv}^{(h)}}\right)_{\dimVhrun,\dimVhrun^{\prime}}, \quad 1 \le \dimVhrun,\dimVhrun^{\prime} \le \dimVh.
\end{align}
Note that the estimator $\vec{\Sigma}^{(h;\nsamples;\tau)}_{\vec{\lnrfrv}^{(h)}}$ is self-adjoint. Moreover the bounds
\begin{align}
\expect{\left\|\vec{\Sigma}^{(h;\nsamples;\tau)}_{\vec{\lnrfrv}^{(h)}}-\expect{\vec{\Sigma}^{(h;\nsamples;\tau)}_{\vec{\lnrfrv}^{(h)}}} \right\|_{2 \to 2}^2}&\lesssim \frac{\taperingint+\log(\dimVh)}{\nsamples}\label{taperingerrorvar},\\
\left\|\expect{\vec{\Sigma}^{(h;\nsamples;\tau)}_{\vec{\lnrfrv}^{(h)}}}-\vec{\Sigma}_{\vec{\lnrfrv}^{(h)}} \right\|_{2\to 2}^2&\lesssim \taperingint^{-2\alpha} \label{taperingerrorbias}
\end{align}
on the variance and the bias hold, see \cite[(Eqs. 13 \& 14)]{cai2010}.
By the triangle inequality we deduce the convergence of $\vec{\Sigma}^{(h;\nsamples;\tau)}_{\vec{\lnrfrv}^{(h)}}\to \vec{\Sigma}_{\vec{\lnrfrv}^{(h)}}$ in expectation as
\begin{align}\label{taperingerrorbiasplusvar}
\expect{\left\|\vec{\Sigma}^{(h;\nsamples;\tau)}_{\vec{\lnrfrv}^{(h)}}-\vec{\Sigma}_{\vec{\lnrfrv}^{(h)}}\right\|_{2\to 2}^2} &\le 2 \left(\left\|\expect{\vec{\Sigma}^{(h;\nsamples;\tau)}_{\vec{\lnrfrv}^{(h)}}}-\vec{\Sigma}_{\vec{\lnrfrv}^{(h)}} \right\|_{2\to 2}^2 +\expect{\left\|\vec{\Sigma}^{(h;\nsamples;\tau)}_{\vec{\lnrfrv}^{(h)}}-\expect{\vec{\Sigma}^{(h;\nsamples;\tau)}_{\vec{\lnrfrv}^{(h)}}} \right\|_{2\to 2}^2}\right)\nonumber \\&\lesssim \frac{\taperingint+\log(\dimVh)}{\nsamples}+\taperingint^{-2\alpha},
\end{align}
which motivates the choice of the tapering parameter $\tau=\nsamples^{\frac{1}{2\alpha+1}}$, c.f. \cite[(Eq. 15)]{cai2010}.
This gives the following result, c.f. \cite[Theorem 2 \& (Eq. 31)]{cai2010}.

\begin{proposition}\label{propsamplingexpec}
For the tapering estimator \eqref{taperingestimator} with $\tau=\nsamples^{\frac{1}{2\alpha+1}}$, i.e.,
\begin{align}\label{optimaltaperingchoice}
	\left(\vec{\Sigma}^{(h;\nsamples)}_{\vec{\lnrfrv}^{(h)}}\right)_{_{\dimVhrun,\dimVhrun^{\prime}}}:=\taperingweight_{\dimVhrun,\dimVhrun^{\prime}}(\nsamples^{\frac{1}{2\alpha+1}} )\left(\bar{\vec{\Sigma}}^{(h;\nsamples)}_{\vec{\lnrfrv}^{(h)}}\right)_{\dimVhrun,\dimVhrun^{\prime}}, \quad 1 \le \dimVhrun,\dimVhrun^{\prime} \le \dimVh,
\end{align}
there holds the error estimate 
\begin{align}\label{taperingerrorbiasplusvaroptimal}
\expect{\left\|\vec{\Sigma}^{(h;\nsamples)}_{\vec{\lnrfrv}^{(h)}}-\vec{\Sigma}_{\vec{\lnrfrv}^{(h)}}\right\|_{2\to 2}^2} \lesssim \nsamples^{-\frac{2\alpha}{2\alpha+1}} + \frac{\log(\dimVh)}{\nsamples}
\end{align}
if the condition $\dimVh \ge  \nsamples^{\frac{1}{2\alpha+1}}$ is satisfied.
In the case  $\dimVh < \nsamples^{\frac{1}{2\alpha+1}}$, we can directly use the estimator from \eqref{astest2}
\begin{align*}
\bar{\vec{\Sigma}}^{(h;\nsamples)}_{\vec{\lnrfrv}^{(h)}}&:= \frac{1}{\nsamples} \sum_{\nsamplesrun=1}^{\nsamples} \left(\vec{\lnrfrv}^{(\nsamplesrun;h)}-\bar{\vec{\lnrfrv}}^{(h)}\right)\otimes \left(\vec{\lnrfrv}^{(\nsamplesrun;h)}-\bar{\vec{\lnrfrv}}^{(h)}\right)\in \R^{\dimVh \times \dimVh}
\end{align*}
and obtain the bound
\begin{align}\label{taperingerrorbiasplusvaroptimal2}
\expect{\left\|\bar{\vec{\Sigma}}^{(h;\nsamples)}_{\vec{\lnrfrv}^{(h)}}-\vec{\Sigma}_{\vec{\lnrfrv}^{(h)}}\right\|_{2\to 2}^2} \lesssim \frac{\dimVh }{\nsamples}.
\end{align}
\end{proposition}
\noindent Note that the result is in fact \emph{rate optimal}.

We introduce the notation
\begin{align}\label{rhoM}
\rho_{h}(\nsamples)&:=\begin{cases}\nsamples^{-\frac{2\alpha}{2\alpha+1}} + \frac{\log(\dimVh)}{\nsamples}, &  \dimVh \ge \nsamples^{\frac{1}{2\alpha+1}} \\  \frac{\dimVh }{\nsamples}, & \dimVh < \nsamples^{\frac{1}{2\alpha+1}} \end{cases} \\ \label{tilderhoM}& \lesssim \tilde{\rho}_{h}(\nsamples):=\begin{cases}\nsamples^{-\frac{2\alpha}{2\alpha+1}} + \frac{d\log(h^{-1})}{\nsamples}, &  \dimVh \ge \nsamples^{\frac{1}{2\alpha+1}} \\  \frac{h^{-d} }{\nsamples}, & \dimVh < \nsamples^{\frac{1}{2\alpha+1}} \end{cases} ,
\end{align}
using $\dimVh\sim s^d h^{-d}$ for the definition of $\tilde{\rho}_{h}(\nsamples)$.
Then, we obtain the following corollary.
\begin{corollary}\label{corboundtildeS}
For the tapering estimator \eqref{taperingestimator} with $\tau=\nsamples^{\frac{1}{2\alpha+1}}$ the bounds
\begin{align}\label{errorexpec}
\expect{ \left\|\tildstiffnessmatrix-\tildestiffnessmatrixM \right\|_{2 \to 2}} &\lesssim \rho^{\frac{1}{2}}_{h}(\nsamples) ~ \lambda_{\max}\left(\massmatrix\right),\\
\label{errorexpec2}
\expect{ \left\|\tildstiffnessmatrix-\tildestiffnessmatrixM \right\|^2_{2 \to 2}}& = \expect{ \left\| \left(\vec{\Sigma}^{(h;\nsamples)}_{\vec{\lnrfrv}^{(h)}}-\vec{\Sigma}_{\vec{\lnrfrv}^{(h)}} \right)  \right\|^2_{2 \to 2}}\lambda_{\max}^2 \left(\massmatrix\right) \lesssim   \rho_{h}(\nsamples)~\lambda_{\max}^2\left(\massmatrix\right) 
\end{align}
hold.
\end{corollary}
\begin{proof}
	First, we obtain by \eqref{lieqneu} 
\begin{align}\label{eq:zzz1}
 \left\|\tildstiffnessmatrix-\tildestiffnessmatrixM \right\|_{2 \to 2}  \le \lambda_{\max}\left(\massmatrix\right) 
	\left\| \vec{\Sigma}^{(h;\nsamples)}_{\vec{\lnrfrv}^{(h)}}-\vec{\Sigma}_{\vec{\lnrfrv}^{(h)}}  \right\|_{2 \to 2} .
\end{align}
We can apply Jensen's inequality for the convex function $\psi(x)=x^2$,
\begin{align*}
\psi\left( \expect{X}\right) \le \expect{\psi(X)}.
\end{align*}
Plugging \eqref{eq:zzz1} into Jensen's inequality  yields
\begin{align*}
\left(\expect{\left\|\tildstiffnessmatrix-\tildestiffnessmatrixM \right\|_{2 \to 2} }\right)^2 \le 
\expect{ \left\| \left(\vec{\Sigma}^{(h;\nsamples)}_{\vec{\lnrfrv}^{(h)}}-\vec{\Sigma}_{\vec{\lnrfrv}^{(h)}} \right)  \right\|^2_{2 \to 2} } \lambda_{\max}^2 \left(\massmatrix\right) .
\end{align*}
Thus, we get from Proposition \ref{propsamplingexpec} and with the definition \eqref{rhoM} the bound
\begin{align*}
\expect{ \left\|\tildstiffnessmatrix-\tildestiffnessmatrixM \right\|_{2 \to 2}} \lesssim \rho^{\frac{1}{2}}_{h}(\nsamples)~\lambda_{\max}\left(\massmatrix\right) .
\end{align*}
Moreover, we have 
\begin{align*}
\expect{ \left\|\tildstiffnessmatrix-\tildestiffnessmatrixM \right\|^2_{2 \to 2}} \le \lambda_{\max}^2 \left(\massmatrix\right)  \expect{ \left\| \left(\vec{\Sigma}^{(h;\nsamples)}_{\vec{\lnrfrv}^{(h)}}-\vec{\Sigma}_{\vec{\lnrfrv}^{(h)}} \right)  \right\|^2_{2 \to 2}} \lesssim   \rho_{h}(\nsamples)~\lambda_{\max}^2 \left(\massmatrix\right) .
\end{align*}
\end{proof}
Next, we need to verify that Assumption \eqref{condspectralgap1}, i.e.,
\begin{align*}
\delta_{\kllevelrun}\ge 4C_1h^{2s} \lambda_{\kllevelrun+1}^{-1} +4\left\|\tildstiffnessmatrix-\tildestiffnessmatrixM \right\|_{2 \to 2} \quad \text{for all } 1\le \kllevelrun \le \kllevel,
\end{align*}
holds at least with high probability. To this end, we have the following result:
\begin{proposition}\label{prop:spectralgap}
Let the mesh size $h$ be sufficiently small and $\nsamples$ be sufficiently large such that, for all $1\le \kllevelrun \le \kllevel$,
\begin{align}\label{condspectralgapfinal1}
  \frac{\delta_{\kllevelrun}}{2}\ge 4C_1 h^{2s} \lambda_{\kllevel+1}^{-1}  \quad \text{and} \quad
  C \nsamples^{-\frac{2\alpha}{2\alpha+1}} \le \frac{\min_{1\le \kllevelrun \le \kllevel }\delta_{\kllevelrun}}{16 ~ \lambda_{\max}\left(\massmatrix\right) }
\end{align}
holds.
Then we have with probability
\begin{align}\label{eq:ppp}
	p_0:=1- 2 \dimVh 5^{\taperingint}\exp\left(- \nsamples  \rho_1  \left( \frac{\min_{1\le \kllevelrun \le \kllevel }\delta_{\kllevelrun}}{48~\lambda_{\max}\left(\massmatrix\right) } \right)^2\right) \quad \text{where} \quad \tau = \nsamples^{\frac{1}{2\alpha+1}}
\end{align}
that the condition \eqref{condspectralgap1} in our spectral gap assumption \ref{ass:spectralgap} is satisfied.
\end{proposition}
\begin{proof}
Recall \eqref{lieqneu}, i.e.,
\begin{align*}
 \left\|\tildstiffnessmatrix-\tildestiffnessmatrixM \right\|_{2 \to 2} \le \lambda_{\max}\left(\massmatrix\right) 
	\left\| \vec{\Sigma}^{(h;\nsamples)}_{\vec{\lnrfrv}^{(h)}}-\vec{\Sigma}_{\vec{\lnrfrv}^{(h)}}  \right\|_{2 \to 2} .
\end{align*}
Thus, it is sufficient to ensure 
\begin{align*}
\delta_{\kllevelrun}\ge 4C_1h^{2s} \lambda_{\kllevelrun+1}^{-1} +4 \left\| \vec{\Sigma}^{(h;\nsamples)}_{\vec{\lnrfrv}^{(h)}}-\vec{\Sigma}_{\vec{\lnrfrv}^{(h)}}   \right\|_{2 \to 2} \lambda_{\max}\left(\massmatrix\right)   \quad \text{for all } 1\le \kllevelrun \le \kllevel.
\end{align*}
Condition \eqref{condspectralgapfinal1} reduces this to 
\begin{align*}
\frac{\min_{1\le \kllevelrun \le \kllevel }\delta_{\kllevelrun}}{8~ \lambda_{\max}\left(\massmatrix\right)  }\ge   \left\| \vec{\Sigma}^{(h;\nsamples)}_{\vec{\lnrfrv}^{(h)}}-\vec{\Sigma}_{\vec{\lnrfrv}^{(h)}}   \right\|_{2 \to 2}.
\end{align*}
We aim to show that this inequality is only violated with small probability.
To this end, we derive by 
repeating the arguments of \cite[Lemmata 2 \& 3]{cai2010} and \cite[Remark 1]{cai2010} the bound
\begin{align}\label{rho1def}
\prob\left\{ \left\|\vec{\Sigma}^{(h;\nsamples)}_{\vec{\lnrfrv}^{(h)}}-\expect{\vec{\Sigma}^{(h;\nsamples)}_{\vec{\lnrfrv}^{(h)}}}\right\|_{2\to 2}>t \right\} \le 2 \dimVh 5^{\taperingint}\exp\left(-\frac{1}{9}\nsamples t^2 \rho_1 \right) \quad \text{for all } 0< t<\rho_1,
\end{align}
where $\rho_1$ is a constant, see \cite[page 2142]{cai2010}, and $\tau = \nsamples^{\frac{1}{2\alpha+1}}$.
Thus, we obtain with \eqref{taperingerrorbias} and by setting $\tau=\nsamples^{\frac{1}{2\alpha+1}}$ the bound
\begin{align}\label{covboundinprobab}
\prob\left\{ \left\|\vec{\Sigma}^{(h;\nsamples)}_{\vec{\lnrfrv}^{(h)}}-\vec{\Sigma}_{\vec{\lnrfrv}^{(h)}} \right\|_{2\to 2} >t \right\}\nonumber &\le \prob\left\{ \left\|\vec{\Sigma}^{(h;\nsamples)}_{\vec{\lnrfrv}^{(h)}}-\expect{\vec{\Sigma}^{(h;\nsamples)}_{\vec{\lnrfrv}^{(h)}}}\right\|_{2\to 2}+\left\|\expect{\vec{\Sigma}^{(h;\nsamples)}_{\vec{\lnrfrv}^{(h)}}}-\vec{\Sigma}_{\vec{\lnrfrv}^{(h)}} \right\|_{2\to 2} >t \right\}\nonumber \\
&\le \prob\left\{
\left\|\vec{\Sigma}^{(h;\nsamples)}_{\vec{\lnrfrv}^{(h)}}-\expect{\vec{\Sigma}^{(h;\nsamples)}_{\vec{\lnrfrv}^{(h)}}}\right\|_{2\to 2}+C \nsamples^{-\frac{2\alpha}{2\alpha+1}} >t \right\}. 
\end{align}
Now, let $\bar{t} \in (0,\rho_1) $ be fixed and let $\nsamples$ be large enough such that $C \nsamples^{-\frac{2\alpha}{2\alpha+1}} \le \frac{\bar{t}}{2}$. Then, it holds that

\begin{align}\label{covboundinprobab2}
\prob\left\{ \left\|\vec{\Sigma}^{(h;\nsamples)}_{\vec{\lnrfrv}^{(h)}}-\vec{\Sigma}_{\vec{\lnrfrv}^{(h)}} \right\|_{2\to 2} >\bar{t} \right\}&\le 
\prob\left\{
\left\|\vec{\Sigma}^{(h;\nsamples)}_{\vec{\lnrfrv}^{(h)}}-\expect{\vec{\Sigma}^{(h;\nsamples)}_{\vec{\lnrfrv}^{(h)}}}\right\|_{2\to 2} > \frac{\bar{t}}{2} \right\} \le 2 \dimVh 5^{\taperingint}\exp\left(-\frac{1}{36}\nsamples \bar{t}^2 \rho_1 \right). 
\end{align}
Now, we set $\bar{t}:=\frac{1}{8} \min_{1\le \kllevelrun \le \kllevel }\delta_{\kllevelrun}$  and observe that the second inequality in condition \eqref{condspectralgapfinal1} ensures 
$C \nsamples^{-\frac{2\alpha}{2\alpha+1}} \le\frac{1}{2}\bar{t} =\frac{1}{16~ \lambda_{\max}\left(\massmatrix\right)  }\min_{1\le  \kllevelrun \le \kllevel }\delta_{\kllevelrun}$. 
Thus, we get 
\begin{align}\label{covboundinprobab3}
\prob\left\{ \left\|\vec{\Sigma}^{(h;\nsamples)}_{\vec{\lnrfrv}^{(h)}}-\vec{\Sigma}_{\vec{\lnrfrv}^{(h)}} \right\|_{2\to 2} > \frac{\min_{1\le \kllevelrun \le \kllevel }\delta_{\kllevelrun}}{8 }\right\}&\le 
\prob\left\{
\left\|\vec{\Sigma}^{(h;\nsamples)}_{\vec{\lnrfrv}^{(h)}}-\expect{\vec{\Sigma}^{(h;\nsamples)}_{\vec{\lnrfrv}^{(h)}}}\right\|_{2\to 2} > \frac{\bar{t}}{2} \right\} \nonumber \\ &
\le 2 \dimVh 5^{\taperingint}\exp\left(- \nsamples  \rho_1  {\left( \frac{\min_{1\le \kllevelrun \le \kllevel }\delta_{\kllevelrun}}{48 ~ \lambda_{\max}\left(\massmatrix\right)  } \right)^{2}}\right),
\end{align}
using $36*8^2=9*2^8$ and thus $\sqrt{36*8^2}=3*2^4=48$.
This proves the claim.
\end{proof}

This establishes a further central ingredient for our later analysis, namely the optimal bound for the sampling error.

\section{Error analysis for covariance operator reconstruction}\label{sec:covoprec}
We are now in the position to derive an overall  error bound for the reconstruction of the continuous covariance operator from finite samples. To this end recall the 
covariance kernel  $R$ in \eqref{carlemannkern}, i.e.,
\begin{align*}
R(\vec{x},\vec{x}^{\prime})=\int_{\Omega}\lnrf (\vec{\om},\vec{x}) \lnrf (\vec {\om}^{\prime},\vec{x}^{\prime})\, \mathrm{d}\prob = \sum_{\kllevelrun=1}^{\infty} \lambda_{\kllevelrun} \phi_{\kllevelrun}(\vec{x}) \phi_{\kllevelrun}(\vec{x}^{\prime}).
\end{align*}
Furthermore let $\phi_{\kllevelrun}^{(h;\nsamples)}$ be given as in \eqref{Sheigenfunc2} and let
$\kllevel\in\mathbb{N}$ be a truncation parameter with $1\le \kllevel \le \dimVh$. We define
\begin{align}\label{mercerh2}
R^{(\kllevel;h;\nsamples)}(\vec{x},\vec{x}^{\prime}) :=\sum_{\kllevelrun=1}^{\kllevel}\lambda^{(h;\nsamples)}_{\kllevelrun}
\phi_{\kllevelrun}^{(h;\nsamples)}(\vec{x})\phi_{\kllevelrun}^{(h;\nsamples)}(\vec{x}^{\prime}).
\end{align}
We are now interested in the approximation error $R-R^{(\kllevel;h;\nsamples)}$. 
To this end, let us define 
\begin{align}\label{deltagl}
G(\kllevel):=\left(\sum_{\kllevelrun=1}^{\kllevel}\left(\frac{ \lambda_{\kllevelrun}} {\delta_{\kllevelrun}} \right)^{2} \right)^{\frac{1}{2}},
\end{align}
which is a function of the truncation parameter $\kllevel$ that is determined by spectral properties of the unknown covariance kernel  $R$  only. \\
However, in special cases, it is possible to compute $G(\kllevel)$ explicitly.
As an example, let us consider  Brownian motion.  In the univariate case we have the kernel
\begin{align}\label{univariateBrownianMotion}
	R^{(1)}_{B}:(0,1)^2 \to \R, \quad R^{(1)}_{B}(x,x^{\prime}):=\min\{x,x^{\prime}\}.
\end{align}
Then it is known that 
\begin{align*}
	\lambda_{R^{(1)}_{B}}(\kllevelrun)= \pi^{-2}\left(\kllevelrun -\frac{1}{2} \right)^{-2}= \left(\frac{2}{\pi}\right)^2 \left(2\ell-1 \right)^{-2} \in \left(\frac{2}{\pi}\right)^2 (2\N-1)^{-2}.
\end{align*}
Thus, we have
\begin{align*}
\delta_{R^{(1)}_{B}}(\kllevelrun):= \frac{2\kllevelrun}{\pi^2(\kllevelrun^2-\frac{1}{4})^2}=\frac{1}{2}\left(\frac{2}{\pi}\right)^2 \frac{\kllevelrun}{\left(\kllevelrun^2 -\frac{1}{4}\right)^{2}}
\end{align*} 
and consequently we obtain
\begin{align*}
	\frac{ \lambda_{R^{(1)}_{B}}(\kllevelrun)}{\delta_{R^{(1)}_{B}}(\kllevelrun)} = \frac{2}{\kllevelrun} \left( \frac{ \kllevelrun^2 -\frac{1}{4}}{2\kllevelrun-1}\right)^{2}
	= 2 \frac{\left(2\kllevelrun+1 \right)^2}{16\kllevelrun}=\frac{\left(2\kllevelrun+1 \right)^2}{8\kllevelrun}.
\end{align*}
Finally, we get
\begin{align*}
	G^2_{R^{(1)}_{B}}(\kllevel)=\sum_{\kllevelrun=1}^{\kllevel}\left(\frac{ \lambda_{R^{(1)}_{B}}(\kllevelrun)} {\delta_{R^{(1)}_{B}}(\kllevelrun)} \right)^{2} = \frac{1}{64}\sum_{\kllevelrun=1}^{\kllevel} \frac{(2\kllevelrun+1)^4}{\kllevelrun^2}\approx \kllevel^{3} \quad \text{for } \kllevel \text{ large.}
\end{align*}
Moreover, for a two-dimensional Brownian motion, we have the eigenvalues
\begin{align*}
\lambda_{R^{(2)}_{B}}(\ell_1,\ell_2):= \lambda_{R^{(1)}_{B}}(\kllevelrun_{1}) \lambda_{R^{(1)}_{B}}(\kllevelrun_{2})=\left(\frac{2}{\pi} \right)^{4} (2\ell_1-1)^{-2}(2\ell_2-1)^{-2}, \quad \ell_1,\ell_2 \in \N.
\end{align*}
Hence, we observe 
\begin{align*}
\text{image}(\lambda_{R^{(2)}_{B}})=\lambda_{R^{(2)}_{B}}(\N \times \N) = \left(\frac{2}{\pi} \right)^{4} (2\N-1)^{-2},
\end{align*}
i.e., the  image of $\lambda_{R^{(2)}_{B}}$ consists up to the pre-factor $\left(\frac{2}{\pi} \right)^{4}$ again of  the squared inverses of the  odd numbers. 
Hence, we can sort the eigenvalues (ignoring multiplicities)  by the ordering in the odd natural numbers. Moreover, it holds that
\begin{align}
\lambda_{R^{(2)}_{B}}(\ell_1,\ell_2)=\lambda_{R^{(1)}_{B}}(1) \lambda_{R^{(1)}_{B}}\left(k(\ell_1,\ell_2)\right), \quad \text{where } k(\ell_1,\ell_2):=\frac{(2\ell_1-1)(2\ell_2-1)+1}{2} .
\end{align}
In the general $d$-dimensional situation, we have 
\begin{align*}
	\lambda_{R^{(d)}_{B}}(\ell_1,\dots,\ell_d)=\prod_{j=1}^{d} \lambda_{R^{(d)}_{B}}(\ell_{j})=\left(\lambda_{R^{(1)}_{B}}(1) \right)^{d-1} \lambda^{(1)}(k(\ell_1,\dots,\ell_{d})),
\end{align*}
where 
\begin{align*}
k(\ell_1,\dots,\ell_{d}):=\frac{1}{2} \left(\prod_{j=1}^{d} (2\ell_j-1) +1\right).
\end{align*}
Furthermore, we have with $\ell:=k(\ell_1,\dots,\ell_{d})$
\begin{align*}
\delta_{R^{(d)}_{B}}(\kllevelrun)=\left(\lambda_{R^{(1)}_{B}}(1) \right)^{d-1}\delta_{R^{(1)}_{B}}(\kllevelrun).
\end{align*}
Thus, for the $d$-dimensional Brownian motion, we get
\begin{align}\label{eq:gl:asymptotic}
G^2_{R^{(d)}_{B}}(\kllevel)&=\sum_{\kllevelrun=1}^{\kllevel}\left(\frac{ \lambda_{R^{(d)}_{B}}(\kllevelrun)} {\delta_{R^{(d)}_{B}}(\kllevelrun)} \right)^{2} = \sum_{\kllevelrun=1}^{\kllevel} \left( \frac{\ \lambda_{R^{(1)}_{B}}(\kllevelrun)}{\delta_{R^{(1)}_{B}}(\kllevelrun)} \right)^{2}=G^2_{R^{(1)}_{B}}(\kllevel)=\frac{1}{64}\sum_{\kllevelrun=1}^{\kllevel} \frac{(2\kllevelrun+1)^4}{\kllevelrun^2}\approx \kllevel^{3} \quad \text{for } \kllevel \text{ large.}
\end{align}
Note here that $G_{R^{d}_{B}}(\kllevel)$ is now completely independent of the dimension.

In the general situation, we have the following result: 
\begin{theorem}\label{finalthm}
Let $\dimVh \ge \kllevel$.
Under the condition of Proposition \ref{prop:spectralgap}, i.e., \eqref{condspectralgapfinal1} and the assumption $C_2 h^{s}  \lambda_{\kllevel}^{-1} \le  1$,  there holds with probability $p_0$ given in \eqref{eq:ppp} that 
\begin{align}\label{eq:final-1}
	\norm{{R}-R^{(\kllevel;h;\nsamples)}}_{L^{2}(\dom\times\dom)}&\lesssim  \kllevel^{-\frac{2s}{d}-\frac{1}{2}}  + \kllevel^{\frac{1}{2}} h^{s}  + \left(   \kllevel^{\frac{1}{2}} + G(\kllevel)\right) \left\|\tildstiffnessmatrix-\tildestiffnessmatrixM \right\|_{2\to 2} 
\end{align}
with $G(\kllevel)$ from \eqref{deltagl}. If we use the assumption $C_2 h^{s}  \lambda_{\kllevel}^{-1} \le  1$ again, we obtain the simplified bound  
\begin{align}\label{eq:final-2}
	\norm{{R}-R^{(\kllevel;h;\nsamples)}}_{L^{2}(\dom\times\dom)}&\lesssim \kllevel^{-\frac{2s}{d}-\frac{1}{2}}+ \left( \kllevel^{\frac{1}{2}} + G(\kllevel) \right) \left\|\tildstiffnessmatrix-\tildestiffnessmatrixM \right\|_{2\to 2}   .
\end{align}
\end{theorem}
\begin{proof}
Since \eqref{condspectralgapfinal1} holds, it is ensured that the condition \eqref{condspectralgap1} in our spectral gap assumption \ref{ass:spectralgap} is satisfied with  the  probability $p_0$ from \eqref{eq:ppp}.
Now recall the definition of $R$ in \eqref{carlemannkern}, of $R^{(\kllevel)}$ in \eqref{RL}, of $R^{(\kllevel;h)}$ in \eqref{mercerh} and of $R^{(\kllevel;h;\nsamples)}$ in \eqref{mercerh2}.
We will split the proof into three parts. Observe that we can decompose the approximation error as
\begin{align}\label{errordecomposition}
R(\vec{x},\vec{x}^{\prime})&-R^{(\kllevel;h;\nsamples)}(\vec{x},\vec{x}^{\prime}) = \\ &\underbrace{R(\vec{x},\vec{x}^{\prime})- R^{(\kllevel)}(\vec{x},\vec{x}^{\prime})}_{E_1}
+\underbrace{R^{(\kllevel)}(\vec{x},\vec{x}^{\prime})-R^{(\kllevel;h)}(\vec{x},\vec{x}^{\prime})}_{E_2}
+\underbrace{R^{(\kllevel;h)}(\vec{x},\vec{x}^{\prime})
-R^{(\kllevel;h;\nsamples)}(\vec{x},\vec{x}^{\prime})}_{E_3} \nonumber.
\end{align}
The first error term
\begin{align}\label{error1}
E_1(\vec{x},\vec{x}^{\prime}):=R(\vec{x},\vec{x}^{\prime})- R^{(\kllevel)}(\vec{x},\vec{x}^{\prime})=\sum_{\kllevelrun=\kllevel+1}^{\infty}\lambda_{\kllevelrun}\phi_{\kllevelrun}(\vec{x})\phi_{\kllevelrun}(\vec{x}^{\prime})
\end{align}
depends solely on the truncation parameter $\kllevel$.
The second error term
\begin{align}\label{error2}
E_2(\vec{x},\vec{x}^{\prime}):=R^{(\kllevel)}(\vec{x},\vec{x}^{\prime})-R^{(\kllevel;h)}(\vec{x},\vec{x}^{\prime})=\sum_{\kllevelrun=1}^{\kllevel} \lambda_{\kllevelrun} \phi_{\kllevelrun}(\vec{x})\phi_{\kllevelrun}(\vec{x}^{\prime})-\sum_{\kllevelrun=1}^{\kllevel} \lambda^{(h)}_{\kllevelrun} \phi^{(h)}_{\kllevelrun}(\vec{x})\phi^{(h)}_{\kllevelrun}(\vec{x}^{\prime})
\end{align}
depends on the spatial discretization $h$ and on the truncation parameter $L$.
The third error term
\begin{align}\label{error3}
E_3(\vec{x},\vec{x}^{\prime})&:=R^{(\kllevel;h)}(\vec{x},\vec{x}^{\prime})-R^{(\kllevel;h;\nsamples)}(\vec{x},\vec{x}^{\prime})\nonumber \\&=\sum_{\kllevelrun=1}^{\kllevel} \lambda^{(h)}_{\kllevelrun} \phi^{(h)}_{\kllevelrun}(\vec{x}) \phi^{(h)}_{\kllevelrun}(\vec{x}^{\prime})-\sum_{\kllevelrun=1}^{\kllevel} \lambda^{(h;\nsamples)}_{\kllevelrun} \phi^{(h;\nsamples)}_{\kllevelrun}(\vec{x})\phi^{(h;\nsamples)}_{\kllevelrun}(\vec{x}^{\prime})
\end{align}
depends on the sample size $\nsamples$, the spatial discretization $h$ and  on the truncation parameter $L$.

To derive a bound for \eqref{error1} we use the orthonormality of the eigenfunctions $\{\phi_{\ell}\}_{\ell=1}^{\infty}$ and obtain directly for the integral operator
\begin{align}\label{intoperror1}
\mathcal{E}_{1}:L^{2}(\dom) \to L^{2}(\dom), \quad v \mapsto \int_{\dom}E_{1}(\vec{x},\cdot) v(\vec{x}) \, \mathrm{d}\lebesgue(\vec{x}) 
\end{align}
the identity
\begin{align}\label{E1error}
\norm{\mathcal{E}_{1}}_{L^{2}(\dom)\to L^{2}(\dom) } = \norm{E_1}_{L^{2}(\dom \times \dom)}  = \left(\sum_{\kllevelrun=\kllevel+1}^{\infty}  \lambda_{\kllevelrun}^2 \right)^{\frac{1}{2}}.
\end{align}
Now, we can use inequality \eqref{thm:truncationErrorLambda} of Theorem \ref{thm:truncationError} and derive the bound
\begin{align*}
\left(\sum_{\kllevelrun=\kllevel+1}^{\infty}  \lambda_{\kllevelrun}^2 \right)^{\frac{1}{2}} \le C_{\ref{thm:truncationError}} 
\left(\sum_{\kllevelrun=\kllevel+1}^{\infty}  \kllevelrun^{-\frac{4s}{d}-2} \right)^{\frac{1}{2}} 
\le C_{\ref{thm:truncationError}} \left(\int_{\kllevelrun=\kllevel}^{\infty}  x^{-\frac{4s}{d}-2} \, \mathrm{d}x \right)^{\frac{1}{2}} \le   C_{\ref{thm:truncationError}} \left(\frac{4s}{d}+1 \right)^{-\frac{1}{2}} \kllevel^{-\frac{2s}{d}-\frac{1}{2}} .
\end{align*}

For the second term \eqref{error2}, we have
\begin{align*}
E_2(\vec{x},\vec{x}^{\prime})&:=\sum_{\kllevelrun=1}^{\kllevel} \lambda_{\kllevelrun} \phi_{\kllevelrun}(\vec{x})\phi_{\kllevelrun}(\vec{x}^{\prime})-
\sum_{\kllevelrun=1}^{\kllevel} \lambda^{(h)}_{\kllevelrun} \phi^{(h)}_{\kllevelrun}(\vec{x})\phi^{(h)}_{\kllevelrun}(\vec{x}^{\prime})\\
&= \sum_{\kllevelrun=1}^{\kllevel} \left( \lambda_{\kllevelrun} -\lambda^{(h)}_{\kllevelrun} \right)\phi_{\kllevelrun}(\vec{x})\phi_{\kllevelrun}(\vec{x}^{\prime})
+\sum_{\kllevelrun=1}^{\kllevel} \lambda^{(h)}_{\kllevelrun} \left(\phi_{\kllevelrun}(\vec{x})-\phi^{(h)}_{\kllevelrun}(\vec{x})\right)
\phi_{\kllevelrun}(\vec{x}^{\prime})\\
&+\sum_{\kllevelrun=1}^{\kllevel} \lambda^{(h)}_{\kllevelrun} \phi^{(h)}_{\kllevelrun}(\vec{x}) \left(\phi_{\kllevelrun}(\vec{x}^{\prime})-\phi^{(h)}_{\kllevelrun}(\vec{x}^{\prime})\right)
\\
&=:E_{2;1}(\vec{x},\vec{x}^{\prime})+E_{2;2}(\vec{x},\vec{x}^{\prime})+E_{2;3}(\vec{x},\vec{x}^{\prime}).
\end{align*}
The orthonormality of $\{\phi_{\ell}\}_{\ell=1}^{\infty}$ and an application of Proposition \ref{prop:babuska} leads to  
\begin{align}\label{eq:yyy1}
\normL{{E}_{2;1}}{\dom\times\dom}&\le \left(\sum_{\kllevelrun=1}^{\kllevel} |\lambda_{\kllevelrun} -\lambda^{(h)}_{\kllevelrun}|^2\right)^{1/2} \le   \left(\sum_{\kllevelrun=1}^{\kllevel} C^2_1 h^{4s}\lambda_{\kllevelrun}^{-2}\right)^{1/2}\nonumber 
\\&\le C_1 h^{s}  \left(\sum_{\kllevelrun=1}^{\kllevel} h^{2s}\lambda_{\kllevelrun}^{-2}\right)^{1/2} \le \frac{C_1}{C_2} \kllevel^{\frac{1}{2}}h^s,
\end{align}
where the last inequality holds since $C_2 h^{s}  \lambda_{\kllevel}^{-1} \le 1$. 
This bound implies 
\begin{align}\label{condhL}
C_2 h^{s} \le  \lambda_{\kllevel} \le \kllevel^{-\frac{2s}{d}-1}.
\end{align}
Now we turn to $E_{2;2}$. The orthonormality of $\{\phi_{\ell}\}_{\ell=1}^{\infty}$ and the bound $\lambda^{(h)}_{\kllevelrun} \le \lambda_{\kllevelrun}$ for $1\le \kllevelrun \le \kllevel$ due to the Courant min-max principle for decreasing eigenvalues implies
\begin{align*}
\normL{{E}_{2;2}}{\dom\times\dom}^2
\le\sum_{\kllevelrun=1}^{\kllevel}\left(\lambda^{(h)}_{\kllevelrun}\right)^{2}
\normL{\phi_{\kllevelrun}-\phi_{\kllevelrun}^{(h)}}{\dom}^2 \le \sum_{\kllevelrun=1}^{\kllevel}\lambda^{2}_{\kllevelrun}
\normL{\phi_{\kllevelrun}-\phi_{\kllevelrun}^{(h)}}{\dom}^2 \le C_2^{2} h^{2s} \kllevel,
\end{align*}
where we used  $\normL{\phi_{\kllevelrun}^{(h)}-\phi_{\kllevelrun}}{\dom}\leq C_{2}\lambda_{\kllevelrun}^{-1} h^{s}$  in the last step, see Proposition \ref{prop:babuska}.
Thus, we obtain
\begin{align}\label{eq:yyy2}
\normL{{E}_{2;2}}{\dom\times\dom} &\le C_{2} h^{s} \kllevel^{\frac{1}{2}}.
\end{align}

\noindent
For the term $E_{2,3}$ the orthonormality of the series $\{\phi_{\kllevelrun}^{(h)}\}_{\kllevelrun=1}^{\dimVh}$, which stems from the eigenvalue problem \eqref{genalizedeigenvalue}, leads to
\begin{align*}
\normL{{E}_{2;3}}{\dom\times\dom}^2
&=\sum_{\kllevelrun=1}^{\kllevel}\left(\lambda^{(h)}_{\kllevelrun}\right)^2 \normL{\phi_{\kllevelrun}-\phi_{\kllevelrun}^{(h)}}{\dom}^2.
\end{align*}
Thus, we obtain with verbatim the same computations as for \eqref{eq:yyy2}, the bound
\begin{align}\label{eq3}
\normL{{E}_{2;3}}{\dom\times\dom} &\le C_{2} h^{s} \kllevel^{\frac{1}{2}}.
\end{align}
Altogether with the estimates \eqref{eq:yyy1}, \eqref{eq:yyy2} and \eqref{eq3}, we derive
\begin{align}\label{E2error}
\normL{{E}_{2}}{\dom\times\dom}&=\normL{{E}_{2;1}+{E}_{2;2}+{E}_{2;3}}{\dom\times\dom}\le \normL{{E}_{2;1}}{\dom\times\dom}+ \normL{{E}_{2;2}}{\dom\times\dom}+\normL{{E}_{2;3}}{\dom\times\dom} \nonumber \\&\le \left( \frac{C_1}{C_2} + 2 C_2 \right) h^{s} \kllevel^{\frac{1}{2}} \le \left( \frac{C_1}{C_2} + 2 C_2 \right) C^{-1}_2\kllevel^{-\frac{2s}{d}-1} \kllevel^{\frac{1}{2}}=  \left( \frac{C_1}{C_2} + 2 C_2 \right) C^{-1}_2\kllevel^{-\frac{2s}{d}-\frac{1}{2}}.
\end{align}
For the third term $E_3$ in \eqref{error3}, we proceed in an analogous fashion as in the splitting of \eqref{error2}. We then have the decomposition
\begin{align*}
&E_3(\vec{x},\vec{x}^{\prime})=\sum_{\kllevelrun=1}^{\kllevel} \lambda^{(h)}_{\kllevelrun} \phi_{\kllevelrun}^{(h)}(\vec{x})\phi_{\kllevelrun}^{(h)}(\vec{x}^{\prime})
-\sum_{\kllevelrun=1}^{\kllevel} \lambda^{(h;\nsamples)}_{\kllevelrun} \phi^{(h;\nsamples)}_{\kllevelrun}(\vec{x})\phi^{(h;\nsamples)}_{\kllevelrun}(\vec{x}^{\prime})
\\
&= \sum_{\kllevelrun=1}^{\kllevel} \left( \lambda^{(h)}_{\kllevelrun} -\lambda^{(h;\nsamples)}_{\kllevelrun} \right)
\phi_{\kllevelrun}^{(h)}(\vec{x})\phi_{\kllevelrun}^{(h)}(\vec{x}^{\prime})
+\sum_{\kllevelrun=1}^{\kllevel} \lambda^{(h;\nsamples)}_{\kllevelrun} \left(\phi_{\kllevelrun}^{(h)}(\vec{x})-\phi^{(h;\nsamples)}_{\kllevelrun}(\vec{x})\right)\phi^{(h)}_{\kllevelrun}(\vec{x}^{\prime}) \\&+\sum_{\kllevelrun=1}^{\kllevel} \lambda^{(h;\nsamples)}_{\kllevelrun} \left(\phi_{\kllevelrun}^{(h)}(\vec{x}^{\prime}) -\phi_{\kllevelrun}^{(h;\nsamples)}(\vec{x}^{\prime})\right)
\phi^{(h;\nsamples)}_{\kllevelrun}(\vec{x}) \\
&=:E_{3;1}+E_{3;2}+E_{3;3}.
\end{align*}
Now, for the term $E_{3,1}$, the orthonormality of the basis $\{\phi_{\kllevelrun}^{(h)}\}_{\kllevelrun=1}^{\dimVh}$ and an application of \eqref{weyleigenvalue} implies
\begin{align}\label{eq:zzz1a}
\normL{{E}_{3;1}}{\dom}
&=\left( \sum_{\kllevelrun=1}^{\kllevel} \left( \lambda^{(h)}_{\kllevelrun} -\lambda^{(h;\nsamples)}_{\kllevelrun} \right)^2\right)^{\frac{1}{2}} \le \kllevel^{1/2} \left\|\tildstiffnessmatrix-\tildestiffnessmatrixM \right\|_{2 \to 2}.
\end{align}
\noindent
To derive an upper estimate for the term $E_{3,2}$, we will use our bound from Theorem \ref{theoremeigenfucntions} for the eigenfunction approximation in \eqref{phivecerror}. Thus we have with probability $p_0$ that
\begin{align}\label{phierror}
\normL{\phi^{(h)}_{\kllevelrun}-\phi^{(h;\nsamples)}_{\kllevelrun}}{\dom}^2\le 4C^2 \frac{\left\|\tildstiffnessmatrix-\tildestiffnessmatrixM \right\|^2_{2\to 2}}{\delta_{\kllevelrun}^2},
\end{align}
where we note at this point that this bound makes use of  \eqref{condspectralgap1} from the spectral gap assumption \ref{ass:spectralgap}.
 Due to the orthonormality of the basis $\{ \phi^{(h)}_{\kllevelrun}\}_{\kllevelrun=1}^{\dimVh}$, this yields
\begin{align*}
\normL{{E}_{3;2}}{\dom\times\dom}^2
&=\sum_{\kllevelrun=1}^{\kllevel} (\lambda^{(h;\nsamples)}_{\kllevelrun})^2 \normL{\phi^{(h)}_{\kllevelrun}-\phi^{(h;\nsamples)}_{\kllevelrun}}{\dom}^2 \le 4C^2  \sum_{\kllevelrun=1}^{\kllevel}\left(\frac{\left\|\tildstiffnessmatrix-\tildestiffnessmatrixM \right\|_{2\to 2} }{\delta_{\kllevelrun}} \right)^{2}\left(\lambda^{(h;\nsamples)}_{\kllevelrun}\right)^{2}.
\end{align*}
Hence, we obtain
\begin{align*}
\normL{{E}_{3;2}}{\dom\times\dom}
\le 2C\left\|\tildstiffnessmatrix-\tildestiffnessmatrixM \right\|_{2\to 2} \left(\sum_{\kllevelrun=1}^{\kllevel}\left( \frac{\lambda^{(h;\nsamples)}_{\kllevelrun}
}{\delta_{\kllevelrun}} \right)^{2}\right)^{\frac{1}{2}}.
\end{align*}
Moreover, we can use \eqref{weyleigenvalue}, i.e., 
\begin{align*}
\left| \lambda^{(h)}_{\kllevelrun} -\lambda^{(h;\nsamples)}_{\kllevelrun} \right| \le \left\|\tildstiffnessmatrix -\tildestiffnessmatrixM \right\|_{2 \to 2}, \quad 1\le  \kllevelrun \le \min\{\dimVh,\kllevel\}.
\end{align*}
to get
\begin{align}\label{lambdatriag2}
\left|\lambda^{(h;\nsamples)}_{\kllevelrun} \right|\le \left| \lambda^{(h)}_{\kllevelrun} -\lambda^{(h;\nsamples)}_{\kllevelrun} \right| +\left| \lambda^{(h)}_{\kllevelrun} \right| \le  \left\|\tildstiffnessmatrix -\tildestiffnessmatrixM \right\|_{2 \to 2} +\left| \lambda_{\kllevelrun} \right|,
\end{align}
where we used again $\lambda^{(h)}_{\kllevelrun} \le \lambda_{\kllevelrun}$ for $1\le \kllevelrun \le \kllevel$ in the last step.
This yields with probability $p_0$
\begin{align}\label{eq:zzz2}
\normL{{E}_{3;2}}{\dom\times\dom}
\le 2C  \left\|\tildstiffnessmatrix-\tildestiffnessmatrixM \right\|_{2\to 2} \left(\sum_{\kllevelrun=1}^{\kllevel}\left( \frac{ \left\|\tildstiffnessmatrix -\tildestiffnessmatrixM \right\|_{2 \to 2}  +  \left| \lambda_{\kllevelrun} \right|}{\delta_{\kllevelrun}} \right)^{2}\right)^{\frac{1}{2}}.
\end{align}
For the term $E_{3,3}$, we use orthonormality of the $\{\phi_{\kllevelrun}^{(h;\nsamples)}\}$ to get
\begin{align}\label{e32e33}
\normL{E_{3;3}}{\dom\times \dom}^2&=\sum_{\kllevelrun=1}^{\kllevel} (\lambda^{(h;\nsamples)}_{\kllevelrun})^2 \normL{\phi^{(h)}_{\kllevelrun}-\phi^{(h;\nsamples)}_{\kllevelrun}}{\dom}^2 =\normL{E_{3;2}}{\dom\times \dom}^2.
\end{align}
Thus, we again obtain  probability $p_0$
\begin{align}\label{E33error}
\normL{{E}_{3;3}}{\dom\times\dom}
\le 2C  \left\|\tildstiffnessmatrix-\tildestiffnessmatrixM \right\|_{2\to 2} \left(\sum_{\kllevelrun=1}^{\kllevel}\left( \frac{ \left\|\tildstiffnessmatrix -\tildestiffnessmatrixM \right\|_{2 \to 2}  +  \left| \lambda_{\kllevelrun} \right|}{\delta_{\kllevelrun}} \right)^{2}\right)^{\frac{1}{2}}.
\end{align}

Altogether, combining the inequalities \eqref{eq:zzz1a}, \eqref{eq:zzz2} and \eqref{E33error} we obtain probability $p_0$
\begin{align}\label{E3error1}
&\normL{{E}_{3}}{\dom\times\dom}=\normL{{E}_{3;1}+{E}_{3;2}+{E}_{3;3}}{\dom\times\dom} \le \normL{{E}_{3;1}}{\dom\times\dom}+\normL{{E}_{3;2}}{\dom\times\dom}+\normL{{E}_{3;3}}{\dom\times\dom} \nonumber \\&\le
 \left\|\tildstiffnessmatrix-\tildestiffnessmatrixM \right\|_{2\to 2} \left( \kllevel^{\frac{1}{2}}+4C \left(\sum_{\kllevelrun=1}^{\kllevel}\left( \frac{ \left\|\tildstiffnessmatrix -\tildestiffnessmatrixM \right\|_{2 \to 2}  +  \left| \lambda_{\kllevelrun} \right|}{\delta_{\kllevelrun}} \right)^{2}\right)^{\frac{1}{2}} \right).
\end{align}
Thus, we get with  \eqref{condspectralgap1} from the spectral gap assumption \ref{ass:spectralgap}\begin{align}\label{goodset}
\delta_{\kllevelrun}\ge 4C_1h^{2s} \lambda_{\kllevelrun+1}^{-1} +4\left\|\tildstiffnessmatrix-\tildestiffnessmatrixM \right\|_{2 \to 2} \ge 4\left\|\tildstiffnessmatrix-\tildestiffnessmatrixM \right\|_{2 \to 2},
\end{align}
which is satisfied  with probability $p_0$, that 
\begin{align}\label{goodsetestimate}
&\sum_{\kllevelrun=1}^{\kllevel}\left( \frac{ \left\|\tildstiffnessmatrix -\tildestiffnessmatrixM \right\|_{2 \to 2}  +  \left| \lambda_{\kllevelrun} \right|}{\delta_{\kllevelrun}} \right)^{2} \le 
2 \sum_{\kllevelrun=1}^{\kllevel}\left( \frac{ \left\|\tildstiffnessmatrix -\tildestiffnessmatrixM \right\|_{2 \to 2}} {\delta_{\kllevelrun}} \right)^{2} +2 \sum_{\kllevelrun=1}^{\kllevel} \left(\frac{ \left| \lambda_{\kllevelrun} \right|}{\delta_{\kllevelrun}} \right)^{2}
\nonumber \\&\le\frac{1}{8}  \kllevel +  2 \sum_{\kllevelrun=1}^{\kllevel} \left(\frac{ \lambda_{\kllevelrun} }{\delta_{\kllevelrun}} \right)^{2}
= \frac{1}{8}  \kllevel + 2 G^2(\kllevel).
\end{align}
Thus, we obtain, using $\sqrt{a+b} \le \sqrt{a}+\sqrt{b}$ for $a,b\ge 0$, that the bound \eqref{E3error1} reduces to
\begin{align}\label{E3error}
\normL{{E}_{3}}{\dom\times\dom} \le
 \left\|\tildstiffnessmatrix-\tildestiffnessmatrixM \right\|_{2\to 2} \left( \kllevel^{\frac{1}{2}}+\sqrt{2}C \kllevel^{\frac{1}{2}} + 4\sqrt{2} C G(\kllevel)\right),
\end{align}
which holds with  probability $p_0$.
Finally, combining the inequalities \eqref{E1error}, \eqref{E2error} and \eqref{E3error} we obtain 
\begin{align*}
\norm{R-R^{(\kllevel;h;\nsamples)}}_{\dom \times \dom} \le \norm{E_1 + E_2 + E_3}_{\dom \times \dom} \le \norm{E_1 }_{\dom \times \dom} +\norm{E_2 }_{\dom \times \dom} +\norm{E_3}_{\dom \times \dom}, 
\end{align*}
which shows the assertion.
\end{proof}
The bound in Theorem \ref{finalthm} still depends on $\left\|\tildstiffnessmatrix-\tildestiffnessmatrixM \right\|_{2 \to 2} $. This term however  involves the $\nsamples$-dependent sampling approximation \eqref{tildestiffnessample} and is thus a random variable. In the following, we therefore are interested in the expected value of the error. 
To this end, note that we can not use Corollary \ref{corboundtildeS} directly as we need the spectral gap assumption \ref{ass:spectralgap} to be valid. But this is only 
the case with probability $p_0$, which needs to be properly be taken care of.
Moreover, we need to introduce the quantity
\begin{align}\label{HL}
H(\kllevel):= {\left(\frac{1}{48}\min_{1\le \kllevelrun \le \kllevel }\delta_{\kllevelrun}\right)^{2}},
\end{align}
which is a function of the truncation parameter $\kllevel$ that depends on spectral properties as given in \eqref{defn:gap} of the true kernel $R$ . For general operators this quantity is not easy to obtain.


However, in the simple case of $d$-dimensional Brownian motion we can obtain a precise value as
\begin{align}\label{eq:Hlasymptotic}
H_{R^{(d)}_{B}}(\kllevel)&:=\frac{1}{{48^{2}}}\min_{\kllevelrun=1}^{\kllevel} \delta^{2}_{R^{(d)}_{B}}(\kllevelrun)=\frac{1}{{2304}} \min_{\kllevelrun=1}^{\kllevel}\left(\lambda_{R^{(1)}_{B}}(1) \right)^{2(d-1)}\left(\delta_{R^{(1)}_{B}}(\kllevelrun)\right)^{2}\nonumber \\
&=\frac{1}{{2304}} \left(\lambda_{R^{(1)}_{B}}(1) \right)^{2(d-1)} \frac{1}{4}\left(\frac{2}{\pi}\right)^4\min_{\kllevelrun=1}^{\kllevel}\left(\frac{\kllevelrun}{\left(\kllevelrun^2 -\frac{1}{4}\right)^{2}}\right)^{2} 
=\frac{1}{{9216}}\left(\frac{2}{\pi}\right)^{4d} \min_{\kllevelrun=1}^{\kllevel}\left(\frac{\kllevelrun}{\left(\kllevelrun^2 -\frac{1}{4}\right)^{2}}\right)^{2}
\nonumber  \\ &\approx \kllevel^{-6} \quad \text{for } \kllevel \text{ large.}
\end{align}
Note here that, in contrast to the value $G_{R^{(d)}_{B}}(\kllevel)$ in \eqref{eq:gl:asymptotic}, the dimension enters the overall value of $H_{R^{(d)}_{B}}(\kllevel)$ in \eqref{eq:Hlasymptotic}  now  exponentially.

Altogether, in the general situation, we have the following result.
\begin{theorem}\label{thm:final2}
Under the conditions of Proposition \ref{prop:spectralgap}, i.e. \eqref{condspectralgapfinal1}, with condition \eqref{condhL}  and with definition \eqref{tilderhoM},  there holds 
\begin{align}\label{eq:final-3}
\expect{\norm{{R}-R^{(\kllevel;h;\nsamples)}}_{L^{2}(\dom\times\dom)}} &\lesssim  \kllevel^{-\frac{2s}{d}-\frac{1}{2}} + \left( \kllevel^{\frac{1}{2}}+ G(\kllevel) \right) \tilde{\rho}^{\frac{1}{2}}_{h}(\nsamples) \lambda_{\max}\left(\massmatrix\right) \nonumber\\
	&+  L^{\frac{1}{2}} h^{-d} \exp\left(- \nsamples  \rho_1 H(\kllevel)\lambda_{\max}^{-2} \left(\massmatrix\right)  \right),
\end{align}
where $H(\kllevel)$ is defined in \eqref{HL} and $\rho_1$ denotes a constant given in \eqref{rho1def}.
\end{theorem}
\begin{proof}
From Theorem \ref{finalthm}, we have  with condition \eqref{condhL} that
\begin{align*}
&\expect{\norm{{R}-R^{(\kllevel;h;\nsamples)}}_{L^{2}(\dom\times\dom)}} \lesssim  \kllevel^{-\frac{2s}{d}-\frac{1}{2}}   + h^{s} \kllevel^{\frac{1}{2}}+  \expect{ \normL{E_{3}}{\dom\times \dom}} \\
&\lesssim  \kllevel^{-\frac{2s}{d}-\frac{1}{2}}   +  \expect{ \normL{E_{3}}{\dom\times \dom}} \\
&\lesssim  \kllevel^{-\frac{2s}{d}-\frac{1}{2}}  +  \expect{ \normL{E_{3;1}}{\dom\times \dom}+ \normL{E_{3;2}}{\dom\times \dom}+ \normL{E_{3;3}}{\dom\times \dom}}. 
\end{align*}
The term $\normL{{E}_{3;1}}{\dom\times \dom}$ from  \eqref{eq:zzz1a} can be handled directly and we obtain with Corollary  \ref{corboundtildeS} the bound
\begin{align}\label{expectE31}
\expect{\normL{{E}_{3;1}}{\dom\times \dom}}&\le  \kllevel^{1/2} \expect{\left\|\tildstiffnessmatrix-\tildestiffnessmatrixM \right\|_{2 \to 2}} \lesssim  \kllevel^{1/2}  \rho^{\frac{1}{2}}_{h}(\nsamples)
\end{align}
Next, we define the set, where the assumption \ref{ass:spectralgap} is violated, as bad and the good set simply as the complement of the bad set, i.e.,
\begin{align*}
M_{\text{bad}}&:=\left\{\vec{\om} \in \Omega :  \frac{\min_{1\le \kllevelrun \le \kllevel }\delta_{\kllevelrun}}{4}< \left\|\tildstiffnessmatrix-\tildestiffnessmatrixM \right\|_{2 \to 2}\right\},\\
M_{\text{good}}&:=\left\{\vec{\om} \in \Omega :  \frac{\min_{1\le \kllevelrun \le \kllevel }\delta_{\kllevelrun}}{4}\ge \left\|\tildstiffnessmatrix-\tildestiffnessmatrixM \right\|_{2 \to 2}\right\}.
\end{align*}
We note that Proposition \ref{prop:spectralgap} yields 
\begin{align*}
	\left| M_{\text{bad}}\right| \le 2 \dimVh 5^{\taperingint}\exp\left(- \nsamples  \rho_1   H(\kllevel) \lambda_{\max}^{-2} \left(\massmatrix\right) \right)\quad \text{where} \quad \tau =\nsamples^{\frac{1}{2\alpha+1}}.
\end{align*}
Now, we split the term $\expect{\normL{{E}_{3;j}}{\dom\times \dom}}$ as follows
\begin{align*}
\expect{\normL{{E}_{3;j}}{\dom\times \dom}}=\int_{\vec{\om} \in M_{\text{bad}} } \normL{{E}_{3;j}}{\dom\times \dom} \, \mathrm{d}\prob(\vec{\om})+\int_{\vec{\om} \in M_{\text{good}} } \normL{{E}_{3;j}}{\dom\times \dom} \, \mathrm{d}\prob(\vec{\om}),
\end{align*}
and we also recall that $\normL{{E}_{3;2}}{\dom\times \dom}=\normL{{E}_{3;3}}{\dom\times \dom}$ as we have seen in \eqref{e32e33}.
In order to treat the first integral relating to the bad set, we will show that the integrand is point-wise bounded by a deterministic quantity on $ M_{\text{bad}}$, i.e., that $\normL{{E}_{3;j}}{\dom\times \dom} \le C$ holds.
First, we use \eqref{philhmortho} to get at least
\begin{align*}
\normL{\phi^{(h)}_{\kllevelrun}-\phi^{(h;\nsamples)}_{\kllevelrun}}{\dom}^2&\le 2\normL{\phi^{(h)}_{\kllevelrun}}{\dom}^2+2\normL{\phi^{(h;\nsamples)}_{\kllevelrun}}{\dom}^2\le 4.
\end{align*}
Moreover, we use again \eqref{lambdatriag2}, i.e.,
\begin{align*}
\left|\lambda^{(h;\nsamples)}_{\kllevelrun} \right|\le \left| \lambda^{(h)}_{\kllevelrun} -\lambda^{(h;\nsamples)}_{\kllevelrun} \right| +\left| \lambda^{(h)}_{\kllevelrun} \right| \le  \left\|\tildstiffnessmatrix -\tildestiffnessmatrixM \right\|_{2 \to 2} +\left| \lambda_{\kllevelrun} \right|.
\end{align*}
Thus, we obtain for $j=2,3$ the bound
\begin{align*}
\normL{{E}_{3;j}}{\dom\times\dom}^2
&=\sum_{\kllevelrun=1}^{\kllevel} (\lambda^{(h;\nsamples)}_{\kllevelrun})^2 \normL{\phi^{(h)}_{\kllevelrun}-\phi^{(h;\nsamples)}_{\kllevelrun}}{\dom}^2 \le 4 \sum_{\kllevelrun=1}^{\kllevel} (\lambda^{(h;\nsamples)}_{\kllevelrun})^2 
\\&\le 4 \sum_{\kllevelrun=1}^{\kllevel} \left(  \left\|\tildstiffnessmatrix -\tildestiffnessmatrixM \right\|_{2 \to 2} +\left| \lambda_{\kllevelrun} \right|\right)^2  \le 8  \sum_{\kllevelrun=1}^{\kllevel}  \left\|\tildstiffnessmatrix -\tildestiffnessmatrixM \right\|^2_{2 \to 2}  + 8 \sum_{\kllevelrun=1}^{\kllevel}\left| \lambda_{\kllevelrun} \right|^2
\\&\le 8  \kllevel  \left\|\tildstiffnessmatrix -\tildestiffnessmatrixM \right\|^2_{2 \to 2}  + 8\kllevel \left| \lambda_{1} \right|^2
\end{align*}
and get
\begin{align*}
\normL{{E}_{3;j}}{\dom\times\dom}\le 2\sqrt{2} \ \kllevel^{\frac{1}{2}}\left(  \left\|\tildstiffnessmatrix -\tildestiffnessmatrixM \right\|_{2 \to 2} +\lambda_{1}  \right)
\end{align*}
using $\sqrt{a+b} \le \sqrt{a}+\sqrt{b}$ for $a,b\ge 0$.
Hence, we obtain on the bad set
\begin{align*}
\int_{\vec{\om} \in M_{\text{bad}} } \normL{{E}_{3;j}}{\dom\times \dom} \, \mathrm{d}\prob(\vec{\om})\le  2\sqrt{2}  \ \kllevel^{\frac{1}{2}} \int_{\vec{\om} \in M_{\text{bad}} }   \left\|\tildstiffnessmatrix -\tildestiffnessmatrixM \right\|_{2 \to 2} \, \mathrm{d}\prob(\vec{\om}) +  2\sqrt{2}  \ \kllevel^{\frac{1}{2}} \left|  M_{\text{bad}} \right|  \lambda_{1}.
\end{align*}
For the first term, we use \eqref{errorexpec} to get
\begin{align*}
\int_{\vec{\om} \in M_{\text{bad}} }   \left\|\tildstiffnessmatrix -\tildestiffnessmatrixM \right\|_{2 \to 2} \, \mathrm{d}\prob(\vec{\om})& \le \int_{\vec{\om} \in \Omega }    \left\|\tildstiffnessmatrix -\tildestiffnessmatrixM \right\|_{2 \to 2} \, \mathrm{d}\prob(\vec{\om})
\lesssim  \rho^{\frac{1}{2}}_{h}(\nsamples)\lambda_{\max}\left(\massmatrix\right) 
\end{align*}
and thus obtain 
\begin{align*}
\int_{\vec{\om} \in M_{\text{bad}} } \normL{{E}_{3;j}}{\dom\times \dom} \, \mathrm{d}\prob(\vec{\om})\lesssim 
2 \sqrt{2}L^{\frac{1}{2}} \lambda_{\max}\left(\massmatrix\right) 
\rho^{\frac{1}{2}}_{h}(\nsamples)+2\sqrt{2} \left|  M_{\text{bad}} \right| L^{\frac{1}{2}} \lambda_1.
\end{align*}
We are left with estimating the integral of $\normL{E_{3}}{\dom\times\dom}$ over the set $M_{\text{good}}$. To this end, we employ the bound \eqref{E3error}, which is only valid on the good set as we have the condition \eqref{goodset}. Recall now
\begin{align*}
\normL{{E}_{3}}{\dom\times\dom} \le
 \left\|\tildstiffnessmatrix-\tildestiffnessmatrixM \right\|_{2\to 2} \left( \kllevel^{\frac{1}{2}}+\sqrt{2}C \kllevel^{\frac{1}{2}} + 4\sqrt{2} C G(\kllevel)\right) \quad \text{on }M_{\text{good}}.
\end{align*}
Thus we get 
\begin{align*}
\int_{M_{\text{good}} } \normL{E_{3}}{\dom\times\dom} \, \mathrm{d}\prob (\vec{\om}) & 
\le \left( \kllevel^{\frac{1}{2}}+\sqrt{2}C \kllevel^{\frac{1}{2}} + 4\sqrt{2} C G(\kllevel)\right) \int_{M_{\text{good}}}  \left\|\tildstiffnessmatrix-\tildestiffnessmatrixM \right\|_{2\to 2} \, \mathrm{d}\prob (\vec{\om}) 
\\&\le \left( \kllevel^{\frac{1}{2}}+\sqrt{2}C \kllevel^{\frac{1}{2}} + 4\sqrt{2} C G(\kllevel)\right)\expect{ \left\|\tildstiffnessmatrix-\tildestiffnessmatrixM \right\|_{2\to 2} }
\\& \lesssim \left( \kllevel^{\frac{1}{2}}+\sqrt{2}C \kllevel^{\frac{1}{2}} + 4\sqrt{2} C G(\kllevel)\right) \rho^{\frac{1}{2}}_{h}(\nsamples)\lambda_{\max}\left(\massmatrix\right) .
\end{align*}
Together with Proposition \ref{prop:spectralgap} and the notation from \eqref{tilderhoM} this completes the proof.
\end{proof}

\section{Discussion}
It now remains to put the result of Theorem 5.2 into context. To this end, let us aim at an error bound 
$$\expect{\| R-R^{(L,h,M)}\|_{L^{2}(\dom \times \dom )}} \leq c \varepsilon$$ 
with a prescribed, fixed accuracy $\varepsilon >0$ and a small constant $c$. The question is then: What are the conditions on the discretization parameters $L,h,M$ to achieve this aim? For the sake of simplicity, we restrict ourselves here to the situation $\lambda_{\max}\left(\massmatrix\right) \approx \lambda_{\min}\left(\massmatrix\right) \approx 1$, which is satisfied for instance for any orthonormal basis.\footnote{For the nodal basis, we would obtain, after proper scaling,  $\lambda_{\max}\left(\massmatrix\right) \approx \lambda_{\min}\left(\massmatrix\right) \approx h^d$. The resulting discussion for this case is left to the reader.}
Moreover, we assume the smoothness assumption \eqref{eqA22}, i.e., $s > d/2$.
We proceed in a term by term fashion as follows:

For the first term in our error bound \eqref{eq:final-3} we obtain from $L ^{-2s/d -1/2} \stackrel{!}{\le}  \varepsilon$ directly the condition $L \ge \varepsilon^{-2d/(4s+d)}$.
We also point out that choosing $\kllevel$ much larger will destroy the balance of the error contributions. Hence, we will assume
\begin{align}\label{aaa}
L_{\varepsilon} := \lceil \varepsilon^{-\frac{2d}{4s+d}} \rceil.
\end{align}

The coupling \eqref{aaa} influences the choices for $h$ in terms of $\kllevel$ via the Proposition  \ref{prop:spectralgap}, i.e. \eqref{condspectralgapfinal1} and the relation \eqref{condhL}. We get as sufficient condition on $h$
\begin{align}\label{finalcondhL}
h^{2s} \lesssim \min\left\{H^{\frac{1}{2}}(\kllevel) \lambda_{\kllevel+1},\lambda^2_{\kllevel} \right\} \lesssim  \min \left\{ \delta_{\kllevelrun} \lambda_{\kllevel+1},\lambda^{2}_{\kllevel}  \ : \ 1\le \kllevelrun \le \kllevel \right\}
\end{align}
with the notation \eqref{HL}. 

For the second error term we now assume $G(\kllevel)=\kllevel^{\tilde{\gamma }}$
for some $\tilde{\gamma} \in \R$. Thus  
\begin{align}\label{defgamma}
\kllevel^{\frac{1}{2}}+G(\kllevel)\lesssim \kllevel^{\gamma}, \quad \text{with} \quad \gamma:=\max\{\frac{1}{2},\tilde{\gamma}\}.
\end{align}
Consequently, we have $(\kllevel^{\frac{1}{2}}+G(\kllevel)) \tilde{\rho}^{\frac{1}{2}}_{h}(\nsamples) \lesssim \kllevel^{\gamma} \tilde{\rho}^{\frac{1}{2}}_{h}(\nsamples) $. At this point, we note that the clauses in the definition of $\tilde{\rho}_h(M)$ in \eqref{tilderhoM} make a case distinction necessary.

In the case $\dimVh < \nsamples^{\frac{1}{2\alpha+1}}$ with  $\dimVh =s^{d} h^{-d}$, we obtain $h > s \nsamples^{-\frac{1}{d(2\alpha+1)}}$. Thus, with \eqref{finalcondhL}, we have the following inequality for $h$
\begin{align}\label{firstchain}
\nsamples_{\varepsilon}^{-\frac{1}{d(2\alpha+1)}} \lesssim h_{\varepsilon} \lesssim  \min\left\{H^{\frac{1}{4s}}(\kllevel_{\varepsilon}) \lambda^{\frac{1}{2s}}_{\kllevel_{\varepsilon}+1},\lambda^{\frac{1}{s}}_{\kllevel_{\varepsilon}} \right\}.
\end{align}
Moreover we get from \eqref{tilderhoM} the bound $\tilde{\rho}_h(M)\lesssim h^{-d} \nsamples^{-1}$. We therefore infer the inequality
\begin{align*}
	(\kllevel_{\varepsilon}^{\frac{1}{2}}+G(\kllevel_{\varepsilon})) \tilde{\rho}^{\frac{1}{2}}_{h_{\varepsilon}}(\nsamples_{\varepsilon})\lesssim  \kllevel_{\varepsilon}^{\gamma} \nsamples_{\varepsilon}^{\frac {1}{2(2\alpha+1)}} \nsamples_{\varepsilon}^{-\frac{1}{2}} = \nsamples_{\varepsilon}^{-\frac{\alpha}{2\alpha+1}}\kllevel_{\varepsilon}^{\gamma}.
  \end{align*}
  Thus, we obtain the sufficient condition 
 \begin{align}\label{firstsuffcond}
\nsamples_{\varepsilon}^{-\frac{\alpha}{2\alpha+1}}\kllevel_{\varepsilon}^{\gamma}
	 \stackrel{!}{\le} \varepsilon 
\end{align}  
which implies an error bound of size $\varepsilon$. 
From this condition, we would like to infer a condition on $\nsamples_{\varepsilon}$.
The sufficient condition \eqref{firstsuffcond} implies 
\begin{align*}
\nsamples_{\varepsilon}^{-\frac{\alpha}{2\alpha+1}}\kllevel_{\varepsilon}^{\gamma}
	 \le \varepsilon \Leftrightarrow  \nsamples_{\varepsilon} \ge \varepsilon^{-\frac{2\alpha+1}{\alpha}} \kllevel_{\varepsilon}^{\gamma\frac {2\alpha+1}{\alpha} },
\end{align*}
i.e., we have a lower bound on $\nsamples_{\varepsilon}$. This means that we can satisfy \eqref{firstchain} by making  $\nsamples_{\varepsilon}$ sufficiently large.
Precisely, we have 
\begin{align}\label{condM1}
	\nsamples_{\varepsilon} \gtrsim \max\left\{ \varepsilon^{-\frac{2\alpha+1}{\alpha}} \kllevel_{\varepsilon}^{\gamma\frac {2\alpha+1}{\alpha} },\left( \min\left\{H^{\frac{1}{4s}}(\kllevel_{\varepsilon}) \lambda^{\frac{1}{2s}}_{\kllevel_{\varepsilon}+1},\lambda^{\frac{1}{s}}_{\kllevel_{\varepsilon}} \right\} \right)^{-d(2\alpha+1)}  \right\}.
\end{align}
Furthermore we have for the third error contribution the condition
\begin{align*}
 \kllevel_{\varepsilon}^{\frac{1}{2}} h_{\varepsilon}^{-d} \exp\left(- \nsamples_{\varepsilon}  \rho_1 H(\kllevel_{\varepsilon}) \right)\stackrel{!}{\le} \varepsilon \Leftrightarrow  \exp\left(- \nsamples_{\varepsilon}  \rho_1 H(\kllevel_{\varepsilon}) \right) \le  \varepsilon \kllevel_{\varepsilon}^{-\frac{1}{2}} h_{\varepsilon}^{d}.
 \end{align*}
We use \eqref{firstchain} and \eqref{aaa} to observe 
\begin{align*}
\varepsilon \kllevel_{\varepsilon}^{-\frac{1}{2}} \nsamples_{\varepsilon}^{-\frac{1}{2\alpha+1}}=\varepsilon \lceil \varepsilon^{-\frac{2d}{4s+d}} \rceil^{-\frac{1}{2}} \nsamples_{\varepsilon}^{-\frac{1}{2\alpha+1}}	 \le \varepsilon \kllevel_{\varepsilon}^{-\frac{1}{2}} h_{\varepsilon}^{d}.
\end{align*}
Now, we define $\bar{\nsamples}_{\varepsilon} \in \N$ via 
\begin{align}\label{barn}
	\bar{\nsamples}_{\varepsilon}:=\arg\min\left\{M \in \N :  \exp\left(- M  \rho_1 H(\kllevel_{\varepsilon})\right) \le \varepsilon \lceil \varepsilon^{-\frac{2d}{4s+d}} \rceil^{-\frac{1}{2}} M^{-\frac{1}{2\alpha+1}} \right\}.
\end{align}
%

Altogether, to ensure an overall error bound of size $4\varepsilon$ in the case $\dimVh < \nsamples^{\frac{1}{2\alpha+1}}$, this yields the sufficient conditions 
\begin{align}
\label{finalLeps}
L_{\varepsilon} &= \lceil \varepsilon^{-\frac{2d}{4s+d}} \rceil,\\
\label{finalMeps}
 \nsamples_{\varepsilon} &\gtrsim \max\left\{\bar{\nsamples}_{\varepsilon},  \varepsilon^{-\frac{2\alpha+1}{\alpha}} \kllevel_{\varepsilon}^{\gamma\frac {2\alpha+1}{\alpha} },\left( \min\left\{H^{\frac{1}{4s}}(\kllevel_{\varepsilon}) \lambda^{\frac{1}{2s}}_{\kllevel_{\varepsilon}+1},\lambda^{\frac{1}{s}}_{\kllevel_{\varepsilon}} \right\} \right)^{-d(2\alpha+1)}\right\},\\
\label{finalheps}
h_{\varepsilon} &\sim \min\left\{ \nsamples_{\varepsilon}^{-\frac{1}{d(2\alpha+1)}},h_0\right\},
\end{align}
where $h_0$ stems from Proposition \ref{prop:babuska}.

Now, we turn to the second clause in the definition of $\tilde{\rho}_h(M)$ in \eqref{tilderhoM}, i.e., we consider the case $\dimVh \ge \nsamples^{\frac{1}{2\alpha+1}}$ with $\dimVh= s^{d}h^{-d}$. Then we obtain $h \le s  \nsamples^{-\frac{1}{d(2\alpha+1)}}$ and the inequalities \eqref{firstchain} change to 
\begin{align}\label{secondchain}
 h_{\varepsilon} \lesssim  \min\left\{H^{\frac{1}{4s}}(\kllevel_{\varepsilon}) \lambda^{\frac{1}{2s}}_{\kllevel_{\varepsilon}+1},\lambda^{\frac{1}{s}}_{\kllevel_{\varepsilon}}, \nsamples_{\varepsilon}^{-\frac{1}{d(2\alpha+1)}}  \right\},
\end{align} 
i.e., we have no lower bound for  $h_{\varepsilon}$ at this point.
Recall \eqref{tilderhoM}, i.e., 
$\tilde{\rho}_{h}(\nsamples)=\nsamples^{-\frac{2\alpha}{2\alpha+1}} +d \log(h^{-1})\nsamples^{-1}$.
Now, we distinguish whether $\nsamples^{-\frac{2\alpha}{2\alpha+1}} \le d \log(h^{-1})\nsamples^{-1}$ or  $\nsamples^{-\frac{2\alpha}{2\alpha+1}} \ge d \log(h^{-1})\nsamples^{-1}$ holds.
Moreover, we observe that for all $ \beta\ge 0$ there is a $h_{\beta}>0$ such that
\begin{align*}
 \log(h^{-1})^{\frac{1}{2}} \le h^{-\beta} \quad \text{for all }  h\le h_{\beta}.
\end{align*}
Thus for a fixed $\beta >0$, we get the chain of inequalities using $\tilde{\rho}_{h}(\nsamples_{\varepsilon})=\nsamples_{\varepsilon}^{-\frac{2\alpha}{2\alpha+1}} +d \log(h_{\varepsilon}^{-1})\nsamples_{\varepsilon}^{-1}\le 2 d \log(h_{\varepsilon}^{-1})\nsamples_{\varepsilon}^{-1}$
\begin{align*}
(\kllevel_{\varepsilon}^{\frac{1}{2}}+G(\kllevel_{\varepsilon}))  \tilde{\rho}^{\frac{1}{2}}_{h_{\varepsilon}}(\nsamples_{\varepsilon})&\lesssim  \kllevel_{\varepsilon}^{\gamma}  d^{\frac{1}{2}} \log(h_{\varepsilon}^{-1})^{\frac{1}{2}}\nsamples^{-\frac{1}{2}} \lesssim \kllevel_{\varepsilon}^{\gamma} h_{\varepsilon}^{-\beta} \nsamples^{-\frac{1}{2}}
 \le  \kllevel_{\varepsilon}^{\gamma}  \nsamples_{\varepsilon}^{-\frac{1}{2}}\lambda_{\kllevel_{\varepsilon}}^{-\frac{\beta}{s}} \\&\le  \kllevel_{\varepsilon}^{\gamma}  \nsamples_{\varepsilon}^{-\frac{1}{2}} \kllevel_{\varepsilon}^{\frac{(2s+d)\beta}{ds}}= \nsamples_{\varepsilon}^{-\frac{1}{2}}\kllevel_{\varepsilon}^{\frac{(2s+d)\beta+sd\gamma}{ds} }.
  \end{align*}
Then, using \eqref{aaa} we encounter the sufficient condition 
\begin{align*}
\nsamples_{\varepsilon}^{-\frac{1}{2}}\kllevel_{\varepsilon}^{\frac{(2s+d)\beta+sd\gamma}{sd} } \stackrel{!}{\le } \varepsilon \Leftrightarrow \nsamples_{\varepsilon} \ge \kllevel_{\varepsilon}^{\frac{2(2s+d)\beta+2sd\gamma}{sd} } \varepsilon^{-2}=\lceil \varepsilon^{-\frac{2d}{4s+d}} \rceil^{\frac{2(2s+d)\beta+2sd\gamma}{sd} } \varepsilon^{-2}
\end{align*}
to get an error bound of size $\varepsilon$ for some  $\beta >0$ and $h_{\varepsilon}\le h_{\beta}$.
For the third error contribution, we have with \eqref{condspectralgapfinal1} the sufficient condition
\begin{align*}
\kllevel_{\varepsilon}^{\frac{1}{2}} h_{\varepsilon}^{-d} \exp\left(- \nsamples_{\varepsilon}  \rho_1 H(\kllevel_{\varepsilon}) \right) \stackrel{!}{\le} \varepsilon \Leftrightarrow \left(\kllevel_{\varepsilon}^{\frac{1}{2}}\varepsilon^{-1}\exp\left(- \nsamples_{\varepsilon}  \rho_1 H(\kllevel_{\varepsilon}) \right)\right)^{\frac{1}{d}} \le h_{\varepsilon}
\end{align*}
to ensure an error bound of size $\varepsilon$. 
We now define $\tilde{\nsamples}_{\varepsilon} \in \N$ via 
\begin{align}\label{tilden}
	\tilde{\nsamples}_{\varepsilon}:=\arg\min\left\{M \in \N :  \kllevel_{\varepsilon}^{\frac{1}{2}}\varepsilon^{-1}\exp\left(- M  \rho_1 H(\kllevel_{\varepsilon}) \right) \le M^{-\frac{1}{2\alpha+1}}\right\}.
\end{align}
Altogether, in the case $\dimVh > \nsamples^{\frac{1}{2\alpha+1}}$, $\nsamples^{-\frac{2\alpha}{2\alpha+1}} \le d \log(h^{-1})\nsamples^{-1}$, this yields the sufficient conditions
\begin{align}
\label{finalLeps2}
&L_{\varepsilon} = \lceil \varepsilon^{-\frac{2d}{4s+d}} \rceil,\\
\label{finalMeps2}
&  \nsamples_{\varepsilon} \gtrsim \max\left\{\tilde{\nsamples}_{\varepsilon},\lceil \varepsilon^{-\frac{2d}{4s+d}} \rceil^{\frac{2(2s+d)\beta+2sd\gamma}{sd} } \varepsilon^{-2} \right\},\\
\label{finalheps2}
& \left(\kllevel_{\varepsilon}^{\frac{1}{2}}\varepsilon^{-1}\exp\left(- \nsamples_{\varepsilon}  \rho_1 H(\kllevel_{\varepsilon}) \right)\right)^{\frac{1}{d}}\lesssim  h_{\varepsilon} \lesssim  \min\left\{H^{\frac{1}{4s}}(\kllevel_{\varepsilon}) \lambda^{\frac{1}{2s}}_{\kllevel_{\varepsilon}+1},\lambda^{\frac{1}{s}}_{\kllevel_{\varepsilon}}, \nsamples_{\varepsilon}^{-\frac{1}{d(2\alpha+1)}} , h_0
\right\},
\end{align}
to ensure an overall error bound of size $3\varepsilon$, where $h_0$ stems from Proposition \ref{prop:babuska} and \eqref{tilden} ensures that the interval for $h_{\varepsilon}$ is nontrivial.
 
Finally, we consider the case $\dimVh > \nsamples_{\varepsilon}^{\frac{1}{2\alpha+1}}$ and $\nsamples_{\varepsilon}^{-\frac{2\alpha}{2\alpha+1}} \ge d \log(h_{\varepsilon}^{-1})\nsamples_{\varepsilon}^{-1}$. As $1> \frac{2\alpha}{2\alpha+1}$, we note, that, for fixed $h_{\varepsilon}$, this inequality can be satisfied if $\nsamples_{\varepsilon}$ is large enough. 
Moreover, we observe 
\begin{align*}
\nsamples^{-\frac{2\alpha}{2\alpha+1}} \ge d \log(h^{-1})\nsamples^{-1} &\Leftrightarrow \nsamples^{\frac{1}{2\alpha+1}}  \ge d \log(h^{-1})\\& \Leftrightarrow -d^{-1}\nsamples^{\frac{1}{2\alpha+1}}\le \log(h) \Leftrightarrow \exp\left(-d^{-1}\nsamples^{\frac{1}{2\alpha+1}} \right) \le h.
\end{align*}
This yields a lower bound on $h_{\varepsilon}$ as
\begin{align}\label{thirdchainfirstcase}
 \exp\left(-d^{-1}\nsamples_{\varepsilon}^{\frac{1}{2\alpha+1}} \right) \lesssim  h_{\varepsilon} \lesssim  \min\left\{H^{\frac{1}{4s}}(\kllevel_{\varepsilon}) \lambda^{\frac{1}{2s}}_{\kllevel_{\varepsilon}+1},\lambda^{\frac{1}{s}}_{\kllevel_{\varepsilon}}, \nsamples_{\varepsilon}^{-\frac{1}{d(2\alpha+1)}}  \right\}.
\end{align}
We now define $\hat{\nsamples}_{\varepsilon} \in \N$ via 
\begin{align}\label{hatn}
	\hat{\nsamples}_{\varepsilon}:=\arg\min\left\{M \in \N : \exp\left(-d^{-1}M^{\frac{1}{2\alpha+1}} \right)  \le M^{-\frac{1}{d(2\alpha+1)}}  \right\},
\end{align}
and hence condition \eqref{thirdchainfirstcase} can be fulfilled for $\nsamples_{\varepsilon}\ge \hat{\nsamples}_{\varepsilon}$.
As we now have $\tilde{\rho}_{h}(\nsamples)=\nsamples^{-\frac{2\alpha}{2\alpha+1}} +d \log(h^{-1})\nsamples^{-1} \le 2 \nsamples^{-\frac{2\alpha}{2\alpha+1}} $, we obtain the chain of inequalities
\begin{align*}
(\kllevel_{\varepsilon}^{\frac{1}{2}}+G(\kllevel_{\varepsilon}))\tilde{\rho}^{\frac{1}{2}}_{h_{\varepsilon}}(\nsamples_{\varepsilon})&\lesssim  \kllevel_{\varepsilon}^{\gamma} \nsamples_{\varepsilon}^{-\frac{\alpha}{2\alpha+1}}   \lesssim  \kllevel_{\varepsilon}^{\gamma}  \nsamples_{\varepsilon}^{-\frac{\alpha}{2\alpha+1}}.
  \end{align*}
Thus, we encounter the sufficient condition
\begin{align*}
 \kllevel_{\varepsilon}^{\gamma}  \nsamples_{\varepsilon}^{-\frac{\alpha}{2\alpha+1}}\stackrel{!}{\le } \varepsilon \Leftrightarrow \nsamples_{\varepsilon} \ge \kllevel_{\varepsilon}^{\frac{\gamma(2\alpha+1)}{\alpha}}\varepsilon^{-\frac{2\alpha+1}{\alpha}}
\end{align*}
to ensure an error bound of size $\varepsilon$.
For the fourth error contribution, we derive the inequality
\begin{align}\label{suff3}
 \kllevel_{\varepsilon}^{\frac{1}{2}} h_{\varepsilon}^{-d} \exp\left(- \nsamples_{\varepsilon}  \rho_1 H(\kllevel_{\varepsilon}) \right) \stackrel{!}{\le}\varepsilon \Leftrightarrow  \left(\kllevel_{\varepsilon}^{\frac{1}{2}}\varepsilon^{-1}\exp\left(- \nsamples_{\varepsilon}  \rho_1 H(\kllevel_{\varepsilon}) \right)\right)^{\frac{1}{d}} \le h_{\varepsilon}.
\end{align}
Moreover, we can now define 
 $\nsamples^{\prime}_{\varepsilon} \in \N$ via 
\begin{align}\label{primen}
	\nsamples^{\prime}_{\varepsilon}:=\arg\min\left\{M \in \N :   \left(\kllevel_{\varepsilon}^{\frac{1}{2}}\varepsilon^{-1}\exp\left(- \nsamples_{\varepsilon}  \rho_1 H(\kllevel_{\varepsilon}) \right)\right)^{\frac{1}{d}}\le  \exp\left(-d^{-1}\nsamples_{\varepsilon}^{\frac{1}{2\alpha+1}} \right) \right\}.
\end{align}
This finally yields the sufficient conditions
\begin{align}
\label{finalLeps3}
&L_{\varepsilon} = \lceil \varepsilon^{-\frac{2d}{4s+d}} \rceil,\\
\label{finalMeps3}
&   \nsamples_{\varepsilon} \gtrsim \max\left\{\hat{\nsamples}_{\varepsilon},\nsamples^{\prime}_{\varepsilon},  \kllevel_{\varepsilon}^{\frac{\gamma(2\alpha+1)}{\alpha}}\varepsilon^{-\frac{2\alpha+1}{\alpha}},\left( \min\left\{H^{\frac{1}{4s}}(\kllevel_{\varepsilon}) \lambda^{\frac{1}{2s}}_{\kllevel_{\varepsilon}+1},\lambda^{\frac{1}{s}}_{\kllevel_{\varepsilon}} \right\} \right)^{-d(2\alpha+1)}\right\}\\
\label{finalheps3}
& \left(\kllevel_{\varepsilon}^{\frac{1}{2}}\varepsilon^{-1}\exp\left(- \nsamples_{\varepsilon}  \rho_1 H(\kllevel_{\varepsilon}) \right)\right)^{\frac{1}{d}}\lesssim  h_{\varepsilon} \lesssim  \min\left\{H^{\frac{1}{4s}}(\kllevel_{\varepsilon}) \lambda^{\frac{1}{2s}}_{\kllevel_{\varepsilon}+1},\lambda^{\frac{1}{s}}_{\kllevel_{\varepsilon}}, \nsamples_{\varepsilon}^{-\frac{1}{d(2\alpha+1)}} , \exp\left(-d^{-1}\nsamples^{\frac{1}{2\alpha+1}} \right),h_0 \right\},
\end{align}
for the case $\dimVh > \nsamples^{\frac{1}{2\alpha+1}}$, $\nsamples^{-\frac{2\alpha}{2\alpha+1}} \ge d \log(h^{-1})\nsamples^{-1}$ which ensure an overall error bound of size $4\varepsilon$. Note that we always fix $L_{\varepsilon} $ in the first place and that we can always satisfy the conditions \eqref{finalMeps},  \eqref{finalMeps2} and  \eqref{finalheps3} by choosing $\nsamples_{\varepsilon}$ large enough. 
This then fixes $h_{\varepsilon}$ in \eqref{finalheps} and fixes the potential intervals for $h_{\varepsilon}$ in \eqref{finalheps2} and \eqref{finalheps3}.
Of course in practice we would always choose  $\nsamples_{\varepsilon}$ as small and $h_{\varepsilon}$ as large as possible.

We summarize this discussion in a corollary.
\begin{corollary}
Fix an accuracy $\varepsilon>0$ and choose $L_{\varepsilon} = \lceil \varepsilon^{-\frac{2d}{4s+d}} \rceil $.
Then, we the choose $\nsamples_{\varepsilon}$ according to \eqref{finalMeps},  \eqref{finalMeps2} and  \eqref{finalheps3}. Then, we choose $h_{\varepsilon}$ as in \eqref{finalheps} or as maximal value of the potential intervals in \eqref{finalheps2} and \eqref{finalheps3}.
Then, in all possible cases, the error bound 
\begin{align*}
\expect{\norm{{R}-R^{(\kllevel;h;\nsamples)}}_{L^{2}(\dom\times\dom)}} \le 3 \varepsilon
\end{align*}
holds.
\end{corollary}

Note at this point that a practical goal is to select $h_{\varepsilon}$ as large as possible and $\nsamples_{\varepsilon}$ as small as possible and to still ensure an error level of size $3\varepsilon$.
Hence, we advocate to stick to the second case, where a reasonable choice as in  \eqref{finalheps2} and furthermore for $\nsamples_{\varepsilon}$ as in  \eqref{finalMeps2} is possible. This leads naturally to the tapering estimator.

In general our estimates involve the values of $G(\kllevel)$ and $H(\kllevel)$, i.e., spectral properties of the unknown kernel $R$. These numbers can not be easily determined at all. However, in the case of Brownian motion, we were able to derive precise values in \eqref{eq:gl:asymptotic} and \eqref{eq:Hlasymptotic}.  Now, let us finally consider the most simple case of univariate Brownian motion and let us discuss the tapering estimator in that situation.
The univariate Brownian motion (i.e., $d=1$) was given in \eqref{univariateBrownianMotion}, i.e.,
\begin{align*}
	R^{(1)}_{B}:(0,1)^2 \to \R, \quad R^{(1)}_{B}(x,x^{\prime}):=\min\{x,x^{\prime}\}.
\end{align*}
with eigenvalues $\lambda_{R^{(1)}_{B}}(\kllevelrun) \sim \kllevelrun^{-2}$ and the constant from \eqref{defgamma} as $\gamma_{R^{(1)}_{B}}=\frac{3}{2}$. The spatial smoothness is $s=\frac{1}{2}-\delta$ for arbitrary small $\delta>0$. First, we observe for \eqref{finalLeps2} that
\begin{align}
L_{\varepsilon} = \lceil \varepsilon^{-\frac{2d}{4s+d}} \rceil \sim \lceil \varepsilon^{-\frac{2}{3}} \rceil \label{finalLeps3special}.
\end{align}
From  \eqref{finalMeps2}, we have
\begin{align*}
\nsamples_{\varepsilon} \gtrsim\max\left\{\tilde{\nsamples}_{\varepsilon},\lceil \varepsilon^{-\frac{2}{3}} \rceil^{8\beta+3} \varepsilon^{-2} \right\}.
\end{align*}
For the term $\tilde{\nsamples}_{\varepsilon}$, defined in \eqref{tilden} as
\begin{align*}
	\tilde{\nsamples}_{\varepsilon}:=\arg\min\left\{M \in \N :  \kllevel_{\varepsilon}^{\frac{1}{2}}\varepsilon^{-1}\exp\left(- M  \rho_1 H(\kllevel_{\varepsilon}) \right) \le M^{-\frac{1}{2\alpha+1}}\right\},
\end{align*}
we first note that solutions $x^{\star}$ to the equation
\begin{align*}
 \kllevel_{\varepsilon}^{\frac{1}{2}}\varepsilon^{-1}\exp\left(- M  \rho_1 H(\kllevel_{\varepsilon}) \right) = M^{-\frac{1}{2\alpha+1}}
\end{align*} 
can be written in terms of the \emph{product logarithm function} $W$. We have 
\begin{align*}
	x^{\star}= \frac{1}{(2\alpha+1)\rho_1 H(\kllevel_{\varepsilon})} W \left( -(2\alpha+1)\rho_1 H(\kllevel_{\varepsilon}) \left( \kllevel_{\varepsilon}^{-\frac{1}{2}}\varepsilon \right)^{(2\alpha+1)} \right)
\end{align*}
and due to monotonicity, we have that $\tilde{\nsamples}_{\varepsilon}=\lceil x^{\star}\rceil$. Moreover, we observe that we can make the bound on $\nsamples_{\varepsilon}$ a bit larger by considering $\alpha=0$, i.e., we can make the bound independent of $\alpha$.

To get a bound on $h_{\varepsilon}$ from \eqref{finalheps2}, we consider the upper bound (as we would like to choose $h_{\varepsilon}$ as large as possible).
Thus, we see that the term $ H^{\frac{1}{4s}}(\kllevel_{\varepsilon}) \lambda^{\frac{1}{2s}}_{\kllevel_{\varepsilon}+1} $ dominates the eigenvalue term $\lambda^{\frac{1}{s}}_{\kllevel_{\varepsilon}}$ and we obtain
\begin{align*}
h_{\varepsilon}:=\min\left\{ \lceil \varepsilon^{\frac{10}{3}} \rceil , \nsamples_{\varepsilon}^{-1},h_{0} \right\}.
\end{align*}

The case of  higher dimensional Brownian motion could be done in a similar way. Note to this end that $G$ is completely independent of the dimension whereas $H$ decays exponentially with the dimension. The latter exponentially influences the choice for $h_{\epsilon}$  in \eqref{finalheps2} and for $M_{\epsilon}$ in \eqref{finalMeps2} for growing dimensions.

 Finally, let us remark shortly on the cost involved in our resulting approximation algorithm with respect to $\kllevel, h,\nsamples$. The numerical cost do not depend on $\kllevel$ directly as $\kllevel$ is implicitly contained in $h_{\varepsilon}$  as $\kllevel \le \dimVh \sim h^{-d}$. Furthermore it involves $\mathcal{O}(\nsamples h^{-2d})$ operations to assemble the $\nsamples$  associated discrete eigensystem \eqref{genalizedeigenvalue3-estimator}  and it involves $\mathcal{O}(h^{-3d})$ operations to solve the eigensystem in a naive direct way. Moreover note that there exist faster approximative solution techniques, like multipole or algebraic multigrid methods, with substantially reduced cost to tackle the task of the eigensystem solution.

A final analysis of the balancing of the cost versus accuracy and thus the corresponding $\epsilon$-complexity of our approach is left to the reader.

\section{Concluding remarks}\label{sec:conclusion}
We discussed the problem to recover approximately the covariance of a Gaussian random field from a finite number of discretized observations. To this end, we coupled recent sharp estimates on the eigenvalue decay of continuous covariance operators with optimal statistical (tapering) estimators for covariance matrices. The combination of these techniques additionally involved a finite element discretization which made all operators finite rank  and made our approach feasible.

We provided new and sharp error estimates in  
expectation for the reconstruction of the full covariance operator taking the number of samples, the finite element discretization and the truncation of the Karhunen Lo\`{e}ve expansion of the covariance operator and thus its regularity into account.

Note that in contrast to most analytical approaches, we do not need to have access to the continuous covariance function. We instead reconstruct the covariance operator from a finite number of discrete samples. This makes our approach feasible when only measurements of the random coefficients are available. Such a situation is indeed often encountered in many practical problems in uncertainty quantification and machine learning.


\section*{Acknowledgments}
The author MG was partially supported by the Hausdorff Center for Mathematics in Bonn and the Sonderforschungsbereich 1060 {\em The Mathematics of Emergent Effects} funded by the Deutsche Forschungsgemeinschaft.
GL acknowledges the support from the Royal Society via the Newton International fellowship. Parts of a preliminary version of this work were obtained during her visit to IPAM in the Long program: Computational Issues in Oil Field Applications.
CR thanks the Institute for Numerical Simulation for its hospitality.

\bibliographystyle{abbrv}
\bibliography{reference}

\end{document}